\documentclass[11pt,english]{smfart}
\usepackage{amssymb,amsxtra, amsmath, amsfonts}
\usepackage{smfthm}
\theoremstyle{plain}
\usepackage[french, english]{babel}
\makeatletter
\def\amsbb{\use@mathgroup \M@U \symAMSb}
\makeatother

\usepackage[all]{xy}
\SelectTips{cm}{}
\usepackage{tikz}
\usepackage{tikz-cd}
\usepackage[a4paper, marginparwidth=60pt, margin=1.0 in]{geometry}
\usepackage{mathrsfs}
\usepackage{stmaryrd}
\usepackage{xspace}
\usepackage{csquotes}

\usepackage{enumitem}
\usepackage{tensor}
\usepackage{mathtools}
\usepackage{longtable}
\usepackage{mathtools}
\usepackage{hyperref} %[backref=page]
\numberwithin{equation}{section}
\usepackage{placeins}

\usepackage{pgfplots}
\usetikzlibrary{intersections}

\usepackage{pgfplots}
\usetikzlibrary{shapes}
\usetikzlibrary{arrows}
\usetikzlibrary{arrows.meta}
\usetikzlibrary{math}

\tikzset{vertex/.style={draw,circle,scale=1}}
\tikzset{arrow/.style={->,>={Latex[width=1.5mm,length=2mm]}}}
\tikzset{axearrow/.style={arrows={-triangle 90}}}
\tikzset{backarrow/.style={<-,>=latex}}
\tikzset{dubbel/.style={double distance=2pt}}
\tikzset{label/.style={midway,tiny}}
\tikzset{tiny/.style={scale=0.8}}
\tikzset{emph/.style={blue,thick}}

\tikzset{coo/.style={red,tiny}}

\tikzset{colora/.style={green}}
\tikzset{colorb/.style={orange}}

\tikzset{
  treenode/.style = {draw,circle,align=center, inner sep=0pt, text centered,font=\sffamily}, 
  arn/.style = {treenode,text width=1em}, 
}

\usepackage{esvect}

\pgfplotsset{every axis/.append style={
                    axis x line=middle, 
                    axis y line=middle, 
                    axis line style={->}, 
                    xlabel={$x$}, 
                    ylabel={$y$}, 
                    label style={font=\scriptsize}, 
                    tick label style={font=\tiny}, 
                    unit vector ratio*=1 1 1, 
                    xlabel style={at={(ticklabel* cs:1)},anchor=west}, 
                    ylabel style={at={(ticklabel* cs:1)},anchor=south}
                 }}
\pgfplotsset{compat=1.14}

\setitemize[0]{leftmargin=*,itemsep=\the\smallskipamount}
\setenumerate[0]{leftmargin=*,itemsep=\the\smallskipamount}

\let\shortmapsto\mapsto
\renewcommand{\mapsto}{%
   \ifbool{@display}{\longmapsto}{\shortmapsto}%
   }

\newcommand{\sG}{\ensuremath{\mathscr{G}}\xspace}
\newcommand{\sH}{\ensuremath{\mathscr{H}}\xspace}

\newcommand{\sL}{\ensuremath{\mathscr{L}}\xspace}

\newcommand{\fkm}{\ensuremath{\mathfrak{m}}\xspace}

\newcommand{\fkS}{\ensuremath{\mathfrak{S}}\xspace}

\newcommand{\fkX}{\ensuremath{\mathfrak{X}}\xspace}

\newcommand{\BA}{\ensuremath{\amsbb{A}}\xspace}

\newcommand{\BC}{\ensuremath{\amsbb{C}}\xspace}
\newcommand{\BD}{\ensuremath{\amsbb{D}}\xspace}
\newcommand{\BE}{\ensuremath{\amsbb{E}}\xspace}
\newcommand{\BF}{\ensuremath{\amsbb{F}}\xspace}
\newcommand{\BG}{\ensuremath{\amsbb{G}}\xspace}
\newcommand{\BH}{\ensuremath{\amsbb{H}}\xspace}

\newcommand{\BN}{\ensuremath{\amsbb{N}}\xspace}

\newcommand{\BP}{\ensuremath{\amsbb{P}}\xspace}
\newcommand{\BQ}{\ensuremath{\amsbb{Q}}\xspace}
\newcommand{\BR}{\ensuremath{\amsbb{R}}\xspace}

\newcommand{\BZ}{\ensuremath{\amsbb{Z}}\xspace}

\newcommand{\bF}{\ensuremath{\mathbf{F}}\xspace}

\newcommand{\bM}{\ensuremath{\mathbf{M}}\xspace}

\newcommand{\bU}{\ensuremath{\mathbf{U}}\xspace}
\newcommand{\bV}{\ensuremath{\mathbf{V}}\xspace}

\newcommand{\bX}{\ensuremath{\mathbf{X}}\xspace}

\newcommand{\CA}{\ensuremath{\mathcal{A}}\xspace}
\newcommand{\CB}{\ensuremath{\mathcal{B}}\xspace}
\newcommand{\CC}{\ensuremath{\mathcal{C}}\xspace}
\newcommand{\CD}{\ensuremath{\mathcal{D}}\xspace}
\newcommand{\CE}{\ensuremath{\mathcal{E}}\xspace}
\newcommand{\CF}{\ensuremath{\mathcal{F}}\xspace}

\newcommand{\CH}{\ensuremath{\mathcal{H}}\xspace}
\newcommand{\CI}{\ensuremath{\mathcal{I}}\xspace}

\newcommand{\CL}{\ensuremath{\mathcal{L}}\xspace}
\newcommand{\CM}{\ensuremath{\mathcal{M}}\xspace}
\newcommand{\CN}{\ensuremath{\mathcal{N}}\xspace}
\newcommand{\CO}{\ensuremath{\mathcal{O}}\xspace}
\newcommand{\CP}{\ensuremath{\mathcal{P}}\xspace}

\newcommand{\CX}{\ensuremath{\mathcal{X}}\xspace}

\newcommand{\rmB}{{\ensuremath{\rm{B}}\xspace}}

\newcommand{\rmD}{{\ensuremath{\rm{D}}\xspace}}

\newcommand{\rmM}{{\ensuremath{\rm{M}}\xspace}}

\newcommand{\rmO}{{\ensuremath{\rm{O}}\xspace}}

\newcommand{\rmW}{{\ensuremath{\rm{W}}\xspace}}

\newcommand{\rmd}{{\ensuremath{\rm{d}}\xspace}}

\newcommand{\itemb}{\item[$\bullet$]}
\newcommand{\cf}{{\it cf.}\ }
\newcommand{\resp}{resp.\ }
\newcommand{\ie}{{\it i.e.}\ }
\newcommand{\wt}{\widetilde}
\newcommand{\wh}{\widehat}

\newcommand{\ov}{\overline}
\newcommand{\ud}{\underline}
\newcommand{\incl}{\hookrightarrow}

\newcommand{\id}{\mathrm{id}}
\newcommand{\wotimes}{\wh{\otimes}}

\newcommand{\cris}{{\rm cr}}

\newcommand{\dR}{{\rm dR}} 
\newcommand{\Dr}{{ \rm Dr}}

\newcommand{\rig}{{\rm rig}}

\newcommand{\brv}[1]{\breve{#1}}
\newcommand{\perf}{{\rm pf}}
\newcommand{\crit}{{\rm crit}}
\newcommand{\univ}{{\rm univ}}
\newcommand{\loc}{{\rm loc}}

\DeclareMathOperator{\im}{im}
\DeclareMathOperator{\Card}{Card}

\newcommand{\Aut}{{\rm Aut}}

\DeclareMathOperator{\Ker}{Ker}

\newcommand{\coker}{\mathrm{coker}}

\newcommand{\Gr}{{\rm Gr}\xspace}

\newcommand{\Fil}{{\rm Fil}\xspace}

\DeclareMathOperator{\Sym}{Sym}
\newcommand{\End}{{\rm End}}
\newcommand{\CEnd}{{\mathcal End}}

\newcommand{\Dim}{{\rm Dim}}

\DeclareMathOperator{\heig}{height}
\DeclareMathOperator{\rank}{rank}
\newcommand{\leng}{{\rm length}}
\newcommand{\lcm}{{\rm lcm}}

\newcommand{\rmOD}{\rmO_{\rmD}}

\newcommand{\GL}{{\rm GL}}

\newcommand{\PGL}{{\rm PGL}}

\newcommand{\Lie}{{\rm Lie}\xspace}

\newcommand{\spe}{{\rm sp}}

\newcommand{\Gal}{{\rm Gal}}

\newcommand{\Pic}{{\rm Pic\xspace}}

\DeclareMathOperator{\Spec}{Spec}
\DeclareMathOperator{\Spf}{Spf}
\DeclareMathOperator{\Sp}{Sp}
\DeclareMathOperator{\Proj}{Proj}

\newcommand{\Div}{{\rm Div\xspace}}
\newcommand{\Fl}{\mathscr{F}\ell}
\newcommand{\Bl}{{\rm Bl}}
\newcommand{\Nilp}{ {\rm Nilp}}

\newcommand{\Vect}{\mathrm{Vect}}

\newcommand{\bVectd}[1]{\bVectd{#1}}
\newcommand{\bVectdf}[1]{\bVectdf{#1}}

\newcommand{\Perf}{{\rm Perf}}
\newcommand{\Alg}{{\rm Alg}}

\newcommand{\NCRIS}{{\rm NCRIS}}

\newcommand{\Hom}{{\rm Hom}}
\newcommand{\Map}{{\rm Map}}

\newcommand{\CHom}{{\mathcal Hom}}
\newcommand{\CExt}{{\mathcal Ext}}

\newcommand{\Ext}{{\rm Ext}}

\newcommand{\Bcris}{\rmB_{\cris}}
\newcommand{\Bdr}{\rmB_{\dR}}

\newcommand{\Bcrisp}{\rmB_{\cris}^+}
\newcommand{\Bdrp}{\rmB_{\dR}^+}

\newcommand{\Fpbar}{\bar{\BF}_p}
\newcommand{\Qp}{\BQ_p}

\newcommand{\Zp}{\BZ_p}

\newcommand{\brvO}{\breve{\rmO}}

\newcommand{\boeta}{{\boldsymbol{\eta}\xspace}}

\newcommand{\BrTi}{\mathcal{BT}}
\newcommand{\BaCo}{\mathcal{BC}}

\newcommand\restr[2]{{% we make the whole thing an ordinary symbol
  \left.\kern-\nulldelimiterspace % automatically resize the bar with \right
  #1 % the function
  \vphantom{\big|} % pretend it's a little taller at normal size
  \right|_{#2} % this is the delimiter
  }}

\title{On Drinfeld's representability theorem}

\author{Arnaud Vanhaecke}
\address{{\tiny Morningside Center of Mathematics, No. 55, Zhongguancun East Road, Beijing, 100190, China.} }
\urladdr{https://arnaudvanhaeckemaths.github.io/}
\email{arnaud@amss.ac.cn}
\thanks{This work was supported by the Morningside Center of Mathematics, Chinese Academy of Sciences.}

\begin{document}
\maketitle
\date{\today}
\begin{abstract}
In the seventies, V. G. Drinfeld proved that a moduli problem of deformations by quasi-isogenies of certain $p$-divisible groups with extra actions is representable by an explicit semi-stable model of the $p$-adic symmetric space. This theorem, known as \emph{Drinfeld's representability theorem}, has been one of the cornerstones of geometric aspects in $p$-adic Hodge theory. The purpose of these notes is twofold. On the one hand we give a new and more transparent proof of Drinfeld's representability theorem; on the other hand, we give a detailed presentation of Drinfeld's moduli space and the formal model of the $p$-adic symmetric space.
\end{abstract}
\section{Introduction}
Let $p$ be a prime number and $E$ be a finite extension of $\Qp$. Let $d\geqslant 2$ be an integer, the \emph{$p$-adic symmetric space} is the rigid analytic space over $E$ defined by 
$$
\BH_E^d\coloneqq \BP^{d-1}_E\setminus \bigcup_{H\in \sH_E}H,
$$
where $\BP^{d-1}_E$ is the projective space of dimension $d-1$ over $E$, viewed as a rigid analytic space, and the union runs over all $E$-rational hyperplanes. If $d=2$, $\BH_E^2=\BP^1_E\setminus \BP^1(E)$, which is a $p$-adic analog of the \emph{Poincaré half-plane} $\BP^1_{\BC}\setminus \BP^1(\BR)$. One of the crucial properties of the Poincaré half-plane is that it parametrizes certain abelian groups, in the sense that, to any $\tau \in \BP^1(\BC)\setminus \BP^1(\BR)$, we can associate a complex elliptic curve 
$$
E_{\tau}\coloneqq \BC/\Lambda_{\tau},\quad\text{ where }\quad  \Lambda_{\tau}\coloneqq \BZ\oplus \tau \BZ\subset \BC\text{ is a lattice.}
$$
In the 1970s, Drinfeld \cite{drin} proved a remarkable theorem that characterizes $\BH_E^d$ in a similar way. More precisely, he proves that a semi-stable formal model\footnote{The construction of this model is attributed to Deligne, who personally communicated it to Drinfeld (\cf \cite{mus}).} of $\BH_E^d$ represents a formal moduli problem of (deformations by quasi-isogenies of) certain formal abelian groups called \emph{special formal $\rmOD$-modules} (or \emph{SFD} for short). In an analogous way to the complex case, the main point of this construction is to associate to (a deformation by quasi-isogeny of) an SFD a certain \og lattice \fg\footnote{More precisely, it is a chain of lattices in $E^d$, each embedding in a compatible way in $\BC_p=\wh{\bar{E}}$.}. Unfortunately, unlike the complex case, the construction of these \og lattices \fg is not as straightforward and illusions of the analogy quickly fade away.

Drinfeld's original paper \cite{drin} was fairly short, and the details of his proof were spelled out by Boutot and Carayol \cite{boca} in the case $d=2$. For $d>2$, although the details of Drinfeld's proof have surely been fully worked out, there is not, to our knowledge, any exhaustive account\footnote{Nevertheless, see \cite{razi} and \cite{far} for some details.} in the literature.

The consequences and influence of Drinfeld's theorem are immense. In particular, it allows one to define \emph{Drinfeld's tower}, a system of $\GL_d(E)$-equivariant étale coverings $\CM^n_{\Dr}\rightarrow \BH^d_E$ indexed by $n\in \BN$.

In order to give a sense on the importance of this construction, but without pretending to be exhaustive, let us cite a few applications:
\medskip
\begin{itemize}
\itemb For $d=2$, the $p$-adic half plane $p$-adically uniformizes quaternionic Shimura curves; this theorem is known as the \emph{Cherednik--Drinfeld theorem}. Based on Cherednik's work \cite{che}, a sketch of the proof appears in Drinfeld's original paper \cite{drin}, and the details are also spelled out by Boutot and Carayol \cite{boca}. This result in itself has been a crucial input in various number theory developments, ranging from $p$-adic $L$-functions \cite{beda} to Kudla's program \cite{kuraya}.
\itemb For $\ell\neq p$ a prime number, the $\ell$-adic cohomology of Drinfeld's tower encodes the classical local Langlands correspondence and the Jacquet--Langlands correspondence for $\GL_d(E)$ \cite{boy}, \cite{har}, \cite{dat}. These statements were conjectured by Carayol \cite{car1}, who proved the $d=2$ case \cite{car2}. Carayol dubbed this program \emph{non-abelian Lubin--Tate theory} but it is now also called \emph{Carayol's program}\footnote{This name is preferable since the term \og non-abelian Lubin--Tate theory \fg can be confused with a similar enterprise for the \emph{Lubin-Tate tower}.}.
\itemb For $d=2$ and $E=\Qp$, the $p$-adic étale cohomology of Drinfeld's tower encodes the \emph{$p$-adic Langlands correspondence} \cite{codoni} for $\GL_2(\Qp)$. This opened the path for \emph{the $p$-adic Carayol program}, which initiated the geometrization program of the $p$-adic Langlands program. 

\itemb In chromatic homotopy theory, the computation of the $p$-adic proétale cohomology of Drinfeld's symmetric space \cite{codoniste} has led, through Drinfeld's theorem, to the proof of the \emph{rational chromatic splitting conjecture} \cite{bssw}.
\end{itemize}
\medskip

Let us just mention that many of these applications also involve the isomorphism between Drinfeld's tower and the Lubin--Tate tower. The original proof of this two--tower isomorphism, sketched by Faltings \cite{fal} and detailed by Fargues \cite{far}, made heavy use of Drinfeld's theorem. Later, this theorem was widely generalized by Scholze--Weinstein \cite{scwe}.

On the other hand, Drinfeld's theorem opened the way to studying moduli spaces (of quasi-isogenies) of $p$-divisible groups with extra structures. In their book \cite{razi}, Rapoport and Zink show that these moduli problems are representable by formal schemes, known as \emph{Rapoport--Zink spaces}. 
Later, these spaces were generalized to \emph{local Shimura varieties} \cite{ravi} allowing further, conjectural and effective, generalizations of both the $\ell$-adic and $p$-adic Langlands program. Needless to say, Drinfeld's theorem is one of the most influential statements in $p$-adic aspects of arithmetic geometry in the last century. 

In this paper, using modern techniques in $p$-adic geometry, we give a new proof of Drinfeld's theorem. Although experts might grasp the main ideas fairly quickly, we took the opportunity to spell out some details and make the presentation more accessible to non-experts. This justifies the length of the paper, which we hope may serve as an introduction to some classical topics, as well as a reference on Drinfeld's theorem. 

\subsection{Drinfeld's representability theorem}
Let us state Drinfeld's theorem. 

First, we quickly recall the definition of Drinfeld's moduli space (\cf Section \ref{sec:drinmoduli}). Let $\brv{E}$ be the $p$-adic completion of the maximal unramified extension of $E$, and let $\brvO\subset \brv E$ be its ring of integers. Let $\rmD$ be the central division algebra with invariant $1/d$ over $E$, where $d\geqslant 2$, and let $\rmOD\subset \rmD$ be its maximal order. A \emph{special formal $\rmOD$-module} is a formal $p$-divisible group endowed with an action of $\rmOD$ and whose Lie algebra satisfies a certain \emph{special} condition (\cf Definition \ref{def:sfd}). Let $\bX$ be a special formal $\rmOD$-module of height $d^2$ over $\Fpbar$\footnote{There is a unique one up to isogeny (\cf corollary \ref{cor:autosfd}) and this choice does not affect what follows. Nevertheless, we make the most convenient choice for us (\cf \S \ref{subsubsec:defimoduli}).}. \emph{Drinfeld's moduli problem} is the functor on $\Nilp_{\brvO}$, the category of $\brvO$-algebras such that $p$ is nilpotent, defined by 
\begin{equation}\label{eq:intmodu}
\brv{\CM}_{\Dr}\colon R\in \Nilp_{\brvO} \mapsto \left\{(X,\rho) \mid \begin{array}{c} \text{ $X$ a special formal $\rmOD$-module over $R$, } \\
					\rho \colon \bX \otimes_{ \bar \BF_p } R/p \dashrightarrow X \otimes _R R/p \\
					\text{an $\rmOD$-linear quasi-isogeny }\end{array} \right \}.
\end{equation}
Since $\GL_d(E)$ agrees with $\rmOD$-linear self quasi-isogenies on $\bX$ (\cf Corollary \ref{cor:autosfd}), it naturally acts on this functor for $g\in \GL_d(E)$ by $(X,\rho)\mapsto (X,\rho\circ g^{-1})$.

On the other hand, let $\BH_{\rmO}^d$ be Deligne's semi-stable formal model of $\BH_{E}^d$ (\cf Section \ref{sec:drinsym}). It is naturally endowed with an action of $\GL_d(E)$, which induces the inverse of the standard action\footnote{Meaning that it comes from the (right) action of $g^{-1}$ on $E^d$ for $g\in \GL_d(E)$ which induces a left action on $\BP^{d-1}_E$ viewed as a moduli space of rank $1$ quotients.} on $\BP_E^{d-1}$. Moreover, we consider $\BH_{\brvO}^d\coloneqq \BH_{\rmO}^d\wotimes_{\rmO}\brvO$, its base change to $\brvO$ and
\begin{equation}\label{eq:concompz}
\brv{\BH}_{\brv \rmO}^{d}\coloneqq \bigsqcup_{\BZ}\BH_{\brvO}^d= \BH_{\brvO}^d\times \BZ
\end{equation}
such that $\GL_d(E)$ acts via translation by the valuation of the determinant on the second component. The theorem is the following:
\medskip
\begin{theo}\label{theo:main}
There is a $\GL_d(E)$-equivariant isomorphism of formal schemes over $\brvO$
$$
\brv{\CM}_{\Dr}\xrightarrow{\sim} \brv{\BH}_{\brv \rmO}^{d}.
$$

\end{theo}
\medskip
\medskip
\begin{rema}
Bartling \cite{bar} gives a different proof of this theorem for $d=2$ using display theory as presented by Zink \cite{zin}. His strategy is to write down the universal display locally on $\BH^2_{\brvO}$ and then glue it. Even though this approach seems promising, and has the advantage of being very explicit, it appears to be much harder to write down the gluing data for $d>2$.
\end{rema}
\medskip
Let us mention that this isomorphism is compatible with \emph{Weil descent datum} (\cf \S \ref{subsec:weildes}), essentially because it is $\GL_d(E)$-equivariant.

\subsection{Sketch of the proof}
In the remainder of this introduction, we will forget about the components and the Weil descent. Let $\CM_{\Dr}\subset \brv{\CM}_{\Dr}$ be the subfunctor of pairs $(X,\rho)$ such that $\heig \rho =0$; we sketch the proof that $\CM_{\Dr}\xrightarrow{\sim}\BH_{\brvO}^d$.
\subsubsection{Lourenço--Scholze's theorem}
The main idea is to use a theorem due to\footnote{Lourenço attributes this theorem to Scholze \cite[Corollary 4]{lou}; Scholze attributes this theorem to Lourenço \cite[Theorem 18.4.2]{berk}. } Lourenço and Scholze (\cf Theorem \ref{thm:lou}). This theorem states that the following functor is fully faithful:
$$
\fkX\rightsquigarrow (\fkX_{\brv E}^{\rig}, \fkX_k^{\perf}, \spe_{\fkX}),
$$
associating to $\fkX$, a normal formal scheme flat over $\brvO$, the triple consisting of 
\begin{itemize}
\itemb  $\fkX_{\brv E}^{\rig}$ its generic fiber as a rigid space,
\itemb $\fkX_k^{\perf}$ the perfection of its special fiber,
\itemb $\spe_{\fkX}\colon \rvert \fkX_{\brv E}\lvert\rightarrow \lvert \fkX_k \rvert$ its specialization morphism on classical points.
\end{itemize}
In other words, the space $\fkX$ is entirely determined by the triple $(\fkX_{\brv E}, \fkX_k^{\perf}, \spe_{\fkX})$. Using a theorem of Rapoport and Zink (\cf Proposition \ref{prop:razirep}), we get that $\CM_{\Dr}$ is representable by a formal scheme and by the theory of local models we check that this formal scheme is flat and normal over $\brvO$ (\cf Proposition \ref{prop:isnormflat}). Hence, Lourenço--Scholze's theorem breaks Theorem \ref{theo:main} into three fairly independent statements concerning the rigid generic fiber, the perfection of the special fiber and the specialization morphism of Drinfeld's moduli space. The next three paragraphs explain the proof of these three statements.

\subsubsection{The generic fiber}
Let $\rmM_{\brv E}$ be the generic fiber of $\CM_{\Dr}$. In Section \ref{sec:genfib} we prove that there is a $\GL_d(E)$-equivariant isomorphism $\pi\colon \rmM_{\brv E}\xrightarrow{\sim} \BH_{\brv E}^d$. We recall the definition of the Grothendieck--Messing period morphism (\cf \S \ref{subsubsec:defpermap})
$$
\pi\colon\rmM_{\brv E}\rightarrow \BP^{d-1}_{\brv E},
$$
which associates to a special formal $\rmOD$-module its Hodge filtration. This morphism between rigid analytic spaces is $\GL_d(E)$-equivariant and by Grothendieck--Messing theory (\cf Corollary \ref{cor:gmod}) it is étale. We need to show that the period map is an open embedding whose image is $\BH_{\brv E}^d$.  To show this, we will use a lemma by Fargues (\cf Lemma \ref{lem:farlem}) which says the following: since $\pi$ is étale, $\pi$ is an open immersion, if and only if, for any algebraically closed complete field extension $C$ of $\brv E$, $\pi$ induces an injection on $C$-points. Hence, it suffices to prove that $\pi$ induces a bijection on $C$-points:
$$
\pi\colon \rmM_{\brv E}(C)\cong \BH_{\brv E}^d(C).
$$ 
 
So, let $C$ be an algebraically closed complete field extension of $E$. The main input\footnote{Note that on classical points we could use Fontaine's theory of crystalline representations but we need some more involved $p$-adic Hodge theory to treat $C$-points.} is Fargues--Scholze--Weinstein's classification of $p$-divisible groups over $\rmO_C$ (\cf Theorem \ref{thm:farclas}). This classification of $p$-divisible groups in terms of linear algebra is reminiscent of Riemann's theorem for complex abelian varieties (\cf \cite[Section 4.]{schsur}); it states that there is an equivalence of categories
\begin{equation}\label{eq:intclass}
G\rightsquigarrow \{(T,W)\mid W\subset T\otimes_{\Zp} C(-1)\},
\end{equation}
where:
\begin{itemize}
\itemb $G$ is a $p$-divisible group over $\rmO_C$,
\itemb $T=T_p(G)$ is a finite free $\Zp$-module given by the Tate module of $G$,
\itemb $W\subset T\otimes_{\Zp} C(-1)$ is a $C$-vector subspace given by the Hodge--Tate map $W=\Lie(G)\rightarrow T\otimes_{\Zp} C(-1)$.
\end{itemize}

Scholze--Weinstein translate this equivalence into modifications of vector bundles on the Fargues--Fontaine curve (\cf Theorem \ref{thm:scwe}). Let $\CX$ be the Fargues--Fontaine curve\footnote{For us, it is enough to use the schematic one.} and let $\CE_{\Dr}$ be the vector bundle on $\CX$ associated to the special formal $\rmOD$-module $\bX$ over $\Fpbar$ (or more precisely, to its isocrystal twisted by $1$). Then $(X,\rho)\in \rmM_{\brv E}(C)$ defines a modification of $\CE_{\Dr}$:
\begin{equation}\label{eq:intmodif}
0\rightarrow V_{\varpi}(X)\otimes_{E} \CO_{\CX}\rightarrow \CE_{\Dr}\rightarrow \iota_{\infty,*}\Lie(X)\otimes_{\rmO_C}C\rightarrow 0.
\end{equation}
This modification of $\CE_{\Dr}$ is $\rmOD$-linear, trivial, minuscule and of degree $d$. Moreover, using the classification result \eqref{eq:intclass}, we can show that any such modification defines a (height zero) isogeny class of elements in $\rmM_{\brv E}(C)$ (\cf Proposition \ref{prop:modeqpdiv}).

Using Beauville--Laszlo uniformization (\cf Proposition \ref{prop:bpair}) we get a natural bijection (\cf Lemma \ref{lem:odmodif}):
\begin{equation}\label{eq:intbela}
\left \{ 
\text{$\rmOD$-linear minuscule degree $d$ modifications of $\CE_{\Dr}$}
\right \}
\cong {} 
\BP^{d-1}(C)
\end{equation}
This allows us to describe the period map on $C$-points as the composition 
$$
\pi \colon \rmM_{\brv E}(C)\xrightarrow{\eqref{eq:intmodif}+\eqref{eq:intbela}}\BP^{d-1}(C).
$$
Note that if $X$ is an SFD over $\rmO_C$, then its Tate module $T_{p}(X)$ is a free $\rmOD$-module of rank $1$, so its (height zero) quasi-isogeny class consists of a unique element: by the classification \ref{eq:intclass}, this shows that $\pi(C)$ is injective. 

Hence, the last step is to show that $\pi(\rmM_{\brv E}(C))$ is $\BH_{\brv E}^d(C)$. Conversely, this amounts to show that $x\in \BP^{d-1}_{\brv E}(C)$ is contained in a rational hyperplane if and only if the associated modification by \eqref{eq:intbela} is not trivial. For this, we use a trick\footnote{We recall that $\BH_{\brv E}^d\subset \BP^{d-1}_{\brv E}$ is the \emph{weakly-admissible locus} and that we in fact show that the admissible locus coincides with the weakly-admissible one. Hence Chen--Fargues--Shen's theorem \cite{cfs} could be used directly but, since it is very straightforward in our case, we decided to spell it out for convenience.} from Chen--Fargues--Shen \cite{cfs} which reduces this claim to the following (\cf proof of Proposition \ref{prop:cimadr}): If $\CF\incl \CE_{\Dr}$ is a minuscule $\rmOD$-linear modification which is not trivial, then there is a direct summand $\CF'\subset \CF$ such that $\CF'\incl \CE_{\Dr}$ is also a direct summand. To show this, we explicitly compute all the minuscule degree~$d$ modifications of $\CE_{\Dr}$ (\cf Proposition \ref{prop:compodmod}).

These modifications are computed using the \emph{two--tower principle}, which is a translation, in terms of modifications of vector bundles, of the two--tower isomorphism mentioned at the beginning of the introduction. This gives:
\begin{equation}\label{eq:intwotow}
\left \{ \begin{array}{cc}
\text{$\rmOD$-linear minuscule modifications of $\CE_{\Dr}$}\\
\text{of degree $d$}
\end{array}
\right \}
\cong {} 
\left \{ \begin{array}{cc}
\text{Rank $d$ vector bundles $\CF$ with}\\
\text{a trivial modification $\CO^d\incl \CF$}\\
\text{of degree $1$}
\end{array}
\right \}.
\end{equation}
To compute the right-hand side, we give two different proofs, each having its own advantage. 
\begin{itemize}
\itemb The first proof (\cf Proposition \ref{prop:ltmod1}) is very direct but uses\footnote{It seems that this is the only place where this theorem is actually used. Although, this is not transparent in the current presentation.} the classification of vector bundles over the Fargues--Fontaine curve. Note that the original proof of this classification actually depends on Drinfeld's theorem. But Fargues--Fontaine (and later Fargues--Scholze \cite[Theorem II 2.14]{fasc}) give a different proof of this classification, as we will recall (\cf Theorem \ref{thm:clasvb}).
\itemb The second proof (\cf Proposition \ref{prop:ltmod2}) is indirect and can be found in the appendix of \cite{psch}. It does not use the classification of vector bundles on the Fargues--Fontaine curve but uses the representability of Rapoport--Zink spaces and the classification of $p$-divisible groups over a complete discrete valuation ring.
\end{itemize}

\subsubsection{The perfection of the special fiber}
We write $k$ for the residue field of $\brv E$, which is isomorphic to $\bar \BF_p$. We consider $\rmM_{k}^{\perf}$, the perfection (\cf \ref{eq:defperf}) of the special fiber of $\CM_{\Dr}$ and, $(\BH_k^d)^{\perf}$, the perfection of the special fiber of $\BH_{\brvO}^d$. In Section~\ref{sec:perfspe} we prove that there is a $\GL_d(E)$-equivariant isomorphism 
$$
\rmM_{k}^{\perf}\xrightarrow{\sim} (\BH_k^d)^{\perf}.
$$

Let $\BrTi_{\! d}[0]$ be the set of $0$-simplices of the Bruhat--Tits building (\cf \S \ref{subsubsec:btbuild}), which are homothety classes $\bar \eta\coloneqq E^{\times}\cdot \eta$ of lattices $\eta\subset E^d$. We first recall (\cf Subsection \ref{subsec:speh}) that the irreducible components of $\BH_k^d$ can be identified with compactified Deligne--Lusztig varieties $X^{\bar \eta}_k$ (\cf Proposition \ref{prop:cpctdl}), giving a decomposition of $\BH_k^d$ into irreducible components (\cf Corollary \ref{cor:iredcomph}) in the form:
$$
\BH_k^d=\bigcup_{\bar \eta\in \BrTi_{\! d}[0]}X^{\bar \eta}_k. 
$$
The strategy is then to get a similar decomposition of $\rmM_{k}^{\perf}$, to construct an isomorphism between the irreducible components and to glue the isomorphisms together. This strategy is very similar to the one used by Genestier \cite[Chapitre II]{gen} in the function field case; it is remarkable that this works for the perfected special fiber in the $p$-adic case.

As in Drinfeld's original paper, we use classical Cartier theory (\cf Remark \ref{rema:cadeth}). Let $R$ be a perfect $k$-algebra and $W_{\rmO}(R)$ be the ring of ramified Witt vectors over $R$ with its Frobenius automorphism $\sigma\colon W_{\rmO}(R)\rightarrow W_{\rmO}(R)$. Cartier theory (\cf \S \ref{subsubsec:odcarth}) associates to  $X$, an SFD over $R$, a $\BZ/d$-graded $W_{\rmO}(R)$-module:
$$
M\coloneqq M_{\rmO}(X)=\bigoplus_{i\in \BZ/d}M_i,
$$
endowed with 
\begin{itemize}
\itemb $F\colon M\rightarrow M$, a $\sigma$-linear endomorphism of degree $-1$,
\itemb $V\colon M\rightarrow M$, a $\sigma^{-1}$-linear endomorphism  of degree $1$,
\itemb $\Pi \colon M\rightarrow M$, a linear endomorphism  of degree $1$,
\end{itemize}
such that these commute with each other, satisfy $FV=\varpi=\Pi^d$, and satisfy the \emph{special condition}:
$$
\forall i \in \BZ/d,\quad \rank_R M_i/V M_{i-1}=1.
$$
Conversely, such a \emph{special Cartier $R$-module\footnote{In the analogy with complex elliptic curves, the Cartier module can be interpreted (through crystalline cohomology) as a first cohomology group with extra structures. The main difference is that the Cartier module is too big to be a lattice.}} determines an SFD over $R$ (\cf Proposition \ref{prop:classod}).  

The key observation, due to Drinfeld, is that one can construct lattices inside $M$ using \emph{critical indices}: we say that $i\in \BZ/d$ is a critical index for $X$ if the map induced by $\Pi$
$$
\Pi_i\colon M_i/VM_{i-1}\rightarrow M_{i+1}/VM_i,
$$
is zero. If $i\in \BZ/d$ is a critical index for $X$, then we can define $U\coloneqq V^{-1}\Pi\colon M_i\rightarrow M_i$ which turns out to be a \og slope $0$ Frobenius\fg on $M_i$; hence, we set 
\begin{equation}
\eta^i_X\coloneqq M_i^{\Pi=V}=\{x\in M_i \mid U x = x\},
\end{equation}
which defines a Zariski locally constant sheaf of $\rmO$-module of rank $d$ (\cf Lemma \ref{lem:inva}). We choose $\bX$ such that all its indices are critical, and hence we can fix isomorphisms (\cf \S \ref{subsubsec:refob}) $\eta^i_{\bX}\otimes_{\rmO}E\cong E^d$. Now, if we consider $(X,\rho)\in \rmM_k(R)$ such that $i\in \BZ$ is critical, we get, through the quasi-isogeny $\rho$, a Zariski locally constant sheaf of genuine $\rmO$-lattices
\begin{equation}\label{eq:intinclat}
\eta^i_X\incl \eta^i_{\bX}\otimes_{\rmO}E\cong E^d.
\end{equation}
The main issue is that critical indices don't always exist over arbitrary objects in $\Nilp_{\brvO}$ and Drinfeld's original argument, through what he calls \emph{the modified Dieudonné module}, constructs a way to simulate the existence of critical indices. The construction and study of the modified Dieudonné module is the trickiest part of Drinfeld's proof and is detailed for $d=2$ in \cite{boca}. Over perfect rings, the situation is much simpler: first, because Witt vectors are $p$-torsion-free, and hence Cartier theory, behaves much better over perfect rings; second, because, over perfect rings, critical indices always exist on irreducible schemes (\cf Theorem \ref{thm:critexist}). This last fact relies on the existence of critical indices over perfect fields of characteristic $p$ and the fact that universal homeomorphisms between perfect schemes are isomorphisms.

The local existence of critical indices allows us to write:
$$
\rmM_k^{\perf}=\bigcup_{\bar \eta\in \BrTi_{\! d}[0]}\rmM_k^{\bar \eta}
$$
where $\rmM_k^{\bar \eta}\subset \rmM_k^{\perf}$ is the \emph{lattice subscheme} defined by (\cf Definition \ref{def:lattsubsc}):
$$
\rmM_k^{\bar \eta}(R)\coloneqq \{ (X,\rho)\in \rmM_k^{\perf}(R) \mid \text{$i\in \BZ/d$ is critical and $E^{\times}\cdot \eta_X^i=\bar \eta$ through \eqref{eq:intinclat}}\}.
$$

Using Cartier theory, we then construct explicit isomorphisms $f_{\bar \eta}\colon\rmM_k^{\bar \eta}\xrightarrow{\sim}(X_k^{\bar \eta})^{\perf}$ in a compatible way so that we can glue these together and finally obtain:
$$
f\colon\rmM_k^{\perf}\cong (\BH_k^d)^{\perf}.
$$
Without more details, let us mention that this isomorphisms turns out to be $\GL_d(E)$-equivariant.
\subsubsection{The specialization morphism}
The last step (\cf Subsection \ref{subsec:specomp}) is to show that the previous isomorphisms, $\pi$ on the generic fiber and $f$ on the perfected special fiber, commute with the specialization map. More precisely, we need to show that the following diagram:
\begin{center}
\begin{tikzcd}
\rvert \rmM_{\brv E}\rvert  \ar[r,"\pi"]\ar[d,swap,"\spe_{\Dr}"]&\rvert \BH_{\brv E}^d\rvert \ar[d,"\spe_{\BH}" ] \\
\lvert \rmM_{k}\rvert \ar[r,"f"]&\rvert \BH_{k}^d\rvert
\end{tikzcd}
\end{center}
commutes (\cf theorem \ref{thm:specomp}). This can be checked on classical points, and the main point is to carefully describe the maps $\pi$ and $f$ on classical points (\cf Propositions \ref{prop:calsptgen} and \ref{prop:calsptspe} respectively). Although this is quite straightforward, let us mention two of the main points we need to take care of. 

The first point concerns two ways to relate $\BH_{\brv E}^d$ with lattices; we have to show that they are equivalent. 
\begin{itemize}
\itemb The first way is through the construction of the formal model $\BH_{\brvO}^d$ which can be defined as a simplicial scheme over the Bruhat--Tits building $\BrTi_{\! d}$. This gives a map of topological spaces $\lvert \BH_{\brv E}^d\rvert \rightarrow \lvert \BrTi_{\! d} \rvert$. 
\itemb The second way is through the isomorphism $\BH_{\brv E}^d\cong \BP^{d-1}_{\brv E}\setminus \bigcup_{H\in \sH_E}H$: a point $x\in \lvert \BH_{\brv E}^d\rvert$ gives an embedding $x\colon  E^d \incl K$ where $K$ is some valued field extension of $\brv E$: the norm on $K$ induces a norm on $E^d$ through $x$. Now, the topological realization $\lvert \BrTi_{\! d}\rvert$ is naturally isometric to the space of norms on $E^d$ (\cf Proposition \ref{prop:normsp}), which gives a map $\lvert \BH_{\brv E}^d\rvert \rightarrow \lvert \BrTi_{\! d} \rvert$.
\end{itemize}
These two maps coincide through the natural map $\lvert \BrTi_{\! d} \rvert\rightarrow \BrTi_{\! d}$ (\cf Corollary \ref{cor:matchlatnorm}). This allows us to relate the period map $\pi$ with lattices in a compatible way through the integrality of the Hodge filtration.

Let us spell out what this \og integrality of the Hodge filtration \fg means, which is the second point of the proof. Let $K/\brv E$ be a valued finite field extension and $\rmO_K\subset K$ its ring of integers. Let $X$ be an SFD over $\rmO_K$, $X_k=X\otimes_{\rmO_K}k$ its base change to the residue field, and let $M=M_{\rmO}(X_k)$ be the Cartier module of $X_k$. Using Grothendieck--Messing theory, we get a map induced by the Hodge filtration
$$
h_X\colon M\otimes_{\brv O}\rmO_K\rightarrow \Lie(X).
$$
In the moduli problem, the quasi-isogeny gives an isomorphism $E^d\cong M_0\otimes_{\Zp}\Qp$ and composing it with $h_X$ gives an element of $\BH_{\brv E}^d(K)$:
$$
\left ( E^d\rightarrow \Lie(X)_0\otimes_{\Zp}\Qp\right ) \in \BH_{\brv E}^d(K).
$$
This is the description of the period map on classical points. Using the previous compatibility between lattices and norms, we see that the chain of lattices in $M$ defined by \eqref{eq:intinclat} coincides with the lattices defined by the unit ball $\Lie(X)\subset \Lie(X)\bigl [ \tfrac 1 p \bigr ]$. Finally, the compatibility of the specialization morphisms boils down to the commutative diagram
\begin{center}
\begin{tikzcd}
M\otimes_{\brvO}\rmO_K\ar[r,"h_X"]\ar[dr,swap, "\mod V"]& \Lie(X) \ar[d,"\mod \fkm_K"]\\
 & \Lie(X_k)
\end{tikzcd}
\end{center}
where $\fkm_K\subset \rmO_K$ is the maximal ideal. This last step happens to be clear from Grothendieck--Messing theory.
\medskip
\begin{rema}\label{rema:cadeth}
Drinfeld's original proof uses Cartier theory in the classical sense, as described in \cite{laz}. For the perfected special fiber, we also use Cartier theory, but for the generic fiber we use Messing's \emph{Dieudonné crystal} \cite{mes}. To prove that the specialization morphisms are compatible, we recall that over a point, Cartier theory and Messing's theory reduce to classical Dieudonné theory. To avoid this plurality of theories we could simply use the theory of displays as presented by Zink \cite{zin} and its generalization to $\rmO$-displays \cite{acz}. Display theory would also be convenient for our exposition but to avoid its heavier formalism, we decided to stick with the more classical theories as they are used in our main references. 
\end{rema}
\medskip

\tableofcontents

\subsection{Acknowledgment}
This project started a long time ago as a collaboration with Sebastian Bartling. I want to thank him very heartily for sharing the idea of using the Lourenço--Scholze theorem to prove Drinfeld's theorem. I am very grateful to Michael Rapoport who introduced me to Drinfeld spaces and suggested that there should be a new proof of this theorem. I thank Pol Vanhaecke for the wonderful TikZ drawings and his patience in their realization. I also thank Laurent Fargues and Alex Youcis for several discussions on this project. Finally, I thank the Morningside Center of Mathematics, Chinese Academy of Sciences, where this work as been realized, for its excellent working environnement.

\subsection{Notations and conventions}
Let $E$ be a finite extension of $\Qp$, and set $\rmO\subset E$ its ring of integers, $k_E$ its residue field, and $q\coloneqq \Card(k_E)$ the cardinal of $k_E$. We fix  $\varpi\in \rmO$ a uniformizer generating the maximal ideal of $\rmO$. We write $v_E\colon E \rightarrow \BZ\cup\{\infty \}$ for the valuation sending $\varpi$ to $1$ and define the norm $\lvert \cdot \rvert_E=q^{-v_E(\cdot)}$. 

Let $\brv E$ be the completion for the $p$-adic norm of the unramified closure of $E$ and let $\brvO\subset \brv E$ be its ring of integers. We write $k$ for the residue field of $\brv E$ which is an algebraic closure of $k_E$. 

Let $C=\wh{\bar E}$ be the completion for the $p$-adic norm of an algebraic closure of $E$ and let $\rmO_C\subset C$ be the ring of integers.

We write $\kappa$ for an arbitrary algebraically closed field of characteristic $p$. For a ring $R$, we define $\Alg_R$ to be the category of $R$-algebras and $\Nilp_R$ the full subcategory of $R$-algebras such that $p$ is Zariski-locally nilpotent. We write $\Perf_k$ for the category of perfect $k$-algebras. We also write $\Perf_C$ which is the category of perfectoid algebras over $C$.

If $X$ is a scheme (\resp a formal scheme, a rigid space) over some affine base $\Spec R$ (\resp $\Spf R, \Sp R$) we will just say that $X$ is a scheme (\resp a formal scheme, a rigid space) over $R$. Moreover, if $R\rightarrow R'$ is a ring morphism then we write $X\otimes_{R}R'$ or $X_{R'}$ for the corresponding base change. That is, in the context of formal schemes and rigid spaces, the tensor product is implicitly completed. 

For convenience, if $n,m\in \BZ$ are integers, we fix the notation\footnote{This notation is standard during undergraduate studies in France.}
$$
\llbracket n,m \rrbracket \coloneqq \{ j\in \BZ\mid n\leqslant j \leqslant m\},
$$
for the integer interval between $n$ and $m$.

\newpage

\section{Drinfeld moduli space}\label{sec:drinmoduli}
\subsection{Special formal $\rmOD$-modules}

Let $d\geqslant 2$ be an integer. Let $\rmD$ be the central division algebra over $F$ of invariant $1/d$. To fix some additional notations, let us briefly recall what this means.

Let $E_d$ be the unramified extension of $E$ of degree $d$ and let $\rmO_d\subset E_d$ be its ring of integers. The group $\Gal(E_d/E)$ is cyclic of degree $d$ and we fix a generator $\sigma$ such that $\sigma(x)\equiv x^q\ \mod \varpi$. We define $\rmD$ as the non-commutative algebra generated by $E_d$ and an element $\Pi\in \rmD$ satisfying the relations 
$$
\Pi^d=\varpi,\quad \Pi a=\sigma(a) \Pi,\ \forall a\in E_d.
$$
The \emph{maximal order} is the subring $\rmOD\subset \rmD$ generated by $\rmO_d$ and $\Pi$.
\subsubsection{Definition}
Let $R\in \Alg_{\brvO}$. If $X$ is a $p$-divisible group over $R$, $\End_R(X)$ is naturally equipped with the structure of a $\Zp$-algebra : the addition is the pointwise sum, the multiplication is the composition and we easily show that the natural map $\BZ\rightarrow \End_R(X)$ extends to $\BZ_p\rightarrow \End_R(X)$. 
\medskip
\begin{defi}\label{def:pidiv}
\
\begin{itemize}
\itemb A \emph{$\varpi$-divisible $\rmO$-module} on $R$ is a pair $X=(X,\iota)$ where $X$ is a $p$-divisible group and $\iota$ is a non-zero $\Zp$-algebra morphism
$$
\iota \colon \rmO\rightarrow \End_R(X),
$$
which is \emph{strict}, meaning that the induced map $d\iota\colon \rmO\rightarrow \End_R(\Lie(X))$ coincides with the map $\rmO\rightarrow \brvO \rightarrow R\rightarrow\End_R(\Lie(X))$ induced by the $R$-module structure. If, moreover, $X$ is a formal $p$-divisible group, we will simply call $X$ a \emph{formal $\rmO$-module}.
\itemb A \emph{$\Pi$-divisible $\rmOD$-module} over $R$ is a pair $(X,\iota)$ where $X$ is a $p$-divisible group and $\iota$ is a non-zero $\Zp$-linear morphism
$$
\iota\colon \rmOD\rightarrow \End_R(X),
$$
such that the restriction of $\iota$ to $\rmO$ defines a $\varpi$-divisible $\rmO$-module on $R$. If, moreover, $X$ is formal as a $p$-divisible group, we will just call $X$ a \emph{formal $\rmOD$-module}.
\end{itemize}
\end{defi}
\medskip
Note that if $(X,\iota)$ is a $\Pi$-divisible $\rmOD$-module over $R$, $\iota$ is injective since $\rmOD$ is simple. Of course, one can consider $\rmOD$-modules without the strictness assumption (\cf \cite{razion}), but we will not.

We define $R_d\coloneqq R\otimes_{\rmO}\rmO_d$ and get
\begin{equation}\label{eq:aird}
R_d=\prod_{\psi\colon E_d\incl \brv E}R_{\psi}, \quad R_{\psi}\coloneqq R\otimes_{\rmO_d,\psi}\rmO_d,
\end{equation}
where the product is over all the $E$-linear embeddings $\psi\colon E_d\incl \brv{E}$.

We fix once and for all an embedding $\psi_0\colon E_d\incl \brv{E}$. Recall that we fixed a generator $\sigma$ of $\Gal(E_d/E)$ which defines an isomorphism $\Gal(E_d/E)\cong \BZ/d$. Set $\psi_i\coloneqq \psi_0\circ \sigma^i$ for all $i\in \BZ/d$, then $\{\psi\colon E_d\incl \brv{E}\}=\{\psi_i\}_{i\in \BZ/d}$ and the decomposition \eqref{eq:aird} takes the form
\begin{equation}\label{eq;airdi}
R_d\cong \prod_{i\in \BZ/d}R_i,\quad R_i\coloneqq R_{\psi_i}.
\end{equation}

Let $X$ be a $\Pi$-divisible $\rmOD$-module over $R$. Since the $\rmO$-action is strict, $\Lie(X)$ is naturally a $R_d$-module. Hence, we get 
\begin{equation}\label{eq;lidi}
\Lie(X)=\bigoplus_{\psi\colon E_d\incl \brv E}\Lie(X)_{\psi}=\bigoplus_{i\in \BZ/d}\Lie(X)_{i},\quad \Lie(X)_i\coloneqq\Lie(X)_{\psi_i}\coloneqq \Lie(X)\otimes_{R_d}R_{\psi_i}.
\end{equation}
The uniformizer $\Pi \in \rmOD$ induces an endomorphism $\Pi\colon \Lie(X)\rightarrow \Lie(X)$ of degree $+1$ for this grading. The restriction of $\Pi$ to the components of \eqref{eq;lidi} induce $d$ morphisms

$$
\forall i\in \BZ/d,\quad \Pi_i\colon \Lie(X)_i\rightarrow \Lie(X)_{i+1},
$$
and since the action is strict and $\Pi^d=\varpi$, we get $\Pi_{d-1}\circ\dots \circ \Pi_2 \circ \Pi_1\circ \Pi_0=\varpi$. 
\medskip
\begin{defi}\label{def:sfd}
A $\Pi$-divisible $\rmOD$-module $X$ over $R$ is a \emph{special formal $\rmOD$-module} (abbreviated a \emph{SFD-module}) if $X$ is a formal $\rmOD$-module and $\Lie(X)$ is an invertible $R_d$-module, \ie for every $i\in \BZ/d$, $\Lie(X)_i$ is an invertible $R$-module. Moreover, we say that $i\in \BZ/d$ is a \emph{critical index} if $\Pi_i=0$.
\end{defi}
\medskip

\medskip
\begin{lemm}\label{lem:havcrit}
Let $X$ be a special formal $\rmOD$-module over a perfect field $R=k$ of characteristic $p$. Then $X$ admits a critical index.
\end{lemm}
\medskip
\begin{proof}
Since $\varpi\cdot k=0$, we get $\Pi_{d-1}\circ\dots \circ \Pi_2 \circ \Pi_1=0$. But since $\Pi_i$ is a linear map between one-dimensional vector spaces over $k$, it is either $0$ or an isomorphism. Hence, there exists $i\in \BZ/d$ such that $\Pi_i=0$.
\end{proof}

\subsubsection{Cartier theory for formal $\varpi$-divisible $\rmO$-modules over a perfect ring}

Here we recall some basics of Cartier theory over perfect rings. Cartier theory works over arbitrary nilpotent rings, but we will only use it over perfect rings.

We begin with some notations concerning ramified Witt vectors; we follow \cite[2.1]{fafo}. Let $\rmW_{\rmO}\colon \Alg_{\rmO}\rightarrow \Alg_{\rmO}$ be the functor of \emph{ramified Witt vectors}. It comes equipped with a Frobenius endomorphism\footnote{The notation is consistent with the previous $\sigma$ so confusion is harmless.} $\sigma \colon \rmW_{\rmO}\rightarrow \rmW_{\rmO}$ and a Verschiebung endomorphism $\tau \colon \rmW_{\rmO}\rightarrow \rmW_{\rmO}$ such that $\sigma\circ \tau = \varpi$. The endomorphism $\tau$ depends on the choice of the uniformizer $\varpi$ but not $\sigma$. There is a natural transformation $[\cdot]\colon \id\rightarrow\rmW_{\rmO}$ called the \emph{multiplicative lift}. If $R\in \Alg_{\rmO}$, then $\rmW_{\rmO}(R)$ is $\tau$-complete and every element $x\in \rmW_{\rmO}(R)$ has a unique expression
\begin{equation}\label{eq:expel}
x=\sum_{n\geqslant 0}\tau^n[a_{n}],
\end{equation}
where $a_{n}\in R$. 

If $R$ is a perfect $\BF_q$-algebra, we also have the relation $\tau\circ \sigma = \varpi$ and get $\sigma(x)=x=\sum_{n\geqslant 0}\tau^n[a_{n}^q]$ which allows us to simplify \eqref{eq:expel}:
$$
x=\sum_{n\geqslant 0}[b_n]\varpi^n,
$$
where $b_n\coloneqq a_n^{q^{-n}}$. Moreover, $\rmW_{\rmO}(R)$ is then $\varpi$-complete and $\varpi$-torsion-free; it is, up to isomorphism, the unique $\varpi$-complete, $\varpi$-torsion-free lift of $R$.

Let $E'/E$ be a finite extension and $\rmO'\subset E'$ its ring of integers. Let $f\in \BN$ be its unramified degree over $E$, $q'\coloneqq q^f$ the cardinal of its residue field and $\varpi'\in \rmO'$ a uniformizer. We write $\sigma'$ and $\tau'$ respectively the Frobenius and the Verschiebung associated to $\varpi'$ on $W_{\rmO'}$. There is a unique natural transformation (\cf \cite[Lemme 1.2.3]{fafo})
$$
u\coloneqq W_{\rmO}\rightarrow W_{\rmO'},
$$
such that 
$$
\sigma'\circ u=u\circ \sigma^f,\quad \frac{\varpi}{\varpi'}\cdot\tau'\circ {\sigma'}^{f-1} \circ u=u\circ \tau.
$$
Moreover, if $R\in \Perf_{\BF_{q'}}$, let $\rmO'_0\subset \rmO'$ be the ring of integers of the maximal unramified extension of $E$ inside $E'$, then the natural morphism induced by $u$ gives an isomorphism:
$$
W_{\rmO}(R)\otimes_{\rmO_0'}\rmO'\rightarrow W_{\rmO'}(R).
$$
Let $\wh{W}_{\rmO}\subset W_{\rmO}$ be the subfunctor of ramified Witt vectors with finitely many non-zero entries, meaning that for $R\in \Alg_{\rmO}$ and $x\in W_{\rmO}(R)$, we have $x\in \wh{W}_{\rmO}(R)$ if and only if, when written in the form \eqref{eq:expel}, $\{ a_n\neq 0;n\in \BN\}$ is a finite set.

We define the Cartier ring functor as $\BE_{\rmO}\coloneqq \Hom(\wh{W}_{\rmO},\wh{W}_{\rmO})$. Explicitly, for $R\in \Alg_{\rmO}$, the \emph{Cartier ring} $\BE_{\rmO}(R)$ is the $V$-completion of the non-commutative ring $\rmW_{\rmO}(R)\langle F,V\rangle$ with relations
$$
FV=\varpi,\quad Fa=\sigma(a)F,\ Va=\tau(a)V\ \forall a \in \rmW_{\rmO}(R).
$$
Every $x\in\BE_{\rmO}(R)$ can be written uniquely has 
$$
x=\sum_{i,j\geqslant 0}V^i[x_{i,j}]F^j,
$$
where $x_{i,j}\in R$ is a family of elements such that for every $i\in \BN$, $\{x_{i,j}\neq0; j\in \BN\}$ is a finite set.

Cartier theory translates the data of a formal $\rmO$-module into a module over Cartier's ring. We define the corresponding objects.

\medskip
\begin{defi}
Let $R\in \Perf_k$. A \emph{reduced Cartier $R$-module} is a left $\BE_{\rmO}(R)$-module $M$ such that
\begin{itemize}
\itemb $V\colon M\rightarrow M$ is injective,
\itemb $M$ is complete for the $V$-adic topology, \ie the map $M\rightarrow \varprojlim_n M/V^nM$ is an isomorphism,
\itemb $M/VM$ is a locally free $R$-module of finite rank.
\end{itemize}
We call 
$$
\heig_{\rmO} M\coloneqq \rank_RM/\varpi M,\quad \dim M\coloneqq \rank_RM/V M
$$
respectively the \emph{height} of $M$ (which is well defined \cf \cite[Corollary 5.43]{zinc}\footnote{The English translation of Zink's book \cite{zinc} on Cartier theory can be found here: \url{https://imag.umontpellier.fr/~romagny/articles/zink.pdf}}) and the \emph{dimension} of $M$.
\end{defi}
\medskip
Let $M$ be a reduced Cartier $R$-module such that $M/VM$ is free, let $m\geqslant 1$ be its rank and let $\{\gamma_1,\dots,\gamma_m\}\in M$ be a family such that its image in $M/VM$ forms a basis. We then call $\{\gamma_1,\dots, \gamma_m\}$ a \emph{$V$-basis of $M$} and we get that any $v\in M$ can be uniquely written as
$$
v=\sum_{i=1}^m\sum_{n\geqslant 0}V^n[v_{n,i}]\gamma_i.
$$
Hence, the action of $F$ is determined by its action on the $V$-basis, which gives a presentation of $M$ as a $\BE_{\rmO}(R)$-module.

We now recall the main theorem of Cartier theory. For a formal $\rmO$-module $G$ over $R$, we define 
$$
M_{\rmO}(G)\coloneqq \Hom(\wh{W}_{\rmO},G),
$$
where $\Hom$ is taken as natural transformations between functors over perfect algebras. Then $M_{\rmO}(G)$ is naturally a $W_{\rmO}(k)$-module, whose abelian group structure comes from the abelian group structure on $G$ and is equipped with two operations $F$ and $V$ respectively coming from $\sigma$ and $\tau$ on $\wh{W}_{\rmO}$. This naturally defines a $\BE_{\rmO}(R)$-module.
\medskip
\begin{prop}\label{prop:Ecart}
Let $R\in \Perf_k$. There exists an equivalence of categories
$$
\left\{\text{Formal $\rmO$-modules on $R$} \right\}
\xrightarrow{M_{\rmO}}
\left \{\text{Reduced Cartier $R$-modules}\right \}.
$$
Moreover, if $G$ is a formal $\rmO$-module we have
$$
\Lie(G)=M_{\rmO}(G)/VM_{\rmO}(G),\quad \dim G=\dim M_{\rmO}(G),\quad \heig_{\rmO} G=\heig_{\rmO} M_{\rmO}(G).
$$
\end{prop}
\medskip
\begin{proof}
If $\rmO=\BZ_p$ the theorem can be found in \cite{laz} or \cite[4.23]{zinc}. From there, it is adapted to formal $\rmO$-modules by Drinfeld in \cite[\S 1 Theorem]{drin}. We sketch his argument for convenience. Let $E'/E$ be a finite extension and $\rmO'\subset E'$ be its integer ring. The strategy is to show that if the theorem is true for $\rmO$ then it is true for $\rmO'$, which is enough since it is true for $\rmO=\Zp$ by the references.

If $E'/E$ is unramified of degree $f$, then decomposing along the $E$-linear embeddings $E'\incl \brv E$ as in \eqref{eq;airdi}, we get
$$
W_{\rmO}(R)\otimes_{\rmO}\rmO'=\bigoplus_{i\in \BZ/f}W_{\rmO'}(R)_i.
$$
Hence $M_{\rmO}(G)\otimes_{\rmO}\rmO'=\bigoplus_{i\in \BZ/f} M_i$ such that $F$ is of degree $-1$ and $V$ is of degree $1$. Moreover,$V\colon M_i\rightarrow M_{i+1}$ is an isomorphism if $i\neq -1$ by the strictness condition. Set $M_{\rmO'}(G)\coloneqq M_0$ and define 
$$
V'\coloneqq V^f,\quad F'=V^{1-f}F.
$$
It is clear that this defines an equivalence of categories. Finally, 
$$
M'/V'M'=M/VM=\Lie(G)
$$
which concludes the proof in that case.

Suppose that $E'/E$ is totally ramified. The hypothesis that $R$ is perfect will simplify the proof since $W_{\rmO'}(R)\cong W_{\rmO}(R)\otimes_{\rmO}\rmO'$ such that $\sigma=\sigma'$ and $\tau'=\tfrac{\varpi'}{\varpi}\tau$. Hence $M_{\rmO'}(G)\coloneqq M_{\rmO}(G)$ is naturally a $W_{\rmO'}(R)$-module and
$$
F'\coloneqq \tfrac{\varpi'}{\varpi}F,\quad V'\coloneqq V,
$$
which defines a $\BE_{\rmO'}(R)$-module. Finally, since $M_{\rmO}(G)$ is reduced, 
$$
M_{\rmO'}(G)/V'M_{\rmO'}(G)=M_{\rmO}(G)/VM_{\rmO}(G).
$$

Let us mention that another reference for formal $\rmO$-modules is \cite[26.3]{haz}. Finally, let us mention that the theory of $\rmO$-displays in \cite{acz} is developed based on the proof we sketched.
\end{proof}

We recall the compatibility between Cartier theory and Dieudonné theory over an algebraically closed field $k$ of characteristic $p$. If $G$ is a $\varpi$-divisible $\rmO$-module over an algebraically closed field $k$, we define 
$$
D_{\rmO}(G)\coloneqq \Hom(W_{\rmO},G)
$$
where $\Hom$ is taken as natural transformations between functors over perfect algebras. In the same way as $M_{\rmO}(G)$,  $D_{\rmO}(G)$ is a $W_{\rmO}(k)$-module, equipped with two operations $F$ and $V$. Moreover, if $G$ is a formal $\rmO$-module, the action of $V$ on $D_{\rmO}(G)$ is topologically nilpotent, hence $D_{\rmO}(G)$ is naturally a $\BE_{\rmO}(k)$-module. The result is the following:
\medskip
\begin{prop}\label{prop:pdivtoisoc}
Let $G$ be a formal $\rmO$-module over an algebraically closed field $k$. Then there is a natural isomorphism of $\BE_{\rmO}(k)$-modules $D_{\rmO}(G)\cong M_{\rmO}(G)$.
\end{prop}
\medskip
\begin{proof}
In the case $\rmO=\Zp$, we point to \cite[Chapitre V \S 3]{fon} and \cite{hed} for an explicit isomorphism. For formal $\rmO$-modules, it can then be deduced from the proof of Proposition \ref{prop:Ecart}.

\end{proof}
\begin{rema}\label{rem:pdivtoisoc}
Recall that if $k$ is a perfect field of characteristic $p$, an \emph{$E$-isocristal over $k$} (\cf \cite[Section 3]{kott}) is a finite-dimensional $\rmW_{\rmO}(k)\bigl [\tfrac{1}{p}\bigr]$-vector space with a $\sigma$-linear endomorphism. So if $G$ is a formal $\rmO$-module, then $(D_{\rmO}(G)\otimes_{\BZ_p}\Qp, F)$ is an $E$-isocrystal over $k$ which, by Proposition \ref{prop:Ecart}, determines the quasi-isogeny class of $G$.
\end{rema}
\subsubsection{Grothendieck--Messing theory for $p$-divisible $\rmO$-modules}
In this paragraph we recall Grothendieck--Messing theory for $p$-divisible $\rmO$-modules as presented by Fargues \cite[Annexe B Chapitre I]{far} based on the work of Messing \cite{mes}. We also point to \cite[section 3]{acz} in the context of $\rmO$-displays for formal $\rmO$-modules.

Let $\Sigma \coloneqq \Spec(\rmO)$ and $S$ be a scheme over $\Sigma$. We endow $\Sigma$ with its \emph{$\rmO$-divided power structure} coming from $\varpi\in \rmO$, defined in \cite[Annexe B.5 Chapitre 1]{far}. We denote by $\NCRIS_{\rmO}(S/\Sigma)$ (\cf \cite[Définition B.5.7]{far}) the \emph{$\rmO$-crystalline nilpotent site of $S$} with respect to the $\rmO$-divided power structure. This site comes with a structure sheaf $\CO_{S/\Sigma}$. We recall that a \emph{crystal of $\CO_{S/\Sigma}$-modules} is a sheaf $\CF$ on $\NCRIS_{\rmO}(S/\Sigma)$ such that for any map $f\colon (U',T',\delta')\rightarrow (U,T,\delta)$ between objects in this site, the natural map 
\begin{equation}\label{eq:crisprop}
f^*\CF_{(U,T,\delta)}=\CF_{(U,T,\delta)}\otimes_{\CO_{T}}\CO_{T'}\rightarrow \CF_{(U',T',\delta')}
\end{equation}
is an isomorphism.

Let $G$ be a $\varpi$-divisible $\rmO$-module on $S$. Then one can define a universel vector extension (\cf \cite[Proposition B.3.3]{far} and see \cite[Proposition 3.16]{acz} in the case of formal $\rmO$-modules)
\begin{equation}\label{eq:vectext}
0\rightarrow V_{\rmO}(G)\rightarrow E_{\rmO}(G)\rightarrow G \rightarrow 0
\end{equation}
which is an exact sequence of $fppf$ sheaves in $\rmO$-modules over $S$. This exact sequence is functorial in $G$ and in $S$ (\cf \cite[Chapter IV, Proposition 1.15]{mes}). We set 
$$
\BD_{\rmO}(G)_S\coloneqq \Lie (E_{\rmO}(G))
$$ 
the Lie algebra of this $fppf$ sheaf of $\rmO$-modules which defines a sheaf of $\CO_S$-modules. In particular, by taking the Lie algebra of the sequence \eqref{eq:vectext}, we get a map:
$$
h_G\colon\BD_{\rmO}(G)_S\rightarrow \Lie(G).
$$
Recall that this map is surjective as a morphism of $fppf$-sheaves on $S$ (\cf \cite[Chapter IV, Proposition 1.22]{mes}). The kernel of $h_G$ is usually called \emph{the Hodge filtration} and we will say that $h_G$ is induced by the Hodge filtration. The map $h_G$ is functorial in both $S$ and $G$. The main proposition is that $\BD(G)_S$ extends to a crystal on $\NCRIS_{\rmO}(S/\Sigma)$:

\medskip
\begin{prop}
The Lie algebra of $E_{\rmO}(G)$ defines a crystal of $\CO_{S/\Sigma}$-modules $\BD_{\rmO}(G)$, which we will call the \emph{Dieudonné crystal}.
\end{prop}
\medskip
\begin{proof}
The main idea is to associate, to every $(U,T,\delta)\in \NCRIS_{\rmO}(S/\Sigma)$ and to any lift $G'$ to $T$ of $\restr{G}{U}$ the sheaf 
\begin{equation}\label{eq:defcryst}
\BD_{\rmO}(G)_{(U\incl T)}=\BD_{\rmO}(G')_T
\end{equation}
and to show that this defines a crystal. See \cite[Chapter IV, \S 2]{mes} for the case $\rmO=\Zp$ and \cite[Théorème B.6.1]{far} for the extension to $\varpi$-divisible $\rmO$-modules; compare with \cite[Definition 3.24]{acz} for the case of formal $\rmO$-modules through $\rmO$-displays. 
\end{proof}
\medskip
\begin{rema}\label{rema:bbm}
If $S=\Spec R$ with $p\cdot R=0$ and $R'\rightarrow R$ satisfies $R'=\varprojlim_n R'/p^n$ and $R'/p^n$ is a finite $R$-algebra, we would like to evaluate $\BD_{\rmO}(G)$ at $R'$. We will consider
$$
\BD_{\rmO}(G)_{(R'\rightarrow R)}\coloneqq \varprojlim_{n}\BD_{\rmO}(G)_{(R'/p^n\rightarrow R)},
$$
which is consistent with \cite{dej} (see also \cite[Remark 3.2.4]{scwe}).

Following \cite{bbm} (where $\rmO=\Zp$) the (contravariant) Dieudonné crystal can be defined as 
$$
\BD_{\rmO}^{*}(G)\coloneqq \CExt^1_{\CO_{S/\Sigma}}(G,\CO_{S/\Sigma}).
$$
This definition is not the classical one (for example Messing's original book \cite{mes} is used in \cite{razi} and \cite{scwe}) but it is more convenient for different reasons. First, it can be defined on the full crystalline site and not only the nilpotent one. Second, this definition makes sense for finite flat group schemes, and behaves well with respect to truncations for finite group schemes (which is also an advantage of working contravariantly). The functor $\BD_{\rmO}^*(G)$ can be related to $\BD_{\rmO}(G)$ by duality (\ie Cartier duality using a Lubin-Tate group). 
\end{rema}
\medskip

Let $(S\incl T)\in \NCRIS_{\rmO}(S/\Sigma)$. We define the following categories:
\medskip
\begin{itemize}
\itemb $\CD$ the category of pairs $(G,\CF)$ where $G$ is a $\varpi$-divisible $\rmO$-module over $S$ and $\CF\subset \BD_{\rmO}(G)_{(S\incl T)}$ is a direct factor of $\BD_{\rmO}(G)_{(S\incl T)}$ such that $\CF\otimes_{\CO_T}\CO_S=V_{\rmO}(G)$,
\itemb $\CC$ the category of $\varpi$-divisible $\rmO$-modules over $T$.
\end{itemize}
\medskip
Recall that if $G'$ is a $\varpi$-divisible $\rmO$-module over $T$, then by definition \eqref{eq:defcryst} 
$$
\BD_{\rmO}(G')_T= \BD_{\rmO}(G'\times_T S)_{(S\incl T)}.
$$
We get the following theorem which is called the \emph{Grothendieck--Messing} theorem (\cf \cite[Chapter V, Theorem 1.6]{mes} for the case $\rmO=\Zp$ and  \cite[Théorème B.7.1]{far} for its extension to $\varpi$-divisible $\rmO$-modules):
\medskip
\begin{theo}
The functor $\CC\rightarrow \CD$ defined on objects by 
$$
G'\mapsto (G'\times_TS, V_{\rmO}(G')\subset \BD_{\rmO}(G')_{(S\incl T)}),
$$
defines an equivalence of categories.
\end{theo}
\medskip
Finally, we need to compare the Dieudonné crystal with the classical Dieudonné module. The resulting corollary is the following (\cf \cite[Proposition B.8.2]{far} and Remark \ref{rema:bbm}):
\medskip
\begin{coro}\label{cor:diecarcomp}
Let $G$ be a $\varpi$-divisible $\rmO$-module over $k_E$ and let $D_{\rmO}(G)$ be it's Dieudonné module. We have a natural isomorphism
$$
D_{\rmO}(G)\cong \Gamma(\Spec(k_E)\rightarrow \Sigma,\BD_{\rmO}(G))
$$
and the Frobenius is identified with the crystalline Frobenius map on $\BD_{\rmO}(G)$.
\end{coro}
\medskip
\subsubsection{Classification of SFD over perfect rings}\label{subsubsec:odcarth}
We extend the equivalence of Proposition \ref{prop:Ecart} to special formal $\rmOD$-modules. For this, we first introduce special Cartier modules.
\medskip
\begin{defi}
Let $R\in\Nilp_{\brvO}$. A \emph{special Cartier $R$-module} is a pair $M=(M,\Pi)$ where:
\begin{itemize}
\itemb $M$ is a $\BZ/d$-graded reduced Cartier $R$-module, \ie a Cartier $R$-module endowed with a $\BZ/d$-grading
$$
M=\bigoplus_{i\in \BZ/d}M_i,
$$
such that $\deg F =-1$, $\deg V=1$ and $\deg [a]=0$ for $a\in R$,
\itemb $\Pi$ is a $\BE_{\rmO}(R)$-linear endomorphism $\Pi\colon M\rightarrow M$ such that $\deg \Pi=1$ and $\Pi^d=\varpi$,
\end{itemize}
such that for every $i\in \BZ/d$, $M_i/VM_{i-1}$ is free of rank $1$ over $R$. 

Moreover, we call $i\in\BZ/d$ a \emph{critical index} if the map
$$
\Pi_i\colon M_i/VM_{i-1}\rightarrow M_{i+1}/VM_i,
$$
induced by $\Pi$, is zero.
\end{defi}
\medskip
\begin{lemm}\label{lem:specarthei}
 Let $R$ be a perfect $k$-algebra and let $M$ be a special Cartier $R$-module. Then $\dim M=d$ and $\heig_{\rmO} M $ is a multiple of $d^2$.
\end{lemm}
\medskip
\begin{proof}
 The claim on the dimension is clear since 
 $$
 \dim M=\rank_R M/VM=\sum_{i\in \BZ/d}\rank_RM_i/VM_{i-1}=d.
 $$
First, since $V$ is injective and commutes with $\varpi$, we get that $V\colon M_i/\varpi M_{i} \rightarrow M_{i+1}/\varpi M_{i+1}$ is injective, thus $r\coloneqq \rank_R M_i/\varpi M_i$ is independent of $i$ and $\rank_R M/\varpi M=dr$. Moreover, because $R$ is perfect, $M$ has no $\varpi$-torsion: thus $\Pi$ is injective, and since
 $$
 \rank_RM_i/\varpi M_{i}=\sum_{j=0}^{d-1} \rank_R\Pi^jM_{i-j}/\Pi^{j+1} M_{i-j-1}=\sum_{j=0}^{d-1} \rank_RM_{i-j}/\Pi M_{i-j-1}.
 $$
Thus it remains to prove that $\rank_RM_{i}/\Pi M_{i-1}$ is independent of $i$. Computing the rank of $M_{i}/\Pi VM_{i-2}$ in two different ways, we get
 $$
 \rank_RM_{i}/V M_{i-1}+\rank_RVM_{i-1}/\Pi V M_{i-2}= \rank_RM_{i}/\Pi M_{i-1}+\rank_R \Pi M_{i-1}/\Pi V M_{i-2},
 $$
 and since 
 $$
 \begin{gathered}
 \rank_RM_{i}/V M_{i-1}= \rank_R \Pi M_{i-1}/\Pi V M_{i-2}=1,\\
 \rank_RVM_{i-1}/\Pi V M_{i-2}=\rank_RM_{i-1}/\Pi M_{i-2},
 \end{gathered}
 $$
 we get $\rank_RM_{i}/\Pi M_{i-1}=\rank_RM_{i-1}/\Pi  M_{i-2}$.
\end{proof}
\medskip

Let $R\in \Nilp_{\brvO}$, and let $X$ be a special formal $\rmOD$-module over $R$, we associate to $X$ a special Cartier $R$-module. Let $M=M_{\rmO}(X)$ be its associated reduced Cartier $R$-module. By functoriality, $M$ is an $\rmO_d\otimes_{\rmO}\rmW_{\rmO}(R)$-module and by \eqref{eq;airdi}:
$$
\rmO_d\otimes_{\rmO}\rmW_{\rmO}(R)\cong \rmW_{\rmO}(R_d)\cong \prod_{i\in \BZ/d}\rmW_{\rmO}(R_i),
$$
and we obtain a $\BZ/d$-grading
$$
M=\bigoplus_{i\in \BZ/d}M_i,\quad \text{where } M_i\coloneqq M\otimes_{\rmW_{\rmO}(R_d)}\rmW_{\rmO}(R_i).
$$
For this grading $\deg F=-1$, $\deg V=1$, and for $a\in R$, $\deg [a]=0$. Moreover, for $i\in \BZ/d$,
$$
M_i/VM_{i-1}\cong \Lie(X)_i,
$$
which is locally free of rank $1$ over $R$ by the special condition. By functoriality, $\Pi\in \rmOD$ defines and endomorphism $\Pi\colon M\rightarrow M$, such that $\Pi^d=\varpi$ and $\deg \Pi=1$. From the grading and the endomorphism $\Pi$ we can reconstruct the action of $\rmOD$ on $M$. We thus obtain the following proposition: 
\medskip
\begin{prop}\label{prop:classod}
Let $R\in \Nilp_{\brvO}$. The previous construction gives an equivalence of categories
$$
\left\{\text{Special formal $\rmOD$-modules over $R$} \right\}
\xrightarrow{M_{\rmO}}
\left \{\text{Special Cartier $R$-modules}\right \}.
$$
\end{prop}
\medskip
\begin{proof}
Following the proof of Proposition \ref{prop:Ecart} in the unramified case, we see that the action of $\rmO_d$ can be recovered from the grading and the action of $\rmOD$ is then obtained from the action of $\Pi$. See \cite[Chapitre II, section 2, Théorème]{boca} for $d=2$, the proof adapts easily for arbitrary $d\geqslant 2$.
\end{proof}
For the rest of this paragraph, we suppose $R=\kappa$ is an algebraically closed field of characteristic $p$.
\medskip

\begin{lemm}\label{lem:critspe}
Let $M$ be a special Cartier $\kappa$-module of height $d^2$. Then $M$ admits a critical index and for every critical index $i\in \BZ/d$ we get $\Pi M_i= V M_i$.
\end{lemm}
\medskip
\begin{proof}
Note that the situation is the same as in Lemma \ref{lem:havcrit} and we get the existence of a critical index $i\in\BZ/d$ by the same proof, which means that $\Pi M_i\subset V M_i$. Using Lemma \ref{lem:specarthei} we see that $\dim_{\kappa} M_{i+1}/\Pi M_{i}=1$, which means that $V M_i$ and $\Pi M_i$ are two submodules of $M_{i+1}$ having the same colength. Since $\Pi M_i\subset V M_i$, we get $\Pi M_i = V M_i$.
\end{proof}
An essential construction of special Cartier modules is the following: let $M$ be a special Cartier $k$-module of height $d^2$ and let $i\in \BZ/d$ be a critical index. By Lemma \ref{lem:critspe} we know that $\Pi M_i = V M_i$ and since $V\colon M_i\incl M_{i+1}$ is injective we can define $\Phi=V^{-1}\Pi$ which is a $\sigma$-linear endomorphism of $M_i$. Hence $D_i\coloneqq(M_i\otimes_{\Zp}\Qp,\Phi)$ is a $\sigma$-isocrystal. Moreover, $M_i\subset M_i\otimes_{\Zp}\Qp$ is a $\rmW_{\rmO}(k)$-lattice of rank $d$ such that $\Phi M_i= M_i$, so $D_i$ is of slope $0$. Note that the association $M\mapsto D_i$ is functorial.

\medskip
\begin{coro}\label{cor:autosfd}
Let $\kappa$ be an algebraically closed field of characteristic $p$.
\begin{itemize}
\itemb There exists, up to isogeny, a unique special formal $\rmOD$-module of height $d^2$ over $\kappa$.
\itemb Let $X$ a special formal $\rmOD$-module of height $d^2$ over $\kappa$. Then $\End_{\rmOD}=M_d(E)$ and its group of $\rmOD$-linear quasi-isogenies is identified with $\GL_d(E)$.
\end{itemize}
\end{coro}
\medskip
\begin{proof}
By the Dieudonné-Manin classification, there exists a unique unit isocrystal over $\kappa$ of dimension $d$. This gives the uniqueness of $D_i$ and thus the uniqueness of $M=(M,\Pi)$ up to isogeny which proves the first point by Proposition \ref{prop:classod}. The second point follows from the functoriality of the composition $X\mapsto M_{\rmO}(X)=M\mapsto D_i$ and the fact that $\Aut_{\sigma}(D_i)=\GL_d(E)$.
\end{proof}
\subsubsection{Grothendieck--Messing theory of SFD}\label{subsubsec:godsfd}
As in the case of Cartier modules, the Dieudonné crystal of an SFD is endowed with an action of $\rmOD$, which can be translated into a grading and the action of a degree~$1$ endomorphism $\Pi$.
. 

Let $\brv{\Sigma}=\Spec \brvO$ and $S$ be a scheme over $\brv{\Sigma}$. Let $X$ be a special formal $\rmOD$-module and let $\BD_{\rmO}(X)$ be its Dieudonné crystal, which is endowed with an action of $\rmOD$. Note that $S\otimes_{\rmO}\rmO_d=\sqcup_{i\in \BZ/d}S_i$. Thus, for $\CO_{S/\brv \Sigma}$, the structure sheaf of the crystalline site $\NCRIS_{\rmO}(S/\brv{\Sigma})$, we get 
$$
\CO_{S/\brv \Sigma}\otimes_{\rmO}\rmO_d=\bigoplus_{i\in \BZ/d}\CO_{S_i/\brv{\Sigma}}.
$$
This defines a $\BZ/d$-grading $\BD_{\rmO}(X)=\bigoplus_{i\in \BZ/d}\BD_{\rmO}(X)_i$ as before. Moreover, $\Pi\in \rmOD$ defines a map $\Pi\colon \BD_{\rmO}(X)\rightarrow \BD_{\rmO}(X)$ which is of degree $+1$ and satisfies $\Pi^d=\varpi$. As before, the grading together with $\Pi$ determines the action of $\rmOD$.

Since the map $h_X\colon \BD_{\rmO}(X)\rightarrow \Lie(X)$ induced by the Hodge filtration, is functorial in $X$, it preserves the grading and we obtain 
\begin{equation}\label{eq:mesisgrad}
h_X=\sum_{i\in \BZ/d} h_{X,i}\colon \bigoplus_{i\in \BZ/d}\BD_{\rmO}(X)_i\rightarrow \bigoplus_{i\in \BZ/d}\Lie(X)_i.
\end{equation}
For any $i\in \BZ/d$, $h_{X,i+1}\circ \Pi_{i} = \Pi_{i}\circ h_{X,i}$. 

Let us spell out the Grothendieck--Messing theorem in this case. We define

\begin{itemize}
\itemb $\CD_{\rmOD}$ the category of pairs $(X,\CF)$ where $X$ is a special formal $\rmOD$-module over $S$ and $\CF\subset \BD_{\rmO}(X)_{(S\incl T)}$ is an $\rmOD$-stable direct factor of $\BD_{\rmO}(X)_{S\incl T}$ such that $\CF\otimes_{\CO_T}\CO_S=V_{\rmO}(X)$,
\itemb $\CC_{\rmOD}$ the category of special formal $\rmOD$-modules over $T$.
\end{itemize}

\medskip
\begin{coro}\label{cor:gmod}
 The functor $\CC_{\rmOD}\rightarrow \CD_{\rmOD}$ defined on objects by 
$$
X'\mapsto (X'\times_TS, V_{\rmO}(X')\subset \BD_{\rmO}(X')_{(S\incl T)}),
$$
defines an equivalence of categories.
\end{coro}
\medskip
We do not spell out the analogues of the previous propositions for special formal $\rmOD$-modules. Let us just note that from Proposition \ref{prop:classod} and Corollary \ref{cor:diecarcomp}, if $X$ is a special formal module $\rmOD$ module over $k$, we have an isomorphism of special Cartier modules over $k$: 
\begin{equation}\label{eq:cartvsmess}
M_{\rmO}(X)\cong D_{\rmO}(X)=\Gamma(\Spec(k)\rightarrow \Sigma,\BD_{\rmO}(X)).
\end{equation}

\subsection{Drinfeld's moduli space}

This subsection is devoted to the definition and properties of Drinfeld's moduli space. We also describe the main construction producing lattices from critical indices.
\subsubsection{Definition of the moduli problem}\label{subsubsec:defimoduli}
We define the Drinfeld moduli space. Let $\bX$ be the special formal $\rmOD$-module over $k$, the residue field of $\brvO$, whose special Cartier $k$-module is given by:
$$
\bM \coloneqq \rmOD \otimes_{\rmO}\rmW_{\rmO}(k),
$$
where the Verschiebung $\bV$ is defined by\footnote{Hence the Frobenius is defined by $\bF (1\otimes m)= \Pi^{d-1}\otimes \sigma(m)$.} $\bV (1\otimes m) = \Pi\otimes \sigma^{-1}(m)$ for all $m\in W_{\rmO}(k)$. 

We define the following functor on $\Nilp_{\brvO}$: 
\begin{equation}\label{eq:modu}
\brv{\CM}_{\Dr}\colon R\in \Nilp_{\brvO}\mapsto \left\{(X,\rho) \mid \begin{array}{l} \text{ $X$ a special formal $\rmOD$-module over $R$, } \\
					\rho \colon \bX \otimes_{ k} R/p \dashrightarrow X \otimes _R  R/p \\
					\text{ an $\rmOD$-linear quasi-isogeny}\end{array} \right \}.
\end{equation}
By Corollary \ref{cor:autosfd}, this definition is independent, up to isomorphism, of the choice of $\bX$ and the action of $\GL_d(E)$ by $\rmOD$-linear quasi-isogenies of $\bX$ induces an action on the functor $\brv{\CM}_{\Dr}$ defined, for $g\in \GL_d(E)$, by
$$
g\colon x=(X,\rho)\mapsto g\cdot x\coloneqq (X,\rho\circ g^{-1}).
$$
The functor \eqref{eq:modu} is an example of a \emph{Rapoport--Zink space}; in \cite{razi}, Rapoport and Zink show that a wide class of such moduli problems are representable by formal schemes that are locally formally of finite type over $\brvO$. Applying this theory, specifically \cite[3.25]{razi} and \cite[3.54]{razi}, to prove the representability of the moduli problem: 
\medskip
\begin{prop}\label{prop:razirep}
The functor (\ref{eq:modu}) is representable by a formal scheme, which we still write $\brv{\CM}_{\Dr}$, locally formally of finite type over $\brvO$.
\end{prop}

\medskip
\begin{rema}
\
\begin{itemize}
\itemb In Drinfeld's proof \cite{drin}, the representability of the moduli problem (\ref{eq:modu}) is a corollary of his theorem: it is obtained by showing directly that it is represented by the symmetric space.
\itemb It is also possible to use the result of Bartling--Hoof \cite{baho} to get the representability of the functor.

\end{itemize}
\end{rema}
\medskip

Since it is locally formally of finite type, we may define its rigid generic fiber\footnote{In the sense of Berthelot and Raynaud \cf \cite[5.5]{razi} or, equivalently, in the sense of Huber \cf \cite[1.1.12]{hub}}, $\brv{\rmM}_{\brv E}$ which is a rigid space over $\breve E$. 

For $h\in \BZ$, we let $\CM_{\Dr}^h\subset \brv{\CM}_{\Dr}$ be the formal subscheme consisting of those $(X,\rho)\in \brv{\CM}_{\Dr}$ such that the quasi-isogeny is of height $dh$, \ie $\heig_{\rmO}\rho=dh$. If $h=0$, we omit the superscript $0$, \ie $\CM_{\Dr}\coloneqq \CM_{\Dr}^0$. Note that for any $h\in \BN$ we have $\CM^h_{\Dr}\cong \CM_{\Dr}$ and we get a decomposition:
$$
\brv{\CM}_{\Dr}\cong \bigsqcup_{h\in \BZ} \CM^h_{\Dr}\cong \CM_{\Dr}\times \BZ.
$$
The space $\CM_{\Dr}$ is not stable under $\GL_d(E)$, but is stable under 
$$
\GL_d(E)_0\coloneqq \{g\in \GL_d(E) \mid v_{\varpi}(\det g)=0\}.
$$

\subsubsection{Construction of a $V$-basis}
Critical indices give rise to $\rmO$-lattices inside the Cartier module. The following lemma, proved in \cite[Lemma 3.65]{razi} and attributed to Drinfeld, is fundamental.
\medskip
\begin{lemm}\label{lem:inva}
Let $R\in \Nilp_{\brvO}$, $(X,\rho)\in \brv{\CM}_{\Dr}(R)$ and let $M\coloneqq M_{\rmO}(X)$ be the special Cartier $R$-module of $X$. Suppose that $X$ has a critical index $i\in \BZ/d$. Then, Zariski locally, $M_i$ admits a $\rmW_{\rmO}(R)$-basis $\{\gamma_j\}_{j\in \BZ/d}$ such that
$$
\forall j\in \BZ/d,\quad \Pi \gamma_j = V\gamma_j.
$$
In other words, $M_i^{\Pi=V}\coloneqq \bigoplus_{j\in \BZ/d}\rmO\gamma_j$ defines a locally constant Zariski-sheaf of $\rmO$-modules on $\Spec R$.
\end{lemm}
\medskip
\begin{proof}
For convenience, we recall the construction of $M_i^{\Pi=V}$. Since $i\in \BZ/d$ is a critical index, $\Pi M_i \subset VM_i$, and since $V$ is injective, we can define the operator 
$$
U\coloneqq V^{-1}\Pi \colon M_i \rightarrow M_i.
$$
Fix an integer $n\geqslant 1$. For every $R'\in \Alg_R$, let $M_{R'}$ denote the special Cartier module of the base change $X_{R'}$. Rapoport and Zink show in \cite[Lemma 3.65]{razi} that the functor on $\Alg_R$ defined by taking invariants under $U$,
$$
\eta^i_X[n]\colon R'\longmapsto \left ( (M_{R'})_i/V^{nd}(M_{R'})_i\right)^U
$$
is representable by an étale scheme over $\Spec R$.

Since $\eta^i_X[n]$ is naturally endowed with an $\rmO$-module structure, it is étale-locally on $\Spec R$ isomorphic to $\underline{(\rmO/\varpi^n\rmO)}^d_R$ by the height condition. Suppose that $R$ is a local ring, and recall that we have also been given a quasi-isogeny $\bX_k\rightarrow X_k$. Since $\eta^i_X[n]$ is constant when restricted to $\Spec  k$, it is therefore constant over $\Spec R$, i.e.
$$
\eta^i_X[n]\cong \underline{(\rmO/\varpi^n\rmO)}^d_R.
$$
By passing to the limit over $n$, we get 
$$
M_i^{\Pi=V}\coloneqq \varprojlim_n \eta^i_X[n]\cong \ud{\rmO}^d_R
$$
This completes the proof of the lemma.
\end{proof}
\medskip
\begin{rema}
\medskip
Using this lemma, one proves directly that $\brv{\CM}_{\Dr}$ is a $p$-adic formal scheme (\cf \cite[Corollary 3.63]{razi}).
\medskip

\end{rema}
\medskip
\subsubsection{$V$-basis of the reference object}\label{subsubsec:refob}
Recall that we set $\bM\coloneqq \rmOD\otimes_{\rmO}W_{\rmO}(k)$ with $\bV\coloneqq \Pi\otimes \sigma^{-1}$, the special Cartier module of the reference object $\bX$. Notice that all the indices of $\bM$ are critical, and the operator $\bU\coloneqq \bV^{-1}\Pi$ on $\bM_i$ is given by
$$
\bU(x\otimes w) =\Pi x\Pi^{-1}\otimes \sigma(w).
$$
For $i\in\BZ/d$, let $\boeta^i\coloneqq \bM_i^{\bU=1}$ given by
$$
\boeta^i=\bigoplus_{j=0}^{d-1}(\Pi^j\otimes 1)\cdot\rmO\subset \bM_i.
$$

Now, write $e_0,\dots, e_{d-1}$ for the standard basis of $E^d$. Let $g_{\Pi}$ be the $d\times d$ matrix
\begin{equation}\label{eq:pimatrix}
g_{\Pi}\coloneqq
\begin{pmatrix}
0 &          &        &        & \varpi \\
1 & 0        &        &        &        \\
  & \ddots   & \ddots &        &        \\
  &          & 1      & 0      &        \\
  &          &        & 1      & 0
\end{pmatrix},
\end{equation}
and for $i\in \BZ$, define\footnote{This construction comes from \cite[3.73]{razi} where there seems to be a typo: for their map, $\Pi$ does not define an inclusion on lattices. This can be corrected by taking the opposite sign in the definition of $\varepsilon(i,j)$.}
\begin{equation}
\begin{array}{rccc}
\iota_{i}\colon &\boeta^i&\incl& E^d\\
&\Pi^j\otimes 1&\mapsto& g_{\Pi}^{-i}e_{j}.
\end{array}
\end{equation}
This map extends to an isomorphism $\iota_{i}\colon \boeta^i\otimes_{\rmO}E\xrightarrow{\sim}E^d$. 
\medskip
\begin{rema}\label{rem:glequiv}
Note that $\iota_i=\iota_0\circ \Pi^{-i}$ and in particular $\iota_{i+d}=\varpi^{-1}\iota_i$. Since the action of $g\in \GL_d(E)$ on $\boeta^i$ is given by $\Pi^i\circ g\circ \Pi^{-i}$ and that since $\iota_0$ is $\GL_d(E)$-equivariant, it follows that $\iota_i$ is $\GL_d(E)$-equivariant.
\end{rema}
\medskip

For $i\in \BZ/d$ let $\boeta_{i}\coloneqq\iota_{i}(\boeta^i)\subset E^d$ be the image of $\iota_{i}$. In particular, $\boeta_0= \rmO^d\subset E^d$. We get a commutative diagram
\begin{center}
\begin{tikzcd}
\boeta^{-1}\ar[r,"\Pi"]\ar[d,"\iota_{-1}"]&\boeta^0\ar[r,"\Pi"]\ar[d,"\iota_{0}"]&\boeta^1\ar[r,"\Pi"]\ar[d,"\iota_{1}"]&\cdots\ar[r,"\Pi"]&\boeta^{d-1}\ar[d,"\iota_{d-1}"]\ar[r,"\Pi"]&\boeta^0\ar[d,"\iota_{d}"]\\
\varpi\boeta_{d-1}\ar[r,hook]&\boeta_{0}\ar[r,hook]&\boeta_{1}\ar[r,hook]&\cdots\ar[r,hook]&\boeta_{d-1}\ar[r,hook]&\boeta_{d}
\end{tikzcd}
\end{center}
where the bottom arrows are the natural inclusions inside $E^d$. In particular, we get the following fundamental lemma:
\medskip
\begin{lemm}\label{lem:trivtoslop}
There is a natural equivariant bijection
$$
\left \{ 
\text{$\bV$-stable, $\brv{\rmO}\otimes_{\rmO}\rmOD$-submodules of $\bM$ },
\right \}
\cong {} 
\left \{ 
\text{$\rmO$-submodules of $\boeta_0= \rmO^d$}\\
 \right \}
$$
given by $N\mapsto N_0^{\bU=1}$. In particular, subisocrystals of $\bM\otimes_{\rmO}E$ are naturally in bijection with $E$-rational subspaces of $\brv{E}^d$. Moreover, any such subisocrystal has slope $\tfrac{d-1}{d}$.
\end{lemm}
\medskip
\subsubsection{Local models} 

In this paragraph, we recall some basic facts from \cite{razi} on the theory of local models, which allow us to prove the following fact:
\medskip
\begin{prop}\label{prop:isnormflat}
The formal scheme $\brv{\CM}_{\Dr}$ is normal and flat over $\brvO$.
\end{prop}
\medskip
Rapoport and Zink \cite[3.26--3.29]{razi} construct a \emph{local model diagram}, which, in the Drinfeld case, takes the following form:
\begin{center}
\begin{tikzcd}
 & \brv{\CN}_{\Dr}\ar[dl,swap,"\pi_1"]\ar[dr,"\pi_2"]& \\
 \brv{\CM}_{\Dr} & & \wh{\rmM}^{\loc}_{\brvO}
\end{tikzcd}
\end{center}

\begin{itemize}
\itemb The formal scheme $\brv{\CN}_{\Dr}$ represents the functor
$$
R\in \Nilp_{\brvO}\mapsto \left\{(X,\rho,\gamma) \mid (X,\rho)\in \brv{\CM}_{\Dr}(R),\ \gamma \colon \BD_{\rmO}(X)_R\xrightarrow{\sim}\rmOD\otimes_{\rmO}R\right \},
$$
where $\BD_{\rmO}(X)_R$ is the evaluation of the Dieudonné crystal at $\Spec R$ and $\gamma$ is an $\rmOD$-linear isomorphism. The map $\pi_1\colon \brv{\CN}_{\Dr}\rightarrow \brv{\CM}_{\Dr}$ is the natural map defined by $\pi_1\colon (X,\rho,\gamma)\mapsto (X,\rho)$ which is a $\CP$-torsor, where $\CP$ is the smooth formal group
$$
R\in \Nilp_{\brvO}\mapsto \Aut(\rmOD\otimes_{\rmO}R).
$$
This implies that $\brv{\CN}_{\Dr}$ is represented by a scheme of finite type over $\brv{\CM}_{\Dr}$ and that $\pi_1$ is formally smooth.
\itemb The formal scheme $\wh{\rmM}^{\loc}_{\brvO}$ is the $p$-adic completion of $\rmM^{\loc}_{\brvO}$, the scheme over $\brvO$ representing the functor 
$$
\rmM^{\loc}_{\brvO}\colon R\in \Alg_{\brvO}\mapsto \{(L,f)\mid f\colon \rmOD\otimes_{\rmO} R\rightarrow L\},
$$
where,
\medskip
\begin{itemize}
\item $L$ is an $\rmOD\otimes_{\rmO}R$-module which is invertible as an $\rmO_d\otimes_{\rmO}R$-module\footnote{Recall that this means $L=\bigoplus_{i\in \BZ/d}L_i$, where the $t_i$'s are invertible $R$-modules, endowed with a degree one operator $\Pi$ such that $\Pi^d=\varpi$.}.
\item $f\colon \rmOD\otimes_{\rmO} R\rightarrow L$ is an $\rmOD\otimes_{\rmO} R$-linear surjection.
\end{itemize}
\medskip
The scheme $\rmM^{\loc}_{\brvO}$ is easily realized as a closed subscheme of some Grassmannian, which proves it is representable by a proper scheme (\cf Remark \ref{rem:eqlocmod}). The map $\pi_2$ is defined for $R\in \Nilp_{\brvO}$ by
$$
(X,\rho,\gamma)\in \brv{\CN}_{\Dr}(R)\longmapsto (h_R\circ\gamma^{-1},\Lie(X))\in \wh{\rmM}^{\loc}_{\brvO}(R)
$$
where the corresponding map is the composition $\rmOD\otimes_{\rmO}R\xrightarrow{\gamma^{-1}}\BD_{\rmO}(X)_R\xrightarrow{h_R}\Lie(X)$ and is therefore surjective. Moreover, by Grothendieck--Messing (\cf Corollary \ref{cor:gmod}), $\pi_2$ is formally smooth.
\end{itemize}
It follows that Proposition \ref{prop:isnormflat} reduces to proving that  $\rmM^{\loc}_{\brvO}$ is normal and flat over $\brvO$. We are going to show that $\rmM^{\loc}_{\brvO}$ is regular and irreducible by expressing it as an iterated blow-up. From the construction in \S \ref{subsubsec:refob}, we get the following description of $\rmM^{\loc}_{\brvO}$:
\medskip
\begin{lemm}\label{lem:explocmod}
The scheme $\rmM^{\loc}_{\brvO}$ represents the functor on $\Alg_{\brvO}$ associating to each $R\in \Alg_{\brvO}$ diagrams of the form 
\begin{equation}\label{diag:drin}
\begin{tikzcd}
  \boeta_{1-d} \arrow[r,hook]\arrow[d,"\varphi_{1}"] & \boeta_{2-d} \arrow[r,hook]\arrow[d,"\varphi_{2}"] & \cdots \arrow[r,hook] & \boeta_{0} \arrow[r,"\varpi"]\arrow[d,"\varphi_{d}"] & \boeta_{1-d} \arrow[d] \\
  L_{1} \arrow[r,"x_{1}"] & L_{2} \arrow[r,"x_{2}"] & \cdots \arrow[r,"x_{-1}"] & L_{0} \arrow[r,"x_{0}"] & L_{1},
\end{tikzcd}
\end{equation}
such that for any $i\in \BZ/d$:
\begin{itemize}
    \item[$\bullet$] $L_i$ is an invertible $R$-module,
    \item[$\bullet$]  $x_i\colon L_i\rightarrow L_{i+1}$ is a morphism of $R$-modules,
    \item[$\bullet$] $\varphi_i\colon \boeta_i\rightarrow L_i$ is a morphisms of ${\brvO}$-modules such that 
    $$
    \varphi_i\otimes_{\brvO}R\colon \boeta_i\otimes_{\brvO} R\rightarrow L_i,
    $$
    is surjective.
\end{itemize}
\end{lemm}
\medskip

\medskip
\begin{coro}
The scheme $\rmM^{\loc}_{\brvO}$ is regular, irreducible, and $\rmM^{\loc}_{\brv E}\cong \BP_{\brv E}^{d-1}$. In particular, it is normal and flat over $\brvO$.
\end{coro}
\medskip
\begin{proof}
Using Lemma \ref{lem:explocmod}, we can write $\rmM^{\loc}_{\brvO}$ as an iterated blow-up of $\BP^{d-1}_{\brvO}= \BP(\boeta_{d-1}\otimes_{\rmO}\brvO)$. Let $T_0,\dots, T_{d-1}$ be the homogeneous coordinates on $\BP^{d-1}_{\brvO}$ defined by the basis $e_0,\dots, e_{d-1}\in \boeta_0$ and for $i\in \llbracket 0,d-1\rrbracket$ let $R_i\subset \BP^{d-1}_{\brvO}$ be the closed subscheme defined by the homogeneous ideal $\langle \varpi, T_i ,\dots, T_{d-1}\rangle$. Then 
$$
\rmM^{\loc}_{\brvO}\cong \Bl_{\wt{R}_{d-1}}(\dots(\Bl_{R_0}\BP^{d-1}_{\brvO})\dots ),
$$
where $\wt R_i$ is the strict transform of $R_i$ for $i\in\llbracket 0,d-1\rrbracket$. Hence, by \cite[0BFM]{stacks}, $\rmM^{\loc}_{\brvO}$ is irreducible. By \cite[02OS]{stacks} we have $\rmM^{\loc}_{\brv E}= \rmM^{\loc}_{\brvO}\otimes_{\brvO}\brv E \cong \BP_{\brv E}^{d-1}$. Hence $\rmM^{\loc}_{\brvO}$ is dominant over $\Spec \brvO$, thus its structure sheaf is $\varpi$-torsion-free: $\brvO$ is a discrete valuation ring so  $\rmM^{\loc}_{\brvO}$ is flat over $\brvO$. Since each $R_i$ is defined by a regular sequence, $\wt{R_i}\rightarrow R_i$ is smooth and since $\BP^{d-1}_{\brvO}$ is regular, $\rmM^{\loc}_{\brvO}$ is regular by \cite[Proposition 19.4.4]{ega44}. Because a regular scheme is normal, this proves the corollary. 
\end{proof}

\medskip
\begin{rema}\label{rem:eqlocmod}
Using the presentation of $\rmM^{\loc}_{\brvO}$ as a blow-up, we can give equations for it using \cite[Chapitre 2 Lemme 1.6.1]{gen}\footnote{For this  Lemma, the reference to \cite{ful} seems incorrect; it can be deduced from \cite[Remark A.6.1]{ful}.}. Precisely, there is a closed embedding $\rmM^{\loc}_{\brvO}\incl \BP_{\brvO}^{d-1}\times \dots \times \BP_{\brvO}^{d-1}$ given by the $\varphi_i$'s of Lemma \ref{lem:explocmod} which identifies $\rmM^{\loc}_{\brvO}$ with the closed subscheme defined by the equations , where we denote by $T_0^{(i)},\dots, T_{d-1}^{(i)}$ the homogeneous coordinates on the $i$-th copy of $\BP^{d-1}_{\brvO}$:
$$
\begin{gathered}
T_{j}^{(i)}T_k^{(i+1)}=T_{j}^{(i+1)}T_{k}^{(i)},\quad 0<k<j<d\\
T_{0}^{(i+1)}T_{j}^{(i)}=\varpi T_{j}^{(i+1)}T_0^{(i)},\quad 0<j<d.
\end{gathered}
$$
Note that since we transformed $\rmOD$ into a lattice chain, the equations are slightly different from the ones in\footnote{Note that there are small typos in the equations given in \cite[3.76]{razi}, in particular it seems that \og $\pi$ \fg is on the wrong side.} \cite[3.76]{razi} where the embedding $\rmM^{\loc}_{\brvO}\incl \BP_{\brvO}^{d-1}\times \dots \times \BP_{\brvO}^{d-1}$ is given by the $f_i$'s. The equations are obtained from the previous ones by shifting the indices:
$$
\begin{gathered}
T_{j-1}^{(i)}T_k^{(i+1)}=T_{j}^{(i+1)}T_{k-1}^{(i)},\quad 0<k<j<d,\\
T_{j}^{(i+1)}T_{d-1}^{(i)}=\varpi T_{j-1}^{(i)}T_0^{(i+1)},\quad 0<j<d.
\end{gathered}
$$
These equations can also be found by considering the universal object $(L^{\univ},f^{\univ})$, where $L^{\univ}$ is the restriction of the product of the ample line bundles on $\BP_{\brvO}\coloneqq \BP_{\brvO}^{d-1}\times \dots \times \BP_{\brvO}^{d-1}$ and the map defined by $\Sym f^{\univ} : \Sym_{\brvO}(\rmOD\otimes_{\rmO}\brvO)\rightarrow H^0(\BP_{\brvO},\Sym L^{\univ})$. This map sends $\Pi^j$ along the $i$-th factor to $T_j^{(i)}$ and since $f^{\univ}$ is $\rmOD$-linear the kernel of $\Sym f^{\univ}$ is generated by:
$$
\begin{gathered}
\Pi^j\otimes \Pi^{k-1}-\Pi^k\otimes \Pi^{j-1},\, 0<k<j<d\\
\Pi^j \otimes \Pi^{d-1}-\varpi \otimes \Pi^{j-1},\, 0<j<d.
\end{gathered}
$$
which gives exactly the equations above.

Finally, note that these equations are described by $2\times2$-minors, only $(d-1)^2$ of them suffice (taking $k=j+1$ in the first family and $j=1$ in the second), which shows that $\rmM^{\loc}_{\brvO}$ is a global complete intersection (hence Cohen--Macaulay). With a bit more effort, we can even show that $\rmM_{\brvO}^{\loc}$ is a \emph{semi-stable} scheme over $\brvO$, making it a semi-stable model of the projective space.
\end{rema}
\medskip

\newpage

\section{Drinfeld symmetric space}\label{sec:drinsym}
In this section, we recall the construction of the Deligne formal model $\BH_{\rmO}^d$ of Drinfeld's symmetric space $\BH_E^d$.
\subsection{The Bruhat--Tits building for $\GL_d(E)$}
In this subsection we recall some facts on the Bruhat--Tits building of\footnote{Technically, it is not the building of $\GL_d(E)$, but disjoint copies of the Bruhat--Tits building of $\PGL_d(E)$ on which $\GL_d(E)$ acts (\cf remark \ref{rem:buildpgl}).} $\GL_d(E)$.
\subsubsection{The Bruhat--Tits building of $\GL_d(E)$}\label{subsubsec:btbuild}
The purpose of this paragraph is to recall the construction of the \emph{Bruhat--Tits building} for $\GL_{d}(E)$ as a simplicial complex $\brv{\BrTi}_{\! d}$. The set of$0$-simplices, denoted $\brv{\BrTi}_{\! d}[0]$, is the set of $\rmO$-lattice $\eta\subset E^d$, that is $\rmO$-submodules of $E^d$ of maximal rank ($=d$). For $\eta,\eta'\subset E^d$ two $\rmO$-lattices, we define their \emph{relative (logarithmic) index} by: 
$$
\log[\eta : \eta']\coloneqq\leng_{\rmO}\left ( \frac{\eta}{\eta \cap \eta'}\right )-\leng_{\rmO}\left ( \frac{\eta'}{\eta \cap \eta'}\right ).
$$
We will simply call $\log[\eta:\rmO^d]$ the \emph{index} of $\eta$. In particular, in what follows, the notation $\eta_i$ will indicate that $\eta_i$ has index $i$. As an example, let us note that if $(n_0,\dots , n_{d-1})\in \BZ^d$ and $\eta= \bigoplus_{i=0}^{d-1}\varpi^{n_i}\rmO e_i\subset E^d$, where $\{e_0,\dots, e_{d-1}\}\subset E^d$ is the standard basis, then 
\begin{equation}\label{eq:calcindex}
\log[\eta:\rmO^d]=-\sum_{i=0}^{d-1}n_i.
\end{equation}
In particular, multiplying a lattice by $\varpi$, drops the index by $d$. We define 
\begin{equation}\label{eq:brtidec}
\brv{\BrTi}_{\!d}\coloneqq \bigsqcup_{h\in \BZ} \BrTi^h_{\! d}\cong \BrTi_{\!d}^0\times \BZ,
\end{equation}
where:
\begin{itemize}
\itemb for an integer $r$, with $0 \leqslant r <d$, an $r$-simplex $\Delta\in\BrTi^h_{\! d}[r]$ is a set of cardinality $r+1$ consisting of $\rmO$-lattices in $E^d$, $\Delta = \{\eta_{i_0}, \dots, \eta_{i_r}\}$, with indices $-h \leqslant i_0 < \dots < i_r < d-h$, such that:
\begin{equation}\label{eq:condbrtg}
\varpi  \eta_{i_r} \subsetneq \eta_{i_0} \subsetneq \dots \subsetneq \eta_{i_r}, \quad \log[\eta_i : \rmO^d]=i,\ \forall i=i_0,\dots, i_r,
\end{equation}
\itemb the face maps are induced by inclusions of subsets.
\end{itemize}
Notice that the second condition in \eqref{eq:condbrtg} implies in particular that 
$$
\log[\eta_{i_k} : \eta_{i_{k-1}}]=i_k-i_{k-1}.
$$ 

The group $\GL_{d}(E)$ acts naturally on $\brv{\BrTi}_{\! d}$ in the following way: endow $E^d=\Sym^1E^d$ with the standard action of $\GL_d(E)$, then $g\in \GL_d(E)$ acts on $\brv{\BrTi}_{\! d}$ as
$$
\eta \subset E^d\mapsto g\cdot \eta \subset E^d.
$$
This defines an action of $\GL_d(E)$ on $\brv{\BrTi}_{\! d}$ since it preserves incidence relations. Notice that, by \eqref{eq:calcindex}, the action of $g\in \GL_d(E)$ on the $\BZ$-component of \eqref{eq:brtidec} is translation by $v_E(\det g)$.

Notice that, by \eqref{eq:calcindex}, $g\in \GL_d(E)$ acts through $+v_{E}(\det(g))$ on $\BZ$ in the second component of \eqref{eq:brtidec}. We recall that for any integer $r\in \llbracket 0,d-1\rrbracket$, $\GL_d(E)$ acts transitively on $\brv{\BrTi}_{\! d}[r]$ (\cf \cite[Lemma 1.1 b)]{mus}).

\medskip
\begin{rema}\label{rem:buildpgl}
 Let $E^{\times}\subset \GL_d(E)$ be the center consisting of scalar matrices and notice that the quotient $\BrTi_{\! d}\coloneqq \brv{\BrTi}_{\! d}/E^{\times}$ is the Bruhat--Tits building of $\PGL_d(E)\coloneqq \GL_d(E)/E^{\times}$. It is built on homothety classes of lattices (\ie $\bar \eta \coloneqq E^{\times}\cdot \eta$ for $\eta\subset E^d$ a lattice) and is naturally isomorphic to $\BrTi_{\! d}^h$, for any $h\in \BZ$.
 
 Indeed, fix an integer $h\in\BZ$, let $r\in \llbracket 0,d-1\rrbracket$, and let $\Delta = \{\bar{\eta}_{i_0}, \dots, \bar{\eta}_{i_r}\}\in \BrTi_{\! d}[r]$. Then for each index $i=i_0,\dots,i_r$ there exists a unique representative $\eta_{i}\in \bar{\eta}_i$ such that \eqref{eq:condbrtg} holds. This gives a section $\BrTi_{\! d}\rightarrow \BrTi_{\! d}^{h}\subset \brv{\BrTi}_{\! d}$ of the quotient.
 
We emphasize that  $\BrTi_{\! d}$ is the Bruhat--Tits building of $\PGL_d(E)$ and that $\brv \BrTi_{\! d}$ is \emph{not} the Bruhat--Tits building of $\GL_d(E)$. In fact, to get the Bruhat--Tits building of $\GL_d(E)$, we would also have to consider (non-degenerate) $d$-simplices.
\end{rema}
\medskip

We can realize $\brv{\BrTi}_{\! d}$ as a poset, giving it the structure of a topological space. This space should not be confused with its topological realization $\lvert \brv{\BrTi}_{\! d} \rvert$ which will be discussed in \S\ref{subsubsec:georeal}.

\subsubsection{$\rmO[\Pi]$-modules and lattices}
Let $R\in \Alg_{\rmO}$, we define
$$
R[\Pi]\coloneqq R[X]/\langle X^d-\varpi\rangle,
$$
with the $\BZ/d$-grading defined by $\deg b=0$ for $b\in R$ and $\deg X=1$. If $M$ is a $\BZ/d$-graded $R[\Pi]$-module we write for $i\in \BZ/d$, $\Pi_i\colon M_i\rightarrow M_{i+1}$ the restriction of $\Pi$ to $M_i$. A basic example of an $\rmO[\Pi]$-module is $\boeta^{\bullet}\coloneqq \oplus_{i\in\BZ/d}\boeta^i$, defined in \S\ref{subsubsec:refob}. In that paragraph we described how to associate to $\boeta^{\bullet}$ a $d-1$-simplex $\{\boeta_0,\dots, \boeta_{d-1}\}\in \brv{\BrTi}_{\! d}[d-1]$; this construction motivates what follows.

\medskip
\begin{defi}\label{def:lattod}
For $h\in \BZ$ let $\sL_d^h$ be the functor on $\Nilp_{\rmO}$ defined by 
$$
R\in \Nilp_{\rmO}\mapsto \{(\eta^{\bullet},r^{\bullet})\}
$$
where
\begin{itemize}
\itemb $\eta^{\bullet}$ is a $\BZ/d$-graded Zariski-constructible sheaf of flat $\rmO[\Pi]$-modules on $\Spec R$,
\itemb $r^{\bullet}\colon \ud{\boeta}^{\bullet}\otimes_{\rmO}E\xrightarrow{\sim} \eta^{\bullet}\otimes_{\rmO}E$ is a graded $E[\Pi]$-linear isomorphism, where the left-hand side is the constant sheaf associated to $\boeta^{\bullet}$,
\end{itemize}
such that for $i\in \BZ/d$ and $S\subset \Spec R$ such that $\restr{\eta^i}{S}$ is constant (\ie isomorphic to $\ud{\rmO}^d$), $r^i$ induces an isomorphism
$$
\bigwedge^d\restr{\eta^i}{S}\cong \varpi^{-h}\bigwedge^d\ud{\boeta}^i.
$$
\end{defi}
\medskip
We set $\brv{\sL}_d\coloneqq \bigsqcup_{h\in \BZ}\sL^h_d$ which is naturally endowed with an action of $\GL_d(E)$: $g\in \GL_d(E)$ acts by $g\colon (\eta^{\bullet},r^{\bullet})\mapsto (\eta^{\bullet},r^{\bullet}\circ g^{-1})$. The action of $\GL_d(E)$ on $\boeta^{\bullet}\otimes_{\rmO}E$ is described below.
\medskip
\begin{rema}\label{rem:zerograd}
In the definition, since the action of $\Pi$ is invertible on $\boeta^{\bullet}\otimes_{\rmO}E$ and $\boeta^0\otimes_{\rmO}E\cong E^d$ through the basis given by $\{\Pi^j\}_{j=0,\dots, d-1}$, the data of $r^{\bullet}$ is equivalent to the data of $r^0\colon \ud{E}^d\xrightarrow{\sim} \eta^0\otimes_{\rmO}E$. Explicitly, for $i\in \BZ/d$, $r^i$ is the only isomorphism making the following diagram of isomorphisms commute
\begin{center}
\begin{tikzcd}
\ud{E}^d\ar[r,"r^0"]\ar[d,"g_{\Pi}^i"]& \eta^0\otimes_{\rmO}E\ar[d,"\Pi^i"]\\
\ud{E}^d\ar[r,"r^i"]&\eta^i\otimes_{\rmO}E
\end{tikzcd}
\end{center}
meaning $r^i= \Pi^i\circ r^0\circ g_{\Pi}^{-i}$. Hence $g$ acts on $r^{\bullet}$ via its action by $\Pi^{-i}\circ g^{-1}\circ \Pi^{i}$ on $\boeta^0\otimes_{\rmO}E\cong E^d$ (compare with Remark \ref{rem:glequiv}).
\end{rema}
\medskip
Note that, by the second condition in Definition \ref{def:lattod}, $g\in \GL_d(E)$ sends $\sL_d^h$ to $\sL_d^{h'}$ where $h'=h+v_E(\det g)$.

This functor will appear in the definition of Drinfeld's model of the symmetric space. The purpose of this paragraph is to prove the following proposition:
\medskip
\begin{prop}\label{prop:repldbt}
The functor $\brv \sL_d$ is represented by $\brv{\BrTi}_{\! d}$ as a topological space, in the sense that for $R\in\Nilp_{\rmO}$ we have a functorial bijection
$$
\brv\sL_d(R)\cong \Map(\Spec R,  \brv{\BrTi}_{\! d} )
$$
where $\Map$ denotes the set of continuous maps between topological spaces. Moreover, this bijection is $\GL_d(E)$-equivariant and matches $\sL^h_d$ with $\BrTi_{\! d}^h$.

\end{prop}
\medskip
We insist that we do not consider the topological (Hausdorff) realization of $\brv{\BrTi}_{d}$ but see it as a poset.
\begin{proof}
We may assume that for each $h\in \BZ$, $\BrTi_{\! d}^h$ represents $\sL^h_d$, hence we can suppose $h=0$. Note that the actions of $\GL_d(E)$ on connected components coincide.

We first prove the proposition over a point by constructing two inverse maps $\Delta\mapsto \eta^{\bullet}_{\Delta}$ and $\eta^{\bullet}\mapsto \Delta_{\eta^{\bullet}}$. For $j\in \BZ/d$, let $\langle j \rangle\in \llbracket 0, d-1\rrbracket$ be its unique lift in that interval.

\begin{itemize}
\itemb Let $\Delta=\{\eta_{i_0},\dots, \eta_{i_r}\} \in \brv{\BrTi}_{\! d}^0[r]$ be an $r$-simplex, let $C_{\Delta}\coloneqq \{i_0,\dots, i_r\}\subset \BZ$ be its set of indices and $\bar C_{\Delta}\subset \BZ/d$ its set of indices mod $d$. For $j\in \BZ/d$, let $l\in \BN$ be the least integer such that $j+l\in \bar C_{\Delta}$ and let
$$
j_{+}\coloneqq \langle j+l\rangle\in C_{\Delta},
$$
the closest index right of $j$. Let  
$$
\eta_{\Delta}^j\coloneqq \iota_{\langle j \rangle }^{-1}(\eta_{j_+})\subset \boeta^j\otimes_{\rmO}E.
$$
The inclusion $\eta_{\Delta}^j\subset \boeta^j\otimes_{\rmO}E$ induces an isomorphism
$$
r^j_{\Delta}\colon \eta_{\Delta}^j\otimes_{\rmO}E\cong \boeta^j\otimes_{\rmO}E
$$
and the inclusion $\eta_{\Delta}^j\subset \eta_{\Delta}^{j+1}$ induces $\Pi_j\colon \eta_{\Delta}^j\rightarrow \eta_{\Delta}^{j+1}$; hence we set:
$$
\eta^{\bullet}_{\Delta}\coloneqq \bigoplus_{j\in \BZ/d} \eta_{\Delta}^j,\quad r^{\bullet}_{\Delta}\coloneqq \bigoplus_{j\in \BZ/d}r^j_{\Delta},
$$
where $\Pi$ acts on $\eta^{j}_{\Delta}$ by $\Pi_j$; this defines a $\rmO[\Pi]$-module. The determinant condition in \ref{def:lattod} is immediate.
\itemb Let $\eta^{\bullet}$ be a $\rmO[\Pi]$-module and $r^{\bullet}\colon \boeta^{\bullet}\otimes_{\rmO}E\xrightarrow{\sim}\eta^{\bullet}\otimes_{\rmO}E$ be a $E[\Pi]$-module isomorphism. Let $C\coloneqq \{\langle j \rangle \mid \Pi_j\colon \eta^j\rightarrow \eta^{j+1}\text{ is not an isomorphism}\}$ and order $C=\{i_0,\dots, i_r\}$ increasingly. For $i\in C$ set
$$
\eta_i\coloneqq (\iota_i\circ(r^i)^{-1})(\eta^i)\subset E^d.
$$
By construction, we obtain inclusions $\eta_{i_k}\subsetneq \eta_{i_{k+1}}$ defining a simplex 
$$
\Delta_{\eta}\coloneqq \{\eta_{i_0},\dots, \eta_{i_r}\}\in \BrTi_{\! d}^0[r].
$$
\end{itemize}
Let $X\coloneqq \Spec R$. Given a continuous map $f\colon \Spec\, R \rightarrow \brv{\BrTi}_{\! d}$ of topological spaces we get a stratification of $X$ by setting for $\Delta\in \brv{\BrTi}_{\! d}$, $X_{\Delta}\coloneqq f^{-1}(\Delta)$ giving: 
$$
X=\bigcup_{\Delta\in \brv{\BrTi}_{\! d}} X_{\Delta}.
$$
This defines a finite stratification since $X$ is compact, \ie all but finitely many $X_{\Delta}$ are empty. We get an element of $\sL_d(R)$ defined by $(\eta^{\bullet}_{\Delta},r^{\bullet}_{\Delta})$ on $X_{\Delta}$.

Now, given $(\eta,r)$ on $X$ we define a map $X\rightarrow \brv{\BrTi}_{\! d}$. Let us fix a stratification 
$$
X=\bigcup_{i\in I} Z_i
$$
such that the restriction of $\eta$ to $Z_i$ is constant. Set $\Delta_i\coloneqq \Delta(\restr{\eta}{Z_i})$. The map $f\colon X \rightarrow \brv{\BrTi}_{\! d}$ is then defined by sending $Z_i$ to $\Delta_i$.
\end{proof}

\subsubsection{Geometric realization}\label{subsubsec:georeal}

Since $\brv{\BrTi}_{\! d}$ is a simplicial complex we can construct its \emph{geometric realization} $\lvert \brv{\BrTi}_{\! d} \rvert$. For any integer $n\in \BN$, let 
$$
\Delta^{\! n}\coloneqq \{y=(y_0,\dots, y_{n})\in [0,1]^{n+1}\mid \sum_{j=0}^ny_j=1\},
$$ 
be the standard simplex of dimension $n$ and let 
$$
\lvert \brv \BrTi_{\! d} \rvert\coloneqq \left ( \bigcup_{0\leqslant n < d}\brv \BrTi_{\! d}[n]\times \Delta^n\right ) /\sim
$$
where we quotient by the usual simplicial relations. This realization does not make the metric structure on $\lvert \BrTi_{\! d}\rvert$ transparent and we introduce \emph{apartments} to clarify this.

If $\CB=\{e_0,\dots, e_{d-1}\}\subset E^d$ is a basis then we can associate a maximal split torus $T\subset \GL_d(E)$ of matrices that are diagonal in the basis $\CB$. Moreover, any maximal split torus in $\GL_d(E)$ is of this form. Let $\ud{n}=(n_1,\dots, n_d)\in \BZ^d$ and define the lattice
\begin{equation}\label{eq:torlatdef}
\eta_{\CB}(\ud{n})\coloneqq \bigoplus_{i=0}^{d-1} \varpi^{n_i}\rmO e_i.
\end{equation}
We associate an apartment to $\CB$, only depending on the torus $T$, which is a simplicial subset $\brv \CA_T\subset \brv \BrTi_{\! d}$ defined for $r\in\llbracket 0, d-1 \rrbracket$ by
$$
\brv \CA_T[r]\coloneqq \{\Delta\in \brv \BrTi_{\! d} \mid \Delta\subset \{\eta_{\CB}(\ud{n});\ \ud{n}\in \BZ^d\}\}.
$$
In other words, it is the subcomplex of lattices which are \og diagonal \fg for $T$. Since the action of $\GL_d(E)$ is transitive on the simplices of fixed dimension of $\brv\BrTi_{\! d}$, we deduce:
$$
\brv\BrTi_{\! d}=\bigcup_{T\subset \GL_d(E)}\brv\CA_T.
$$
\medskip
\begin{rema}
Strictly speaking, $\brv \CA_T$ is not an apartment but a disjoint union of apartments. Indeed, we get $\brv\CA_T=\bigsqcup_{h\in \BZ}\CA_T^h$ where $\CA_T^h\subset \BrTi_{\! d}^h$ which is in bijection with $\CA_T$, its image through the projection map $\brv \BrTi_{\! d}\rightarrow \BrTi_{\! d}$ of Remark \ref{rem:buildpgl}. In other words $\CA_T$ is the apartment associated to $T/E^{\times}\subset \PGL_d(E)$ inside the Bruhat--Tits building for $\PGL_d(E)$.
\end{rema}
\medskip

\medskip
\begin{lemm}\label{lem:descap}
Let $T\subset \GL_d(E)$ be a maximal split torus. The inclusion $\brv{\CA}_T\subset \brv{\BrTi}_{\! d}$ induces a homeomorphism
$$
\lvert \brv \CA_T\rvert\cong \BR^{d-1}\times \BZ 
$$
\end{lemm}
\medskip
\medskip
\begin{proof}

Let $\brv{V}_{\BR}=\bigsqcup_{h\in \BZ}V^{h}_{\BR}\subset \BR^d$ where for $h\in \BZ$, $V^{h}_{\BR}$ is the affine hyperplane defined by
$$
V^{h}_{\BR}\coloneqq \{(x_0+h,\dots, x_{d-1}+h)\in \BR^d\mid \sum_{i=0}^{d-1}x_i=0\},
$$
and $\brv{V}_{\BZ}\coloneqq \BZ^d\subset \brv{V}_{\BR}$. We also define $V_{\BZ}\coloneqq \brv{V}_{\BZ}/\BZ$ and $V_{\BR}\coloneqq\brv{V}_{\BR}/\BZ$ which weidentify respectively  with $V^0_{\BZ}$ and $V^0_{\BR}$ through the projection.

We first recall how to define the simplicial structure on $\brv{V}_{\BR}\cong \BR^{d-1}\times \BZ$. Let $i,j\in \BN$ such that $0\leqslant i\neq j\leqslant d-1$ and $n\in \BN$, we define the hyperplane section
$$
H^{i,j}_n\coloneqq  \{(x_0,\dots, x_{d-1})\in \BR^d \mid x_i-x_j=n\}\cap \brv{V}_{\BR}\subset \brv{V}_{\BR}
$$
This hyperplane arrangement defines a simplicial structure on $\brv{V}_{\BR}$ hence on $V_{\BR}$ by projection. We can describe explicitly a simplicial set $A$ such that $\lvert A \rvert = V_{\BR}$. Namely, the $0$-simplices of $A$ are defined by 
$$
A[0]\coloneqq V_{\BZ},
$$
and the $(d-1)$-simplices $A[d-1]$ are defined by the projection of the set
$$
\{\ud{n}^0,\dots, \ud{n}^{d-1}\in \brv{V}_{\BZ}\mid \exists \sigma\in \fkS_d,\, \ud{n}^{j}=\ud{n}^{j-1}+(\delta_{l,\sigma(j)})_{l},\ \forall j=1,\dots, d-1\}
$$
where $\fkS_d=\Aut\{0,\dots, d-1\}$ is the symmetric group. Note that for $r\in \BN$ such that $0\leqslant r \leqslant d-1$ an element of $A[r]$ is defined by a subset of an element of $A[d-1]$ with $r+1$ elements; hence the face maps are the natural inclusions. This defines the simplicial set $A$ and we now explain why $\lvert A \rvert\cong V_{\BR}$. For $\sigma \in \fkS_d$, let
$$
\wt{\Delta}_{\sigma}\coloneqq \{(x_0,\dots, x_{d-1})\in \BR^d \mid 0\leqslant x_{\sigma(0)}\leqslant \dots \leqslant x_{\sigma(d-1)}\leqslant 1\}
$$
and let $\Delta_{\sigma}\subset V_{\BR}$ be its image through the projection. Every maximal simplex in $V_{\BR}$ is of the form $\Delta_{\sigma}+ \ud{n}$ for $\sigma\in \fkS_d$ and $\ud{n}\in V_{\BZ}$. Moreover, $\Delta_{\sigma}$ is the standard $(d-1)$-dimensional simplex\footnote{The map $\Delta^{d-1}\xrightarrow{\sim} \Delta_1$ is given by $(y_i)\mapsto \bigl ( x_i=\sum_{j=0}^iy_j\bigr )$ with inverse $(x_i)\mapsto (y_i=x_i-x_{i-1})$. Notice how this makes the face maps look like those in $\BrTi_{\! d}$.}. This shows that $\lvert A \rvert = V_{\BR}\cong \BR^{d-1}$.

It remains to show that $A\cong \CA_T$ as simplicial sets. Let's fix a basis $\CB=\{e_0,\dots, e_{d-1}\}\subset E^d$ such that $T$ is diagonal in $\CB$, using the lattices from \eqref{eq:torlatdef}, we get a map 
$$
\begin{array}{rcl}
V_{\BZ} & \rightarrow & \lvert  \CA_T\rvert\\
 \ud{n}\ & \mapsto & \bar{\eta}_{\CB}(\ud{n}),
\end{array}
$$
which gives a bijection $V_{\BZ} \cong \CA_T[0]$. We need to show that this bijection matches $(d-1)$-simplices of $A$, which will be enough since the face maps then coincide. Let $\ud{n}^0,\dots, \ud{n}^{d-1}\in V_{\BZ}$, then it is clear that if $\{\ud{n}^j\}_j\in A[d-1]$, we get $\{\bar\eta_{\CB}(\ud{n}^j)\}_{j}\in \BrTi_{\! d}[d-1]$. Now if $\{\bar\eta_{\CB}(\ud{n}^j)\}_{j}\in \BrTi_{\! d}[d-1]$ we get in particular that 
$$
\varpi \eta_{\CB}(\ud{n}^{d-1})\eta_{\CB}(\ud n^0)\subsetneq \eta_{\CB}(\ud n^1)\subsetneq\cdots\eta_{\CB}(\ud{n}^{d-1}).
$$
Therefore,$\bar\eta_{\CB}(\ud{n}^j)/\eta_{\CB}(\ud{n}^{j+1})$ is of length $1$, hence, of the form $k_E e_{\sigma(j)}$. Since all $\sigma(j)$ are different, $\sigma\in \fkS_d$ and $\ud{n}^{j}=\ud{n}^{j-1}+(\delta_{l,\sigma(j)})_{l}$. This proves that $\{\ud{n}^j\}_j\in A[d-1]$ and finishes the proof.

\end{proof}

\medskip
\begin{rema}
Note that in the proof of the lemma, we show that $\lvert \CA_T\rvert \cong \Hom(X_{*}(\bar T),\BR)$ where $\bar T\subset \PGL_d(E)$ is the image of $T$ and $X_{*}(\bar T)$ is the root lattice. Remark that in the proof, we explain how the simplicial structure is defined by the affine root hyperplanes of the  \emph{Coxeter complex of type $\wt A_{d-1}$} on $\Hom(X_{*}(\bar T),\BR)\cong \BR^{d-1}$. In particular, the projections of $\Delta_{\sigma}$ are the affine Weyl chambers.
\end{rema}
\medskip
We mention the following lemma 
\medskip
\begin{lemm}\label{lem:sameap}
Let $\Delta, \Delta'\in \brv{\BrTi}_d$, then there exists a torus $T\subset \GL_d(E)$ such that $\Delta,\Delta'\in \brv{\CA}_T$. Moreover, for $x, x'\in \lvert \BrTi_d\rvert$, there exists a torus $T\subset \GL_d(E)$ such that $x, x'\in \lvert \CA_T\rvert$.
\end{lemm}
\medskip
\begin{proof}
See \cite[lemma 1.2]{mus} in this case. This phenomenon is a general property of buildings and we refer the reader to \cite[Corollary 12.17]{abbr} for a more conceptual argument.
\end{proof}
Since,
\begin{equation}\label{eq:aptdec}
\lvert \brv\BrTi_{\! d}\rvert =\bigcup_{T\subset \GL_d(E)}\lvert \brv\CA_T\rvert\cong \bigcup_{T\subset \GL_d(E)}\BR^{d-1}\times \BZ,
\end{equation}
we can endow $\lvert \brv{\BrTi}_d\rvert $ with a metric $\rmd_{\BrTi}$ given by the metric induced by the sup metric on $\BR^{d}$. More precisely, if $x,x'\in \brv{\BrTi}_{\! d}$, let $T\subset \GL_d(E)$ be a maximal split torus such that $x,x'\in \lvert \brv{\CA}_T\rvert \cong V_{\BR}\subset \BR^d$, then:
\begin{equation}
\rmd_{\BrTi}(x,x')\coloneqq \sup_{i=0,\dots, d-1 }\lvert x_i-x_i'\rvert
\end{equation}
This defines a $\GL_d(E)$-invariant metric on $\lvert \brv{\BrTi}_{\! d}\rvert$ (\cf \cite[\S 1]{mus}).

\subsubsection{Space of norms}
In this paragraph we recall that the geometric realization of the Bruhat--Tits building is the space of norms on $E^d$.

Following \cite{boer}, we recall that this metric space is isometric to $\brv \CN_d$ the \emph{space of norms on $E^d$}. More precisely, for $h\in \BZ$, let
$$
\medskip
\begin{gathered}
\CN_d^h\coloneqq \{\lVert \cdot \rVert \text{ norm on $E^d$ }\mid \lVert \rmO^{\times,d} \rVert=q^{-h}\}
\end{gathered}
\medskip
$$
and define $\brv{\CN}_d\coloneqq \bigsqcup_{h\in \BZ} \CN_d^h$. $\brv{\CN}_d$ is naturally endowed with an action of $\GL_d(E)$ induced by its standard action on $E^d$. Moreover, we set $\CN_d\coloneqq \brv{\CN}_d/E^{\times}$, where $E^{\times}\incl \GL_d(E)$ is the center, which parametrizes equivalence classes of norms on $E^d$.

The space $\brv\CN_d$ is endowed with the (rescaled) \emph{Goldman-Iwahori} metric: for two norms $\lVert \cdot \rVert, \lVert \cdot \rVert'\in \CN_d$, we define 
$$
\rmd_{\CN}(\lVert \cdot \rVert,\lVert \cdot \rVert')\coloneqq \frac{1}{\log q}\sup_{v\in E^d\setminus\{0\}} \lvert \log \lVert v \rVert-\log \lVert v \rVert'\rvert.
$$
We recall the following, which is essentially a consequence of the Arzelà--Ascoli theorem (\cf \cite[Proposition 1.8]{boer}) :
\begin{prop}
The metric space $(\brv{\CN}_d, \rmd_{\CN})$ is complete and any closed bounded subset is compact.
\end{prop}
\medskip
\begin{prop}\label{prop:normsp}
The metric space $(\lvert \brv{\BrTi}_{\! d} \rvert, \rmd_{\BrTi}) $ is isometric to $(\brv{\CN}_d,\rmd_{\CN})$.
\end{prop}
\medskip
The main point to prove this theorem is that any norm on $E^d$ is \emph{diagonalisable}: 
\medskip
\begin{lemm}\label{lem:diagnorm}
Let $\lVert \cdot \rVert$ be a norm on $E^d$. There exists a basis $\{e_0,\dots, e_{d-1}\}\subset E^d$ and a unique set $s_0,\dots, s_{d-1}\in [0,1)$ real numbers, such that 
$$
\forall x\in E^d,\quad \lVert x \rVert =\sup_{i=0,\dots, d-1} q^{s_i}\lvert x_i\rvert_E,
$$
where $x=\sum_{i=0}^{d-1}x_ie_i$.
\end{lemm}
\medskip
\begin{proof}
See \cite[2.4.2-2.4.3]{bgr}.
\end{proof}
We now shortly recall the proof of Proposition \ref{prop:normsp} ,which relies on an analogue of Lemma \ref{lem:sameap} for norms:
\begin{proof}
By passing to the quotient by the center, it is enough to show that $\lvert \BrTi_{\! d} \rvert$ and $\CN_d$ are isometric with the induced metrics. We will construct a bijection and then prove it is an isometry.

First let $\lVert \cdot \rVert$ be a norm on $E^d$. For any $s\in [0,1)$ define
$$
\eta_s\coloneqq \{x\in E^d \mid \lVert x \rVert\leqslant q^{s}\}.
$$
We then have that for $s'< s$, $\eta_{s'}\subset \eta_{s}$ and the set $\{\eta_s\}_{s\in [0,1)}$ is finite. The association $\lVert \cdot \rVert \mapsto \{\eta_s\}_{s\in [0,1)}$ defines a map $\CN_d\rightarrow \BrTi_{\! d}$.

Let $e_0,\dots, e_{d-1}$ be a basis of $E^d$, and le $\bar T\subset \PGL_d(E)$ be the associated torus. We can define $\CN_{\bar T}\subset \CN_d$ the associated apartment, consisting of norms which are diagonal in a basis associated to $\bar T$. By \cite[Proposition 1.17]{boer} we have an isometry $\BR^{d-1}\cong \CN_{\bar T}$ where $\BR^{d-1}$ is endowed with the sup-norm. By Lemma \ref{lem:descap}, we see that the restriction of the map above to $\CA_{\bar T}$  extends to a map on $\lvert \CA_{\bar T}\rvert$ given by the composition
$$
\CN_{\bar T}\cong \BR^{d-1}\cong \rvert \CA_{\bar T}\lvert.
$$
This map is an isometry and since $\CN_d=\bigcup_{\bar T\subset \PGL_d(E)}\CN_{\bar T}$ by Lemma \ref{lem:diagnorm} and $\lvert \BrTi_{\! d}\rvert =\bigcup_{\bar T\subset \PGL_d(E)}\lvert \brv\CA_{\bar T}\rvert$ by \eqref{eq:aptdec} which allows us to conclude the proof.
\end{proof}
\subsection{Drinfeld's and Deligne's formal models}

In this section, we define Drinfeld's formal model of the $p$-adic symmetric space. We relate the definition given in \cite[3.68-3.71]{razi} with Drinfeld's definition in \cite{drin} and use the same notations.

\subsubsection{Definition of Drinfeld's model}

We now define the moduli problem of $\brv{\BH}_{\rmO}^{d}$ following Drinfeld's definition in \cite{drin}.
\medskip
\begin{defi}\label{def:drin}
Let $h\in \BZ$, we define $\BH_{\rmO}^{d,h}$ as the functor on $\Nilp_{\rmO}$ defined by 
$$
\BH_{\rmO}^{d,h}\colon R\in \Nilp_{\rmO}\longmapsto \{(\eta, L, u, r)\},
$$
where:
\begin{itemize}
\itemb $\eta$ is a $\BZ/d$-graded constructible Zariski sheaf of flat $\rmO[\Pi]$-modules on $\Spec R$,
\itemb $L$ is a $\BZ/d$-graded  $R[\Pi]$-module whose graded components are invertible $R$-modules,
\itemb $u\colon \eta \rightarrow L$ is a graded $\rmO[\Pi]$-linear morphism such that the induced morphism
$$
u_R\coloneqq u\otimes R \colon \eta\otimes_{\rmO} R \rightarrow L
$$
is surjective,
\itemb $r\colon \ud{E}^d\xrightarrow{\sim} \eta^0\otimes_{\rmO}E$ is an $E$-linear isomorphism,
\end{itemize}
such that, for $i\in \BZ/d$ we write $S_{i}\coloneqq V(\Pi_i\colon L^i\mapsto L^{i+1})\subset S\coloneqq \Spec R$ for the vanishing locus of $\Pi_i$:
\begin{enumerate}
\item $\restr{\eta^i}{S_i}$ is isomorphic to the constant sheaf  $\ud{\rmO}^d$,
\item $\forall s\in S$, $u_s\colon\eta/\Pi\eta\rightarrow (L/\Pi L)\otimes k(s)$ is injective, where $k(s)$ is the residue field at $s$,
\item $\bigwedge^d\restr{\eta^i}{S_i}=\varpi^{-h-i}\bigwedge^d\restr{\Pi^ir(\ud{\rmO}^d)}{S_i}$.
\end{enumerate}
Finally, we set $\brv{\BH}_{\rmO}^{d}\coloneqq\bigsqcup_{h\in \BZ} \BH_{\rmO}^{d,h}$.
\end{defi}
The action of $\GL_d(E)$ on $\brv{\BH}_{\rmO}^{d}$ is defined for $g\in \GL_d(E)$ by
$$
g\colon (\eta, L, u, r)\mapsto (\eta, L, u, r\circ g^{-1}).
$$
Note that for any $h\in \BZ$, $\BH_{\rmO}^{d,h}\cong \BH_{\rmO}^{d,0}$ which is isomorphic to $\BH_{\rmO}^d\coloneqq \brv{\BH}_{\rmO}/E^{\times}$.

\medskip
The map $(\eta, L, u, r)\mapsto (\eta, r)$ defines a map of functors $\brv{\BH}_{\rmO}^d\rightarrow \brv{\sL}_d$, which by Proposition \ref{prop:repldbt} allows us, for any simplex $\Delta\in\brv{\BrTi}_{\! d}$ to define $\BH^{\Delta}_{\rmO}\subset \brv{\BH}_{\rmO}^d$.
\medskip
\begin{prop}
The functor $\brv{\BH}_{\rmO}^d$ is representable by a $p$-adic formal scheme over $\rmO$ and 
$$
\brv{\BH}_{\rmO}^d=\varinjlim_{\Delta\in \brv{\BrTi}_{\! d}}\BH^{\Delta}_{\rmO},
$$
in the category of formal schemes over $\rmO$.
\end{prop}
\medskip
To prove this proposition we will prove that the $\BH^{\Delta}_{\rmO}$ are representable by expressing them in a different way, the rest of the proposition is then immediate.

\subsubsection{Deligne's formal model}

Let $r\in \llbracket 0, d-1\rrbracket$ and $\Delta\in \brv{\BrTi_{\! d}}[r]$. Recall that we write $C_{\Delta} = \{i_0, \dots, i_r\}$ the set of \emph{indices} of $\Delta$.

\medskip
\begin{defi}
Let $F_{\Delta}$ be the functor on $\Nilp_{\rmO}$ defined for $R\in \Nilp_{\rmO}$ by 
$$
F_{\Delta}\coloneqq \{(\eta_{i_j}\xrightarrow{\varphi_{i_j}}L_{i_j})_j\}
$$
where the set consists of the commutative diagrams:
\begin{equation}\label{diag:drin}
\begin{tikzcd}
  \eta_{i_0} \arrow[r,hook]\arrow[d,"\varphi_{i_0}"] & \eta_{i_1} \arrow[r,hook]\arrow[d,"\varphi_{i_1}"] & \cdots \arrow[r,hook] & \eta_{i_r} \arrow[r,"\varpi"]\arrow[d,"\varphi_{i_r}"] & \eta_{i_0} \arrow[d] \\
  L_{i_0} \arrow[r,"x_{i_0}"] & L_{i_1} \arrow[r,"x_{i_1}"] & \cdots \arrow[r,"x_{i_{r-1}}"] & L_{i_r} \arrow[r,"x_{i_r}"] & L_{i_0},
\end{tikzcd}
\end{equation}
such that for $j\in \llbracket 0,r\rrbracket$, with the convention that $i_{r+1}=i_0$:
\begin{itemize}
    \itemb is an $L_{i_j}$ invertible $R$-modules,
    \itemb $x_{i_j}\colon L_{i_j}\rightarrow L_{i_{j+1}}$ is an $R$-linear morphism,
    \itemb $\varphi_{i_j}$ is a $\rmO$-linear morphism,
    \itemb The induced map
    $$
    \bar \varphi_{i_{j+1}}\colon \eta_{i_{j+1}}/\eta_{i_j}\rightarrow L_{i_{j+1}}/x_{i_j}L_{i_j},
    $$
    is pointwise injective, meaning that $\bar \varphi_{i_j}\otimes k(s)$ is injective for every $s\in \Spec R$, where $k(s)$ is the residue field at $s$.
\end{itemize}
\end{defi}
\medskip
If $\Delta\subset \Delta'$ we get an open immersion $F_{\Delta}\rightarrow F_{\Delta'}$ obtained by inverting the $x_i$'s such that $i\in C_{\Delta'}\setminus C_{\Delta}$. We can define $\varinjlim_{\Delta\in \brv{\BrTi}_{\! d}}F_{\Delta}$ and fixing the lattice in Definition \ref{def:drin}, the following is now immediate:
\medskip
\begin{prop}
We have an isomorphism of functors $F_{\Delta}\cong \BH^{\Delta}_{\rmO}$, in particular 
$$
\brv{\BH}_{\rmO}\cong \varinjlim_{\Delta\in \brv{\BrTi}_{\! d}}F_{\Delta},
$$
and $\BH^{\Delta}_{\rmO}$ is representable by a $p$-adic formal scheme.
\end{prop}
\medskip
The limit $\varinjlim_{\Delta\in \brv{\BrTi}_{\! d}}F_{\Delta}$ is equipped with an action of $\GL_{d}(E)$: for $g \in \GL_{d}(E)$, $g$ acts on a diagram of the form (\ref{diag:drin}) by $\eta_i \mapsto g \eta_i$ and $\varphi_i \mapsto \varphi_i \circ g^{-1}$ for all $i \in C_{\Delta}$. This makes the isomorphism of the proposition $\GL_d(E)$-equivariant. To finish the proof of the proposition, it remains to prove that $F_{\Delta}$ is representable by a $p$-adic formal scheme, which is given by the following lemma:
\medskip
\begin{lemm}\label{lem:semist}
Let $h\in \BZ$ be such that $\Delta\in \BrTi_{\! d}^{-h}$. The functor $F_{\Delta}$ is representable by a $p$-adic formal scheme over ${\rmO}$ which is an open formal subscheme of
\begin{equation}\label{eq:semist}
\Spf \frac{{\rmO}\langle x_{h}, \dots, x_{h+d-1} \rangle}{\langle x_h \cdots x_{h+d-1} - \varpi \rangle},
\end{equation}
where, for $i_k \in C_{\Delta}$, $x_{i_k}$ corresponds to the morphism $L_i \to L_{i_{k+1}}$.
\end{lemm}
\medskip
\begin{proof}
We can suppose $h=0$. It is enough to prove the claim for $\Delta$ a maximal simplex, since if $\Delta'\subset \Delta$ then $F_{\Delta'}\rightarrow F_{\Delta}$ is an open embedding. Moreover, we can suppose $\Delta=\{\boeta_{0},\dots, \boeta_{d-1} \}$ since $\GL_d(E)$ acts transitively on the $(d-1)$-simplices. By Lemma \ref{lem:explocmod} we get that $F_{\Delta}\incl \wh{\rmM}^{\loc}_{\rmO}$ is an open immersion hence $F_{\Delta}$ is representable by a $p$-adic formal subscheme. 

Moreover, we see that $F_{\Delta}$ is determined by the condition that $\varphi_i(e_i)$ vanishes nowhere. Hence by Remark \ref{rem:eqlocmod} we deduce that $F_{\Delta}$ is the closed subscheme of the $p$-adic completion of $(\BA^{d-1}_{\brvO})^d$ hence we get a morphism
\begin{equation}\label{eq:maptoss}
F_{\Delta}\rightarrow \Spf \frac{{\rmO}\langle x_0, \dots, x_{d-1} \rangle}{\langle x_0 \cdots x_{d-1} - \varpi \rangle}
\end{equation}
Moreover, if we write $T_0^{i-1},\dots, T_{d-1}^{i-1}$ the coordinates on the $i$-th copy of $\BP^{d-1}_{\rmO}$ given by $\varphi_{i-1}$, where both indices are considered in $\BZ/d$. Since $T^{i}_i$ corresponds to $\varphi_i(e_i)$, we get the equations for $F_{\Delta}$:
$$
\begin{gathered}
T_j^{i+1}=x_i T_j^{i},\, j\neq i,i+1\\
\varpi T^{i+1}_i=x_i T_{i}^{i}=x_i.
\end{gathered}
$$
Hence $T_j^{i}=x_{i-1}\dots x_{j}$ and the map \eqref{eq:maptoss} is an open immersion.
\end{proof}

\subsubsection{Generic fiber}
Let $\brv{\BH}_E^d$ be the rigid generic fiber of $\brv{\BH}_{\rmO}^d$; by the usual decomposition, we get
$$
\brv{\BH}_E^d=\BH_E^d\times \BZ
$$
where $\BH_{E}^d$ is the rigid generic fiber of $\BH_{\rmO}^d$. This is a rigid space over $E$, which we now describe more explicitly; it is Drinfeld's $p$-adic symmetric space.
\medskip
\begin{prop}\label{prop:symgenfib}
We have a $\GL_d(E)$-equivariant isomorphism of rigid spaces over $E$
$$
\BH_E^d\cong \BP_E^{d-1} \setminus \bigcup_{H \in \sH_E} H,
$$
where $\BP_E^{d-1}$ is the $d-1$-dimensional projective space viewed as a rigid spaces over $E$, $\sH_E$ is the set of hyperplanes of $\BP_E^{d-1}$ defined over $E$, and $\GL_d(E)$ acts through the inverse action on $\BP_E^{d-1}=\BP(E^d)$.
\end{prop}
\medskip
\begin{proof}

Let $A$ be an affinoid $E$-algebra and let $A^+\subset A$ be the ring of definition defined by the power bounded elements. Let $x\in \BH_E^d(A)=\BH_{\rmO}^d(A^+)$, using definition \ref{def:drin} $x=(\eta,L,u,r)$ and set
$$
g(x)\colon A^d\xrightarrow{r^0\otimes \id }\eta^0\otimes_{\rmO}A\xrightarrow{u^0}L^0\otimes_{A^+}A
$$
which defines a point $g(x)\in \BP_E^{d-1}(A)$ such that $g(x)\notin \bigcup_{H \in \sH_E} H(A)$ by Condition 2) of Definition \ref{def:drin}.
\end{proof}
Using this proposition, we can define a map 
$$
r\colon \lvert \brv{\BH}_E^d\rvert \rightarrow \lvert \brv{\BrTi}_{\! d}\rvert.
$$
For $K/E$ a valued finite field extension and $x\in \BH^d_E(K)$, we need define a norm $r(x)\coloneqq \lVert \cdot \rVert_x$ defined in the following way: trough Proposition \ref{prop:normsp}, the point $x$ defines an injective map $x\colon E^d\rightarrow K$, and $\lVert \cdot \rVert_x$ is obtained by restricting the field norm $\lvert \cdot \rvert_K$ to $E^d$. We get the following lemma, which is obtained through the proof of Proposition \ref{prop:normsp}:
\medskip
\begin{lemm}\label{lem:compatlatnorm}

Let $e$ be the ramification index of $K/E$. Let $e_0,\dots, e_{d-1}\in E^d$ be a basis of $E^d$ and $s_1,\dots, s_d\in [0,1)$ be parameters of $\lVert \cdot \rVert_x$  given by Lemma \ref{lem:diagnorm} ordered so that $s_0\geqslant \dots \geqslant s_{d-1}$. Then the simplex $\{\eta_{i_0},\dots, \eta_{i_r}\}$ associated with $\lVert \cdot \rVert_x$ by Proposition \ref{prop:normsp} is the unique chain of lattices such that $x$ induces a diagram
\begin{center}
\begin{tikzcd}
  \eta_{i_0} \arrow[r,hook]\arrow[d,"x"] & \eta_{i_1} \arrow[r,hook]\arrow[d,"x"] & \cdots \arrow[r,hook] & \eta_{i_r} \arrow[r,"\varpi"]\arrow[d,"x"] & \eta_{i_0} \arrow[d] \\
  \fkm_K^{es_{i_0}} \arrow[r,hook] & \fkm_K^{es_{i_1}}\arrow[r,hook] & \cdots \arrow[r,hook] & \fkm_K^{es_{i_r}}\arrow[r,"\varpi"] &\fkm_K^{es_{i_0}},
\end{tikzcd}
\end{center}
and such that the induced map $x\otimes \rmO_K\colon \eta_{i_j}\otimes_{\rmO}\rmO_K\rightarrow \fkm_K^{es_{i_j}}$ is surjective.
\end{lemm}
\medskip

\medskip
\begin{coro}\label{cor:matchlatnorm}
The following natural diagram is commutative: 
\begin{center}
\begin{tikzcd}
\lvert \brv{\BH}_E^d\rvert \ar[r,"r"]\ar[dr]& \lvert \brv{\BrTi}_{\! d}\rvert\ar[d]\\
 & \brv{\BrTi}_{\! d}
\end{tikzcd}
\end{center}
\end{coro}
\medskip
\begin{rema}
Using the theory of Berkovich the map $r$ can be made into a genuine retraction of the Berkovich space $\brv{\BH}_E^d$ on the building $\lvert \brv{\BrTi}_{\! d}\rvert$ which gives the \emph{skeleton} of the space $\BH_E^d$ (\cf \cite{ber}).
\end{rema}
\medskip
\FloatBarrier

\definecolor{mygreen}{RGB}{0,160,0}

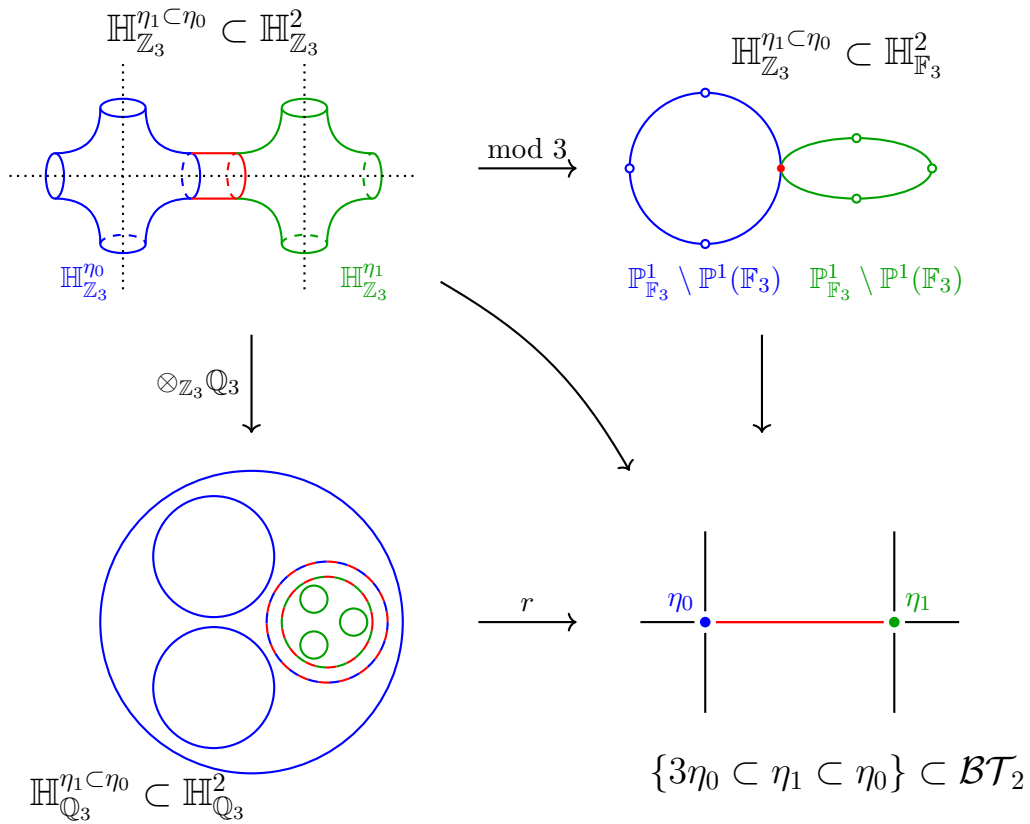
\begin{figure}
\centering
\begin{tikzpicture}[scale=1, thick]%adapt scale if needed
%
%south-west part of figure (circles)
%
\node[] (O) at (0,0){};
\draw[blue] (O) circle (2);
\node[] (O1) at (1,0){};
\draw[red, dashed, dash pattern=on 4pt \space off 4pt,thick] (O) + (0:1) circle (0.8);
\draw[blue, dashed, dash pattern=on 4pt \space off 4pt,dash phase=4pt,thick] (O) + (0:1) circle (0.8);
\draw[blue] (O) + (120:1) circle (0.8);
\draw[blue] (O) + (240:1) circle (0.8);
%\draw[red] (O1) circle (0.6);
\draw[red, dashed, dash pattern=on 4pt \space off 4pt,thick] (O1) circle (0.6);
\draw[mygreen, dashed, dash pattern=on 4pt \space off 4pt,dash phase=4pt,thick] (O1) circle (0.6);
\draw[mygreen] (O1) + (0:0.35) circle (0.18);
\draw[mygreen] (O1) + (120:0.35) circle (0.18);
\draw[mygreen] (O1) + (240:0.35) circle (0.18);
%
% south-east part of figure (BT tree)
%
\node[] (OL) at (5,0){};
\node[] (OR) at (9.5,0){};
\node[] (OL1) at (6,0){};
\draw (OL1) node[above left,blue]{$\eta_0$};
\fill[blue] (OL1) circle (0.07); 
\node[] (OL2) at (8.5,0){};
\draw (OL2) node[above right,mygreen]{$\eta_1$};
\fill[mygreen] (OL2) circle (0.07); 
\draw[] (OL) -- (OL1); %horizontal line 
\draw[red] (OL1) -- (OL2); %horizontal line 
\draw[] (OL2) -- (OR); %horizontal line 
\draw[] (OL1)+(90:1.2) -- (OL1);
\draw[] (OL1)+(270:1.2) -- (OL1);
\draw[] (OL2)+(90:1.2) -- (OL2);
\draw[] (OL2)+(270:1.2) -- (OL2);
\draw[red] (OL1) -- (OL2);
%
%north-east part of figure (intersecting ellipses)
%
\node[] (EL1) at (6,6){};
\draw[name path=ellipse1,blue] (EL1) ellipse (1cm and 1cm);
\node[] (EL2) at (7,6){};
\draw[mygreen] (EL2) arc (180:0:1 and 0.4);
\draw[mygreen] (EL2) arc (180:360:1 and 0.4);
\fill[red] (EL1)+(1,0) circle (1.5pt);
\foreach \angle in {90,180,270} {
  \draw[blue,fill=white] ({6+cos(\angle)},{6+sin(\angle)}) circle (1.5pt);
  }
\foreach \angle in {360,90,270} {
  \draw[mygreen,fill=white] ({8+cos(\angle)},{6+0.4*sin(\angle)}) circle (1.5pt);
  }
\draw[thick,->] (6.75,3.8) -- (6.75,2.5);
%
%north-west part of the picture (the string monster)
%
\begin{scope}[shift={(-2,5)}]
\def\a{0.6}
\def\b{\a/5}
\def\c{\a/2}
%
%ellipses first
%
\draw[dashed,blue] (0,0) arc (180:0:{\c} and \b);
\draw[solid,blue] (0,0) arc (180:360:{\c} and \b);
\draw[solid,blue] (0,3*\a) arc (180:0:{\c} and \b);
\draw[solid,blue] (0,3*\a) arc (180:360:{\c} and \b);
\draw[solid,blue] (-\a,\a) arc (270:90:{\b} and {\c});
\draw[solid,blue] (-\a,\a) arc (-90:90:{\b} and {\c});
\draw[dashed,blue] (2*\a,\a) arc (270:90:{\b} and {\c});
\draw[solid,blue] (2*\a,\a) arc (-90:90:{\b} and {\c});
\draw[dashed,red] (3*\a,\a) arc (270:90:{\b} and {\c});
\draw[solid,red] (3*\a,\a) arc (-90:90:{\b} and {\c});
\draw[dashed,mygreen] (4*\a,0) arc (180:0:{\c} and \b);
\draw[mygreen] (4*\a,0) arc (180:360:{\c} and \b);
\draw[mygreen] (4*\a,3*\a) arc (180:0:{\c} and \b);
\draw[mygreen] (4*\a,3*\a) arc (180:360:{\c} and \b);
\draw[dashed,mygreen] (6*\a,\a) arc (270:90:{\b} and {\c});
\draw[mygreen] (6*\a,\a) arc (-90:90:{\b} and {\c});
%
%connecting lines
%
\draw[solid,blue] plot[smooth,tension=1] coordinates {(0,0) (-\a/4,3*\a/4) (-\a,\a)};
\draw[solid,blue] plot[smooth,tension=1] coordinates {(-\a,2*\a) (-\a/4,9*\a/4) (0,3*\a)};
\draw[solid,blue] plot[smooth,tension=1] coordinates {(\a,3*\a) (5*\a/4,9*\a/4) (2*\a,2*\a)};
\draw[solid,blue] plot[smooth,tension=1] coordinates {(\a,0) (5*\a/4,3*\a/4) (2*\a,\a)};
\draw[thick,red] plot[smooth,tension=1] coordinates {(2*\a,2*\a) (3*\a,2*\a)};
\draw[thick,red] plot[smooth,tension=1] coordinates {(2*\a,\a) (3*\a,\a)};
\draw[solid,mygreen] plot[smooth,tension=1] coordinates {(4*\a,0) (15*\a/4,3*\a/4) (3*\a,\a)};
\draw[solid,mygreen] plot[smooth,tension=1] coordinates {(5*\a,0) (21*\a/4,3*\a/4) (6*\a,\a)};
\draw[solid,mygreen] plot[smooth,tension=1] coordinates {(3*\a,2*\a) (15*\a/4,9*\a/4) (4*\a,3*\a)};
\draw[solid,mygreen] plot[smooth,tension=1] coordinates {(5*\a,3*\a) (21*\a/4,9*\a/4) (6*\a,2*\a)};
%
%BT tree inside monster
%
\draw[dotted] (-2*\a,3*\a/2) -- (7*\a,3*\a/2);
\draw[dotted] (\a/2,-\a) -- (\a/2,4*\a);
\draw[dotted] (9*\a/2,-\a) -- (9*\a/2,4*\a);
\end{scope}

\node[] (txt1) at (-0.5,7.8){\Large{$\BH^{\eta_1\subset\eta_0}_{\BZ_3}\subset\BH^2_{\BZ_3}$}};
\node[] (txt2) at (7.75,7.5){\Large{$\BH^{\eta_1\subset\eta_0}_{\BZ_3}\subset\BH^2_{\BF_3}$}};
\node[] (txt3) at (-1.5,-2.3){\Large{$\BH^{\eta_1\subset\eta_0}_{\BQ_3}\subset \BH^2_{\BQ_3}$}};
\node[] (txt4) at (7.75,-2){\Large{$\{3\eta_0\subset\eta_1\subset\eta_0\}\subset \BrTi_{\!2}$}};
\node[mygreen] (txt7) at (8.4,4.5){{$\BP^1_{\BF_3}\setminus\BP^1(\BF_3)$}};
\node[blue] (txt5) at (6,4.5){{$\BP^1_{\BF_3}\setminus\BP^1(\BF_3)$}};
\node[blue] (txt9) at (-2.2,4.5){{\textcolor{blue}{$\BH_{\BZ_3}^{\eta_0}$}}};
\node[mygreen] (txt10) at (1.5,4.5){{$\BH_{\BZ_3}^{\eta_1}$}};

\draw[thick,->] (3,0) -- (4.3,0) node[above,midway] {$r$};
\draw[thick,->] (0,3.8) -- (0,2.5) node[left,midway]{$\otimes_{\BZ_3}\BQ_3$};
\draw[thick,->] (3,6) -- (4.3,6) node[above,midway] {mod 3};
\draw[thick,->] (2.5,4.5) to [bend left=15] (5,2) node[] {};
\end{tikzpicture}
\caption{The formal scheme $\BH_{\BZ_3}^{\Delta}$ (upper left) for $d=2$ and $E=\BQ_3$ of the Drinfeld symmetric space lying above a maximal simplex $\Delta=\{\eta_0,\eta_1\}\subset \BrTi_{\! 2}$ (lower right), its special fiber $\BH_{\BF_3}^{\Delta}$ (upper right) and its rigid generic fiber $\BH_{\BQ_3}^{\Delta}$ (lower left).}
\end{figure}

\newpage

\section{The generic fibers}\label{sec:genfib}
The purpose of this section is to prove the following 
\medskip
\begin{theo}\label{thm:genef}
There exists a period morphism $\pi \colon \rmM_{\brv E}\rightarrow \BP^{d-1}_{\brv E}$ which is an open immersion with image $\BH_{\brv E}^d$ giving a $\GL_d(E)$-equivariant isomorphism
$$
\brv \pi \colon \brv{\rmM}_{\brv E}\xrightarrow{\sim} \brv{\BH}_{\brv E}^d
$$
\end{theo}
\medskip
The construction of $\brv \pi$ out of $\pi$ is easy, but we need to be a little bit careful for it to be compatible with Weil descent (\cf \S\ref{subsec:weildes}). Recall that we have a decomposition indexed by the height of the quasi-isogeny $\brv{\rmM}= \sqcup_{i\in \BZ}\rmM_i$ and a decomposition $\brv{\BH}_{\brv E}^d=\sqcup_{i\in \BZ}\BH^d_{\brv E,i}$ indexed by the height of the lattice. Then $\pi_0\coloneqq \pi$ and $\pi_i\coloneqq \brv{\varpi}^i\circ \pi_0\circ \varpi^{-i}$ for $i\in \BZ$ which allows us to define $\brv \pi = \sqcup_{i\in \BZ}\pi_i$.
\subsection{Vector bundles on the Fargues--Fontaine curve} In this section we recall some facts about the Fargues--Fontaine curve and the classification of vector bundles over it.
\subsubsection{The Fargues--Fontaine curve}
We begin by recalling some facts from \cite[Chapitre 6]{fafo}. Let $C$ be a perfectoid field containing $E$ and let $\rmO_C\subset C$ be its ring of integers. Let $A_{\inf}\coloneqq \rmW_{\rmO}(\rmO_C^{\flat})$ where $\rmO_C^{\flat}\coloneqq \varprojlim_{x\mapsto x^p}\rmO_C$. The ring $A_{\inf}$ is endowed with a Frobenius map $\varphi$ defined by the Witt vector Frobenius of $W_{\rmO}$ and with a surjective map $\theta \colon A_{\inf}\rightarrow \rmO_C$ called \emph{Fontaine's map}. We choose $\varpi^{\flat}\in \rmO_C^{\flat}$ defined by a compatible system of $p$-th power roots of $\varpi$ and consider its Teichmüller lift $[\varpi^{\flat}]\in A_{\inf}$, we then get $\ker\theta=\xi A_{\inf}$ where $\xi \coloneqq(\varpi-[\varpi^{\flat}])$. 

Let $\Bdrp$ be the \emph{de Rham period ring}, which is defined as the $\ker \theta$-completion of $A_{\inf}\bigl [\tfrac 1 p \bigr]$. Fontaine's map defines a map $\theta \colon \Bdrp \rightarrow C$. Let $\varepsilon\in \rmO_C^{\flat}$ be a compatible system of $p$-th roots of unity, then the series $t\coloneqq\log[\varepsilon]$ converges and defines an element in $\Bdrp$ which generates $\ker \theta$. We also write $\Bdr\coloneqq \Bdrp\bigl [ \tfrac 1 t \bigr]$.

Let $\Bcrisp$ be the \emph{crystalline period ring}, which is defined as the $p$-adic completion of the divided power envelope of $\ker\theta\subset A_{\inf}\bigl [\tfrac 1 p \bigr]$.  The crystalline period ring is endowed with a Frobenius map induced by the Witt vector Frobenius, which we still denote $\varphi$. We have a natural inclusion $\Bcrisp\subset \Bdrp$ and one can check that $t\in \Bcrisp$ so we write $\Bcris\coloneqq\Bcrisp\bigl [ \tfrac 1 t \bigr]$.

Now define the following graded $\Qp$-algebra
$$
P\coloneqq \bigoplus_{i\geqslant 0} (\Bcrisp)^{\varphi=\varpi^i}.
$$
The \emph{Fargues--Fontaine curve} is then defined as $\CX\coloneqq \Proj P$. Fargues and Fontaine show (\cf \cite[Théorème 6.5.2]{fafo}) that $\CX$ is a separated, connected, regular, scheme of dimension $1$ that is complete: there is a degree function $\deg \colon \Div(\CX)\rightarrow \BZ$ inducing an isomorphism 
\begin{equation}\label{eq:pic}
\Pic(\CX)\xrightarrow{\sim} \BZ.
\end{equation}
The main point of the proof is the fact that $(\Bcrisp)^{\varphi=1}$ is a principal ring. We will simply write $\CO\coloneqq \CO_{\CX}$ for the sheaf of regular functions on $\CX$. In particular, for $k\in \BZ$ we write $\CO(k)$ the associated line bundle on $\CX$ given by \eqref{eq:pic}.

The element $t\in \Bcrisp$, together with Fontaine's map, defines a $C$-point $\iota_{\infty}\colon \Spec(C)\incl \CX$ called the \emph{point at infinity} and simply written $\infty\in \CX$. There is a natural isomorphism between the completed local ring of $\CO$ at $\infty$ with the de Rham period ring, \ie $\wh{\CO}_{\infty}\cong \Bdrp$. We write $U=\CX\setminus \{\infty\}$ and we get $\Gamma(U,\CO)=(\Bcrisp)^{\varphi=1}$.

Let $n\geqslant 0$ be an integer. We define an étale cover $\pi_n\colon \CX_n\rightarrow \CX$ where $\CX_n$ is the Fargues--Fontaine curve associated to $E_n$, the (unique) unramified extension of degree $n$ of $E$. In other words $\pi_n\colon \CX_n\coloneqq \Proj P_n\rightarrow \CX$ where
$$
P_n\coloneqq \bigoplus_{i\geqslant 0} (\Bcrisp)^{\varphi^n=\varpi^i},
$$
and the map $\CX_n\rightarrow \CX_0=\CX$ is defined by the natural inclusion $P_0\incl P_n$. We simply write $\CO_n\coloneqq \CO_{{\CX}_n}$, hence $\CO_0=\CO$.
\subsubsection{Vector bundles on the Fargues--Fontaine curve}
To every $\sigma$-isocrystal $(D,\phi)$ on $ k$ we can associate a vector bundle $\CE(D)$ on $\CX$ by the graded $P$-module
$$
E(D)\coloneqq \bigoplus_{i\geqslant 0} (D\otimes_{\brv E}\Bcrisp)^{\phi\otimes\varphi=\varpi^i}.
$$
Hence, the rank of $\CE(D)$ is the dimension of $D$.

In particular, we know that to any $\lambda\in \BQ$ one can associate a simple $\sigma$-isocrystal $D_{\lambda}=(D_{\lambda},\phi_{\lambda})$ of slope $\lambda$ and we simply write $\CO(\lambda)\coloneqq \CE(D_{\lambda})$. Recall that if $\lambda = d/h$ where $d\in \BZ$ and $h>0$ are coprime, then $D_{\lambda}$ is a $\brv E$-vector space of dimension $h$ and $\phi_{\lambda}=C_{\lambda}\sigma$ where $\sigma$ is the Frobenius automorphism and $C_{\lambda}$ is the companion matrix associated to the polynomial $X^{h}-\varpi^{d}$, namely: 
\[
C_{\lambda} =
\begin{pmatrix}
0 & 0 & \cdots & 0 & \varpi^d \\
1 & 0 & \cdots & 0 & 0 \\
0 & 1 & \ddots & \vdots & \vdots \\
\vdots & \ddots & \ddots & 0 & 0 \\
0 & \cdots & 0 & 1 & 0
\end{pmatrix}.
\]
Hence $\CO(\lambda)$ is associated to the $P$-graded module
$$
E_{\lambda}\coloneqq \bigoplus_{i\geqslant 0}(\Bcrisp)^{\varphi^h=\varpi^{i+d}}.
$$
For $\lambda \in \BQ$ we define $m(\lambda)\coloneqq \inf \{k\in \BN\mid k\lambda\in \BZ\}$, and for $n\geqslant 0$ we have 
$$
\pi^{*}_n\CO(\lambda)=\CO_n(n \lambda)^{\oplus\tfrac{m(\lambda)}{m(n\lambda)}}
$$
Moreover, since for $k_1,k_2\in \BZ$ we have $\CO(k_1)\otimes_{\CO}\CO(k_2)=\CO(k_1+k_2)$, for $\lambda_1,\lambda_2\in \BQ$ we get
$$
\CO(\lambda_1)\otimes_{\CO}\CO(\lambda_2)=\CO(\lambda_1+\lambda_2)^{\tfrac{m(\lambda_1)m(\lambda_2)}{m(\lambda_1+\lambda_2)}}.
$$
Fargues and Fontaine compute the following:
\begin{prop}\label{prop:expsec}
Let $\lambda\in \BQ$ and write $\lambda=d/h$ where $d\in \BZ$ and $h>0$ are prime with each other. Then
\begin{itemize}
\itemb If $\lambda =0$, then $H^1({\CX},{\CO})=0$ and  $H^0({\CX},{\CO})=E$
 \itemb If $\lambda > 0$, then $H^1({\CX},{\CO(\lambda)})=0$ and 
 $$
 H^0(\CX,\CO(\lambda))=(\Bcrisp)^{\varphi^h=\varpi^d}.
 $$
 \itemb If $\lambda <0$, then $H^0({\CX},{\CO(\lambda)})=0$ and 
 $$
 H^1({\CX},{\CO(\lambda)})=\Bdrp/(t^{-d}\Bdrp+E_h).
 $$
\end{itemize}
\end{prop}
\begin{exem}\label{exa:dactond}
The isocrystal $\bM$ defines a vector bundle on $\CX$ and we have
$$
\CE_{\Dr}\coloneqq \CE(\bM(1))=\CO\left ( \tfrac 1 d \right )^{\oplus d}.
$$
Hence $H^0(\CX,\CE_{\Dr})=\bigoplus_{i\in \BZ/d}\Bcris^{\varphi^d=\varpi}$ and the action of $\Pi$ is given by $\varphi$. Note that the graded decomposition is not the grading induced by the action of $\rmO_d$, but rather a grading transverse to it.
\end{exem}
\subsubsection{Banach--Colmez spaces}
The cohomology groups of Proposition \ref{prop:expsec} can be interpreted as $C$-points of \emph{Banach--Colmez} spaces. The category $\BaCo$ of Banach--Colmez spaces as initally been constructed by Colmez \cite{col1}, \cite{col2}. We follow Le Bras' presentation \cite[Proposition]{leb}.

Let $\Perf_C$ be the category of perfectoid algebras over $C$ endowed with the proétale topology.  We write $\Vect_E$ (\resp $\Vect_C$) for the category of finite-dimensional $E$-vector spaces (\resp $C$-vector spaces). To $E$ we associate the constant sheaf $\ud{E}$, and to $C$ the additive sheaf $\BG_a$ defined for $S\in \Perf_C$ by $\BG_a(S)=S$. Then (\cf \cite[Definition 1.1]{leb}):
\medskip
\begin{defi}
The category $\BaCo$ of \emph{Banach--Colmez} spaces is the smallest abelian subcategory of proétale sheaves of $E$-vector space son $\Perf_C$ which is stable under extensions and contains $\ud{E}$ and $\BG_a$.
\end{defi}
\medskip

Let us mention that the category $\BaCo$ can also be interpreted in terms of sheaves on the Fargues--Fontaine curve (\cf \cite[Théorème 7.1]{leb}). There are fully faithful embeddings
$$
\Vect_E\incl \BaCo,\quad \Vect_C\incl \BaCo.
$$
Moreover, $\BaCo$ is endowed with a faithful exact functor to $\Qp$-Banach spaces given by $X\mapsto X(C)$. We say that a Banach space \emph{lifts to a Banach--Colmez space} if it is in the essential image of this functor. The main result of Colmez is that these spaces admit a \emph{Dimension} for which they satisfy a \emph{rank theorem} (\cf \cite[Théorème 1.4]{col1}):
\medskip
\begin{theo}\label{thm:bcspace}
The category $\BaCo$ is equipped with an $E$-\emph{Dimension} function\footnote{Classically, the first argument of this Dimension was called the \emph{principal dimension} and the second one the \emph{residual dimension}. Now, the terms of $C$-\emph{dimension} and $E$-dimension, or, in the analogy with $p$-divisible groups, of dimension and height are preferred.} on objects:
$$
\Dim_E\colon \BaCo\rightarrow \BN\times \BZ,
$$
which is additive and such that $\Dim_E\ud{E}=(0,1)$ and $\Dim_E\BG_a= (1,0)$. Moreover, the subcategory $\Vect_E$ (\resp $\Vect_C$) identifies with the objects $X\in \BaCo$ such that the first component (\resp the second component) component of $\Dim_EX$ is $0$.
\end{theo}
\medskip
\medskip
\begin{exem}\label{exa:bccomp}
For $\lambda \in \BQ$ such that $\lambda=d/h$ where $d\in \BZ$ and $h>0$ are coprime, then 
\begin{itemize}
\itemb if $\lambda>0$, $H^0(\CX,\CO(\lambda))$ lifts to a Banach--Colmez space of $E$-Dimension $(d,h)$,
\itemb if $\lambda<0$,  $H^1(\CX,\CO(\lambda))$ lifts to a Banach--Colmez space of $E$-Dimension $(-d,h)$.
\end{itemize}
\end{exem}

\subsubsection{Classification of vector bundles}
Let $\CE$ be a vector bundle on $\CX$, we define the \emph{rank} $\rank \CE$ as the rank of the vector bundle and the \emph{degree} by
$$
\deg \CE\coloneqq\deg \det\CE.
$$
The \emph{Harder--Narasimhan slope} is the function on vector bundles defined by 
$$
\mu\coloneqq \tfrac{\deg}{\rank}.
$$
The main point of this function is that we get a \emph{generalized Harder--Narasimhan theory} as defined in \cite[5.5]{fafo}. Recall the following definition:
\medskip
\begin{defi}
 A vector bundle $\CE$ on $\CX$ is called \emph{semi-stable} if for every nonzero subbundle $\CE'\subset \CE$ we have $\mu(\CE')\leqslant \mu(\CE)$ and \emph{stable} if, moreover, $\mu(\CE')=\mu(\CE)\implies \CE'=\CE$. 
\end{defi}
\medskip

\begin{exem}
We have $\mu(\CO(\lambda))=\lambda$. Let us quickly explain this, we write $\lambda=d/h$ as before. Since $\CX_h\rightarrow \CX$ is an étale map of degree $h$ we have\footnote{This is actually not an argument but a trustworthy proof requires to define $\deg$ more in detail.} $\deg \pi^*_h\CO(\lambda)=h\deg\CO_h(\lambda)$. But we stated that $\pi^*_h\CO(\lambda)=\CO_h(d)^h$ thus $\deg \pi^*_h\CO(\lambda)=dh$. Hence $\deg\CO(\lambda)=d$ and we already explained that $\rank \CO(\lambda)=h$.
\end{exem}
\medskip
Recall the classification theorem for vector bundles on $\CX$.
\medskip
\begin{theo}\label{thm:clasvb}
The map $D\mapsto \CE(D)$ induces a bijection between isomorphism classes of $\sigma$-isocrystals over $k$ and isomorphism classes of vector bundles on $\CX$. Moreover, any semistable vector bundle is of the form $\CO(\lambda)^{\oplus m}$ where $\lambda \in \BQ$ and $m\geqslant 0$ is an integer.

\end{theo}
\medskip
Since the usual proof (\cf \cite[Chapitre 8]{fafo}) of the classification theorem uses Drinfeld's representability theorem we want to recall the second proof given in \cite[8.4.2]{fafo} using Banach--Colmez spaces instead.
\begin{proof}
By \cite[Proposition 5.6.26]{fafo} we only need to check that for any integer $h\geqslant 0$ and any vector bundle $\CE$ on $\CX_h$ that fits into an exact sequence of the form
\begin{equation}\label{eq:extvb}
0\rightarrow \CO_h\left ( -\tfrac 1 n \right ) \rightarrow \CE \rightarrow \CO_h(1)\rightarrow 0,
\end{equation}
where $n\geqslant 1$ is an integer, we have $H^0(\CX_h,\CE)\neq 0$.  We can suppose $h=0$ since we work with any $E$. Applying $H^0$ to \eqref{eq:extvb}, we need to show that the connection map
$$
f\colon H^0(\CX,\CO(1))\rightarrow H^1(\CX,\CO\left ( -\tfrac 1 n \right ) )
$$
is not injective. 

But by proposition \ref{prop:expsec} we get 
$$
H^0(\CX,\CO(1))=(\Bcrisp)^{\varphi=\varpi},\quad H^1(\CX,\CO\left ( -\tfrac 1 n \right ) )=C/E_n.
$$
We now explain why there is no injective map $(\Bcrisp)^{\varphi=\varpi}\rightarrow C/E_n$ by lifting this morphism to the category of Banach--Colmez spaces. The left-hand side lifts to a Banach--Colmez space of $E$-Dimension $(1,1)$ and the right-hand side to a Banach--Colmez space of Dimension $(1,-n)$. Suppose given an injective map $f\colon (\Bcrisp)^{\varphi=\varpi}\rightarrow C/E_n$ and write $V=\coker(f)$ which by Theorem \ref{thm:bcspace} is a Banach--Colmez space of $E$-Dimension $(0,-n-1)$. So $V$ is an $E$-vector space of strictly negative dimension which is impossible.
\end{proof}
An important corollary of the classification theorem is the following:
\medskip
\begin{coro}
Semistable vector bundles on $\CX$ of slope $0$ correspond to finite-dimensional $E$-vector spaces by the association $V\mapsto V\otimes_E\CO$.
\end{coro}
\medskip
We now simply call a \emph{trivial vector bundle} a semistable vector bundle of slope $0$.
\subsubsection{(Minuscule) modifications of vector bundles (at $\infty$)}
We will only consider modifications of vector bundles at infinity. Recall that $U=\CX\setminus \{\infty\}$.
\medskip
\begin{defi}
 Let $\CE$ and $\CE'$ be vector bundles $\CX$. We say that a morphism $f\colon \CE'\rightarrow \CE$ is a \emph{modification} if it is injective and induces an isomorphism
 $$
 \restr{\CE}{U}\cong \restr{\CE'}{U}.
 $$
We will simply say that \emph{$\CE'$ is a modifications of $\CE$}. Finally, if the cokernel of $f$ is killed by $t$, we say that \emph{$\CE'$ is a minuscule modification of $\CE$} and we say that it is a modification of degree $r \in \BN$ if this cokernel is of the form $\iota_{\infty,*}C^{r}$. 
\end{defi}
\medskip
We recall that if $\CE'$ is a minuscule modification of degree $r$ of $\CE$ then (\cf \cite[5.5.2.1]{fafo})
$$
\rank \CE=\rank \CE',\quad \deg \CE = \deg \CE'+r.
$$
\medskip
\begin{rema}
Usually, the \emph{type} of a modification is given by a cocharacter $\mu$. If $\CE$ is a vector bundle of rank $n$ on $\CX$ then a modification of degree $r\in \BN$ corresponds to the minuscule cocharacter
$$
\mu=(\underbrace{1,\dots, 1}_{r}, \underbrace{0,\dots, 0}_{n-r})
$$
of $\GL_n(E)$. 
\end{rema}
\medskip
Notice that if $\CE'$ is a modification of $\CE$ the cokernel is indeed a $\Bdrp$-torsion module.
We recall the classical result of Beauville-Laszlo applied to the curve. We define \emph{Berger pairs} as pairs $(D,M^+)$ where:
\begin{itemize}
\itemb $D$ is a projective $\rmB_e=(\Bcrisp)^{\varphi=1}$-module of finite rank,
\itemb $M^+$ is a $\Bdrp$-submodule of $D\otimes_{\rmB_e}\Bdr$ of maximal rank.
\end{itemize}
These pairs naturally form a category. To a vector bundle $\CE$ on $\CX$ we associate such data by $(H^0(U,\CE),\widehat{\CE}_{\infty})$ where $\widehat{\CE}_{\infty}$ is the completion of the stalk of $\CE$ at $\infty$. The Beauville-Laszlo theorem then states:
\begin{prop}\label{prop:bpair}
The association $\CE \mapsto (H^0(U,\CE),\widehat{\CE}_{\infty})$ defines an equivalence of categories
$$
\left\{\text{Vector bundles on $\CX$} \right\}
\cong {} 
\left \{\text{Berger pairs $(D,M^+)$}\right \}.
$$
\end{prop}
To a Berger pair $(D,M^+)$ we can thus associate a vector bundle written $\CE(D,M^+)$. This allows us to identify modifications of $\CE$ with lattices in $\widehat{\CE}_{\infty}$ \ie Proposition \ref{prop:bpair} gives an equivalence of categories:
$$
\left \{ \begin{array}c
\text{Modifications of $\CE$}
\end{array} \right \}
\cong {} 
\left \{ \begin{array}c
\text{$\Bdrp$-submodules $M^+\subset \widehat{\CE}_{\infty}$}\\
\text{such that $M^+\bigl [ \tfrac 1 t \bigr ]=\widehat{\CE}_{\infty}\bigl [ \tfrac 1 t \bigr ]$ }
\end{array} \right \}
$$
This allows to describe minuscule modifications as lattices $M^+\subset \widehat{\CE}_{\infty}\bigl [ \tfrac 1 t \bigr ]$ such that $t\widehat{\CE}_{\infty}\subset M^+\subset\widehat{\CE}_{\infty}$, sometimes called \emph{minuscule lattices}.
\medskip
\begin{prop}\label{prop:minmod}
Let $\CE$ be a vector bundle on $\CX$. Set $n=\rank \CE$ and let $r$ be an integer such that $0< r<n$. There is a natural bijection
$$
\left \{ \begin{array}c
\text{Isomorphism classes of}\\
\text{minuscule modifications of $\CE$}\\
\text{of degree $r$.}
\end{array} \right \}
\cong {} 
\Gr_r(\widehat{\CE}_{\infty}/t\widehat{\CE}_{\infty}),
$$
where $\Gr_r(\widehat{\CE}_{\infty}/t\widehat{\CE}_{\infty})\cong \Gr_{r,n}(C)$ is the Grassmannian of codimension $r$ subspaces in $\widehat{\CE}_{\infty}/t\widehat{\CE}_{\infty}\cong C^n$.
\end{prop}
\medskip
\begin{rema}
If $D$ is a $\sigma$-isocrystal over $k$ and $\CE=\CE(D)$ then the proposition \ref{prop:minmod} gives a natural bijection between minuscule modifications of degree $r$ of $\CE$ and $C$-points of $\Gr_r(D_{\dR})$ the Grassmannian variety of codimension $r$ subspaces of the $E$-vector space $D_{\dR}\coloneqq (D\otimes_{\breve{E}}C)^{\sG_E}$ where $\sG_E\coloneqq \Gal(\bar E/E)$.
\end{rema}

\subsection{Isotrivial deformations and minuscule modifications}
We now explain how trivial (\ie semistable of slope $0$) minuscule modifications are related to isotrivial deformations of $\varpi$-divisible $\rmO$-modules. The classical references, usually consider $p$-divisible groups, and we slightly adapt them to include the $\rmO$-action.  We first define what this means:
\medskip
\begin{defi}
Let $H$ be a $\varpi$-divisible $\rmO$-module over $k$ (\cf Definition \ref{def:pidiv}). An \emph{isotrivial deformation} of $H$ over $\rmO_C$ is a pair $(G,\rho)$ where $G$ is a $\varpi$-divisible $\rmO$-module over $\rmO_C$ and $\rho$ is a $\rmO$-linear quasi-isogeny 
$$
\rho \colon H\otimes_k\rmO_C/p\dashrightarrow G\otimes_{\rmO_C}\rmO_C/p.
$$
Moreover, given a second isotrivial deformation  $(G',\rho')$ of $H$ over $\rmO_C$, we say that $(G,\rho)$ and $(G',\rho')$ are \emph{isogenous} if $\rho'\circ\rho^{-1}$ lifts to a quasi-isogeny $G\dashrightarrow G'$.
\end{defi}
\medskip
To a $\varpi$-divisible $\rmO$-module $H$ over $k$, we can associated (\cf Remark \ref{rem:pdivtoisoc}) an isocrystal $D$ characterizing the isogeny class of $H$ to which, using theorem \ref{thm:clasvb}, we can then associate a vector bundle $\CE(H)\coloneqq \CE(D(1))$ on $\CX$ which also characterizes the quasi-isogeny class of $H$. In this subsection, we will explain the following theorem due to Scholze Weinstein \cite{scwe} (\cf also \cite[Théorème 1.11]{rio}):
\medskip
\begin{theo}\label{thm:scwe}
Let $H$ be a $\varpi$-divisible $\rmO$-module over $k$. We have a functorial bijection
$$
\left \{ \begin{array}c
\text{ Isogeny classes of isotrivial }\\
 \text{deformations $(G,\rho)$ of $H$ over $\rmO_C$.}
 \end{array} \right \}
\cong {} 
\left \{ \begin{array}c
\text{Isomorphism classes of trivial minuscule}\\
\text{$\rmO$-linear modifications of $\CE(H)$}
\end{array} \right \}
$$
\end{theo}
\medskip

\subsubsection{Scholze--Weinstein's Dieudonné theory for $\varpi$-divisible $\rmO$-modules}
Let $G$ be a $\varpi$-divisible $\rmO$-module over $\rmO_C$  To $G$ we associate its \emph{universal cover} and its \emph{Tate module}
$$
\wt G\coloneqq \varprojlim_{[\varpi]}G,\quad T_{\varpi}(G)\coloneqq \varprojlim_nG[\varpi^n](\rmO_C).
$$
We also set $V_{\varpi}(G)\coloneqq T_{\varpi}(G)\otimes_{\rmO}E$ its \emph{rational Tate module} which is an $E$-vector space. By \cite[Proposition 3.1.3]{scwe}, $\wt G$ defines a crystal on $\NCRIS_{\rmO}((\rmO_C/p)/\rmO)$ and only depends on $G_0\coloneqq G\otimes_{\rmO_C}\rmO_C/p$, the base change of $G$ to $\rmO_C/p$. Using the same argument as in the proof of Proposition \ref{prop:Ecart} to adapt the Dieudonné theory of \cite[Theorem A]{scwe} we get 
$$
\bM_{ \rmO}(G)^{\varphi=\varpi}=\Hom(E/\rmO,G_0)\bigl [ \tfrac{1}{p}\bigr ] = \wt G(\rmO_C/p)=\wt{G}(\rmO_C)
$$
where $\bM_{ \rmO}(G)\coloneqq \BD_{\rmO}(G)_{(A_{\cris}\rightarrow \rmO_C/p)} [ \tfrac{1}{p}\bigr ]$ is a projective $\rmB_e$-module. In other terms, let $H$ be a $\varpi$-divisible $\rmO$-module over $k$ and suppose that $(G,\rho)$ is an isotrivial deformation of $H$ over $\rmO_C$ then
$$
\wt{G}(\rmO_C)=(D_{\rmO}(H)\otimes_{\brvO} \Bcrisp)^{\varphi=\varpi}=\Gamma(\CX,\CE(H)).
$$
We recall the following lemma:
\medskip
\begin{lemm}
Let $H$ be a $\varpi$-divisible $\rmO$-module over $k$ and let $(G,\rho)$ and $(G',\rho')$ be isotrivial deformations of $H$ over $\rmO_C$ then $\rho'\circ \rho^{-1}$ induces an isomorphism $\wt{G}\cong \wt{G}'$.
\end{lemm}
\medskip
The following remark stresses the translation between $p$-divisible groups and $\varpi$-divisible $\rmO$-modules:
\medskip
\begin{rema}
 Let $G$ be a $\varpi$-divisible $\rmO$-module over $\rmO_C$. The definitions of $\wt G_0$ and $T_{\varpi}(G)$ we give here are the same as those for $p$-divisible groups in \cite{scwe}. Indeed, it is clear that 
$$
 \wt{G}=\varprojlim_{[p]}G,\quad T_{\varpi}(G)=\varprojlim_n G[p^n](\rmO_C).
$$
Moreover, let 
$$
(D_{\Zp}(H)\otimes_{W_{\Zp}(k)} \Bcrisp)^{\varphi_p=p}\cong (D_{\rmO}(H)\otimes_{\brvO} \Bcrisp)^{\varphi_E=\varpi},
$$
where $\varphi_p$ on the left is given by the Witt vector Frobenius of $W_{\BZ_p}$. The strategy for proving this is the same as for Proposition \ref{prop:Ecart}. If $O\rightarrow O'$ is totally ramified, then it is obvious. If $O\rightarrow O'$ is unramified, then we use the map
$$
D_{\rmO}(G)_0\otimes_{\brvO'}\Bcrisp =D_{\rmO'}(G)\otimes_{\brvO'}\Bcrisp \rightarrow D_{\rmO}(G)\otimes_{\brvO}\Bcrisp=\bigoplus_{i\in \BZ/f}D_{\rmO}(G)_i\otimes_{\brvO'}\Bcrisp,
$$
given by $x\otimes y\mapsto (V^{j}x,\varphi^{-j}(y))_{j\in \BZ/f}$.

\end{rema}
\medskip

Let $n\in \BN$ and $G_n\coloneqq G \otimes_{\rmO_C}\rmO_C/p^n$ be the base change to $\rmO_C/p^n$. Following Scholze--Weinstein, we define a map (\cf \cite[Lemma 3.2.1]{scwe}) $s_G\colon \wt G_n\rightarrow E_{\rmO}(G_n)$ given on points by the following: to $x\in \wt G_n$ viewed as a sequence $x_m\in G_n$ we associate a lift $y_m\in E_{\rmO}(G_n)$ and then set 
$$
s_G(x)=\lim_m \varpi^my_m.
$$
This map defines a map of crystals which combined with the logarithm gives a map $\wt G_0\rightarrow \bM_{ \rmO}(G_0)$. Recall that the logarithm (\cf \cite[\S B.5.2]{far}) defines a map $G\rightarrow \Lie(G)\otimes_{\rmO_C}C$. Moreover, Messing's theory gives a map
$$
h_{G}\colon\bM_{ \rmO}(G_0)\rightarrow \Lie(G)\otimes_{\rmO_C}C.
$$
These maps are compatible in the following way, expressing a compatibility between the Hodge filtration and the universal cover:
\medskip
\begin{prop}\label{prop:compmodfil}
The following diagram commutes:
\begin{center}
\begin{tikzcd}
\wt G \ar[r]\ar[d]&\bM_{ \rmO}(G)\ar[d,"h_G"]\\
G\ar[r, "\log_G"] & \Lie(G)\otimes_{\rmO_C}C
\end{tikzcd}
\end{center}
\end{prop}
\medskip
\medskip
\begin{rema}
Note that the logarithm in this section depends on the chosen $\rmO$-PD structure (\cf \cite[Remarque B.5.12]{far})
\end{rema}
\medskip
\subsubsection{From $\varpi$-divisible $\rmO$-modules over $\rmO_C$ to minuscule modifications}
 The three objects $V_{\varpi}(G)$, $\wt G_0$ and $\Lie\, G$ are related in the following way (\cf \cite[Proposition 5.1.6]{scwe}):

\medskip
\begin{prop}\label{eq:exasqpdiv}
The following sequence is exact:
$$
0\rightarrow V_{\varpi}(G)\otimes_{E}\CO\rightarrow \wt G(\rmO_C)\rightarrow \Lie\, G\otimes_{\rmO_C}C\rightarrow 0.
$$
\end{prop}
\medskip

\medskip
\begin{coro}\label{cor:sw}
Let $G$ be a $\varpi$-divisible $\rmO$-module over $\rmO_C$ . The exact sequence \ref{eq:exasqpdiv} extends to an exact sequence of quasi-coherent sheaves on $\CX$
$$
0\rightarrow V_G\otimes_{E}\CO\rightarrow \CE(G)\rightarrow \iota_{\infty *}\Lie(G)\rightarrow 0
$$
which defines a trivial minuscule modification of $\CE(G)$.
\end{coro}
\medskip

\subsubsection{From minuscule modifications to $\varpi$-divisible $\rmO$-modules}
This characterization depends on the classification of $\varpi$-divisible $\rmO$-modules by Fargues and Scholze--Weinstein(\cf \cite[Theorem 5.2.1]{scwe}):
\medskip
\begin{theo}\label{thm:farclas}
There is an equivalence of categories 
$$
\left \{ \begin{array}c
\text{$\varpi$-divisible $\rmO$-modules over $\rmO_C$}
\end{array} \right \}
\cong {} 
\left \{ \begin{array}c
\text{Pairs $(T,W)$, where $T$ is a free $\rmO$-module of finite rank,}\\
\text{$W\incl T\otimes_{\rmO}C(-1)$ defining a $C$-vector subspace}
\end{array} \right \}
$$
\end{theo}
\medskip
Let $G$ be a $\varpi$-divisible $\rmO$-module over $\rmO_C$, then the associated pair $(T,W)$ is $(T_{\varpi}(G),\Lie\, G\otimes_{\rmO_C} C)$ and the map is given by the dual of the \emph{Hodge-Tate map}, which we now recall. The Hodge-Tate map of $G^D$:
$$
\alpha_{G^{D}}^{\vee}\colon W\rightarrow T\otimes_{\rmO}C(-1),
$$
is the dual of the inverse limit of $\alpha_{G^D[p^n]}\colon G^D[p^n]\rightarrow \omega_{G[p^n]}$ defined sending $f\colon G[p^n]\rightarrow \BG_m$ to $\alpha_{G^D[\varpi^n]}(f)\coloneqq f^*\tfrac{dx}{x}$. 

Using this theorem, Scholze--Weinstein prove the following (\cf \cite[Theorem 6.2.1]{scwe},  \cite[Proposition 5.1.6]{scwe}):
\medskip
\begin{prop}\label{prop:modeqpdiv}
Let $H$ be a $\varpi$-divisible $\rmO$-module over $k$. Given a minuscule modification of $\CE(H)$,
$$
0\rightarrow \CF \rightarrow \CE(H) \rightarrow \iota_{\infty *}W\rightarrow 0
$$
Then $\CF$ is semi-stable of slope $0$ if and only if there exists $(G,\rho)$, an isotrivial deformation of $H$ over $\rmO_C$ such that $\CF=V_{\varpi}(G)\otimes_{E}\CO$, $W=\Lie\, G\otimes_{\rmO_C}C$ and the modification is given by $\rho$. 
\end{prop}
\medskip

We still need to treat isogenies of isotrivial deformations, which is given by the following:
\medskip
\begin{coro}
Let $(G,\rho)$ and $(G',\rho')$ be two isotrivial deformations over $\rmO_C$ of a $\varpi$-divisible $\rmO$-module $H$ over $k$. Then the quasi-isogeny $\rho'\circ\rho^{-1}\colon G_0\dashrightarrow G_0'$ lifts to a quasi-isogeny $G\dashrightarrow G'$ if and only if the associated modifications are isomorphic.

\end{coro}
\medskip

\subsection{Modifications of $\CE_{\Dr}$}
In this subsection we compute minuscule modifications of $\CE_{\Dr}$, the proposition is the following:
\medskip
\begin{prop}\label{prop:compodmod}
Let $\CE$ be a $\rmD$-linear modification of $\CE_{\Dr}$ of degree $d$. Then, either $\CE$ is trivial, \ie $\CE\cong \CO\otimes_E \rmD$, else there exists $r\in \llbracket 1,d-1\rrbracket$ such that
$$
\CE\cong \CO\bigl ( \tfrac 1 d \bigr )^{\oplus r} \oplus \CO\bigl (- \tfrac {r} {d(d-r)} \bigr )^{\oplus m},
$$
where $m=\lcm(d,d-r)$.
\end{prop}
\medskip
\subsubsection{Minuscule $\rmD$-linear modifications}
Let us first define what $\rmD$-bundles and $\rmD$-modification mean:
\medskip
\begin{defi}
A \emph{$\rmD$-bundle} on $\CX$ is a vector bundle $\CE$ on $\CX$ of rank $d^2$ endowed with a faithful action of the algebra $\rmD$ (\ie a locally free right $\rmD$-module of rank $1$ in the category of vector bundles on $\CX$). We moreover say that a map $f\colon \CE'\rightarrow \CE$ between $\rmD$-bundles is \emph{$\rmD$-linear} if $f$ commutes with the action of $\rmD$. In particular, we say that $\CE'$ is a $\rmD$-linear modification of $\CE$ if it is given by a $\rmD$-linear map. 
\end{defi}
\medskip
Notice that if $\CE'$ is a $\rmD$-linear modification of $\CE$ then the cokernel of the map $\CE'\incl \CE$ is a coherent sheaf concentrated at $\infty$ endowed with an action of $\rmD^{\times}$. Moreover, if it is minuscule, then its degree is divisible by $d$. We could also express this definition in terms of $\rmD^{\times}$-torsors on $\CX$. For example $\CE_{\Dr}$ and $\CO\otimes_E\rmD$ are $\rmD$-bundles on $\CX$. 

We prove the following lemma:
\medskip
\begin{lemm}\label{lem:odmodif}
We have a natural bijection
$$
\left \{ 
\text{$\rmD$-linear minuscule modifications of $\CE_{\Dr}$}
\right \}
\cong {} 
\BP^{d-1}(C)
$$
\end{lemm}
\medskip
\begin{proof}
Let $\CE=\CE_{\Dr}$, $\wh{\CE}_{\infty}/t\wh{\CE}_{\infty}\cong \bM \otimes_{\brv E}C$ is endowed with its natural $\rmD$ action. Adding the action of $\rmD$ to Proposition \ref{prop:minmod}, we see that $\rmD$ stable points of $\Gr_d(\wh{\CE}_{\infty}/t\wh{\CE}_{\infty})$ correspond to $\rmD$-linear minuscule modifications of $\CE$. Using Lemma \ref{lem:trivtoslop} we get that $\rmD$-stable points of $\Gr_d(\bM \otimes_{\brv E}C)$ are identified with
$$
\BP(\boeta_0\otimes_{\rmO}C)=\BP^{d-1}(C).
$$
\end{proof}
\begin{rema}
There is a more \og representation-theoretic\fg way to see this lemma. Notice that $\bM \otimes_{\brv E}C\cong C^d\otimes_C(C^d)^*$ where $\rmD^{\times}\incl \GL_d(C)$ acts as the standard representation on the left factor and trivially on the right factor. Hence, we are looking for $\GL_d(C)$-equivariant subspaces of $C^d\otimes_C(C^d)^*$ of codimension $d$ which corresponds to $1$-dimensional subspaces of $(C^d)^*$: this is just $\BP^{d-1}(C)$.
\end{rema}
\subsubsection{The two tower principle}
The so called "two tower principle" usually designates the isomorphism between the Lubin--Tate tower and the Drinfeld tower. Through the Fargues--Fontaines curve this can naturally be described as an equivalence between $\rmD$-bundles and rank $d$-bundles on $\CX$. We prove the following:
\medskip
\begin{prop}\label{prop:twin}
There is an anti-equivalence of categories
$$
\left \{ 
\text{$\rmD$-bundles on $\CX$}
\right \}
\xrightleftharpoons[\CHom\bigl (\  \cdot\  ,\CO\bigl (\tfrac 1 d \bigr)\bigr )]{\CHom_{\rmD}\bigl(\  \cdot\  ,\CO\bigl (\tfrac 1 d \bigr)\bigr)}
\left \{ 
\text{Rank $d$ vector bundles on $\CX$}
\right \}
$$
given by $\CE\mapsto \CHom_{\rmD}(\CE,\CO\bigl(\tfrac 1 d \bigr))$ and whose inverse is given by $\CF\mapsto\CHom(\CF,\CO\bigl (\tfrac 1 d \bigr))$. Moreover, this anti-equivalence is compatible with modifications, sending $\rmD$-linear minuscule modification of degree $dr$ to minuscule modifications of degree $r$.
\end{prop}
\medskip

Recall that the action of $\rmD$ on $\CO\bigl ( \tfrac 1 d \bigr )$ is given by the Frobenius (\cf Example \ref{exa:dactond}). This proposition can be seen as a Morita equivalence (\cf \cite[Chapter II, Theorem 3.5]{bas} where it is stated for modules\footnote{The translation for vector bundles can be either deduced from the proofs or from Berger pairs. With more care, it can also be deduced from isocrystals.}.

\begin{proof}
To avoid ugly displays, let $\lambda=\tfrac 1 d$. To show that the functors are Morita equivalences between the categories of vector bundles and vector bundles endowed with a (faithful) action of $\rmD$, we check the conditions of \cite[Chapter II, Definition 3.2]{bas}:
$$
\CO(\lambda)\otimes_{\CO}\CO(-\lambda)=\CEnd_{\CO}\CO(\lambda)= \rmD\otimes_E \CO
$$
hence we get :
$$
\CO(\lambda)\otimes_{\rmD\otimes_E\CO}\CO(-\lambda)=\CO.
$$
Restricted to $U=\CX\setminus \{\infty\}$, we get that this Morita equivalence sends modifications to modifications and considering the action by $t$, we get that it sends minuscule modifications to minuscule modifications. It only remains to show the claims on rank and degree, and, since we already have a Morita equivalence, it suffices to check that $\CF\mapsto\CHom(\CF,\CO(\lambda))$ sends rank $d$ vector bundles to rank $d^2$-vector bundles and degree $r$ modifications to degree $dr$ modifications.

First, let $\CF$ be a rank $d$ vector bundle, then by the multiplicativity of the rank for the tensor product we get
$$
\rank \CHom(\CF,\CO(\lambda))=\rank(\CF^{\vee}\otimes_{\CO}\CO(\lambda))=d^2.
$$

For the claim on the degree, we need to check $\CExt^1(\iota_{\infty,*}C,\CO(\lambda))=\iota_{\infty,*}C^d$. Let 
$$
r\coloneqq \dim_C\Ext^1(\iota_{\infty,*}C,\CO(\lambda)).
$$
Since this Ext-sheaf is supported at $\infty\in \CX$, it is enough to check that $r=d$. We apply $\Hom(\ \cdot \ , \CO(\lambda))$ to the fundamental exact sequence 
$$
0\rightarrow \CO\rightarrow \CO(1)\rightarrow \iota_{\infty,*}C\rightarrow 0.
$$
Using Proposition \ref{prop:expsec}, and the fact that there are no nonzero maps $\iota_{\infty,*}C\rightarrow \CO(\lambda)$, we get the exact sequence
$$
0\rightarrow H^0(\CX,\CO(\lambda))\rightarrow \Ext^1(\iota_{\infty,*}C,\CO(\lambda))\rightarrow H^1(\CX,\CO(\lambda-1))\rightarrow 0
$$
which lifts to an exact sequence of Banach--Colmez spaces: 
$$
0\rightarrow (\Bcrisp)^{\varphi^d=\varpi}\rightarrow C^r \rightarrow \frac{\Bdrp}{t^{d-1}\Bdrp+E_d}\rightarrow 0.
$$
Computing $E$-Dimensions using Example \ref{exa:bccomp} we get $(1,d)$ for the left term and $(d-1,d)$ for the right one: by Theorem \ref{thm:bcspace}, we get $r=d$.
\end{proof}

\medskip
\begin{exem}
Let $\lambda=\tfrac 1 d$, since $\CE_{\Dr}=\CO(\lambda)^{\oplus d}$ the proof of the proposition gives $\CHom_{\rmD}(\CE_{\Dr},\CO(\lambda))=\CO^{\oplus d}$. Hence a $\rmD$-linear modification 
$$
0\rightarrow \CE \rightarrow \CE_{\Dr}\rightarrow \iota_{\infty,*}C^d\rightarrow 0
$$
is sent to the modification
$$
0\rightarrow \CO^d\rightarrow \CF\rightarrow \iota_{\infty,*}C\rightarrow 0,
$$
where $\CF\coloneqq \CHom_{\rmD}(\CE,\CO(\lambda))$. This gives the next corollary.
\end{exem}
\medskip

\medskip
\begin{coro}\label{cor:modmod}
The equivalence of Proposition \ref{prop:twin} induces a natural bijection
$$
\left \{ \begin{array}{cc}
\text{$\rmD$-linear modifications of $\CE_{\Dr}$}\\
\text{of degree $d$}
\end{array}
\right \}
\cong {} 
\left \{ \begin{array}{cc}
\text{Rank $d$ vector bundles $\CF$ equipped with}\\
\text{a trivial modification $\CO^d\incl \CF$}\\
\text{of degree $1$}
\end{array}
\right \}.
$$
\end{coro}
\medskip
\subsubsection{Trivial minuscule modifications of degree $1$}
By Corollary \ref{cor:modmod} we reduced the proof of Proposition \ref{prop:compodmod} to the proof the following (\cf \cite[Theorème 5.6.29]{fafo}, but we give a different proof):
\medskip
\begin{prop}\label{prop:ltmod1}
Let $\CF$ be a rank $d$ vector bundle on $\CX$ admitting a trivial modification of degree $1$, then there exists $r\in \llbracket 0, d-1 \rrbracket$ such that
$$
\CF\cong \CO^r\oplus \CO\bigl (\tfrac 1 {d-r} \bigr ).
$$
\end{prop}
\medskip
Indeed, computing $\CHom(\CF,\CO\bigl (\tfrac 1 {d} \bigr ))$ we get back Proposition \ref{prop:compodmod}. 

We first give a shorter proof, using the classification of vector bundles on the Fargues--Fontaine curve.
\begin{proof}
Let 
$$
0\rightarrow \CO^d\rightarrow \CF \rightarrow \iota_{\infty,*}C\rightarrow 0
$$
be a modification. By Theorem \ref{thm:clasvb} we can write $\CF=\bigoplus_{i\in I} \CO(\lambda_i)$ for $\lambda_i\in \BQ$.

Let us show that for all $i\in I$, $\lambda_i\geqslant 0$. Suppose otherwise, let $j\in I$ be such that $\lambda_j<0$. Since $\Hom(\CO,\CO(\lambda_j))=0$ we get that $\CO^d\incl \bigoplus_{i\neq j}\CO(\lambda_i)$, so $\CO(\lambda_j)$ is contained in the cokernel. This is a contradiction since the cokernel is supported at $\infty\in \CX$.

Now since $\deg \CF=1$ by additivity of the degree function, and since $\deg \CO(\lambda_i)\geqslant 1$ if and only if $\lambda_i\geqslant 0$, we deduce that there exists at most one $i\in I$ such that $\lambda_i>0$ and that $\lambda_i=\tfrac 1 {r'} $ for $r'\in \BN$. Since $\rank \CO(\lambda_i)=r'$ we get $r'\leqslant d$ and finally
$$
\CF\cong \CO^r\oplus \CO\bigl (\tfrac 1 {d-r} \bigr ).
$$
\end{proof}

To show different techniques, we give a different proof. Using Theorem \ref{thm:scwe} we reduce the proposition to the following:
\medskip
\begin{prop}\label{prop:ltmod2}
Let $K$ be either $C$ or a finite extension of $\breve{E}$. Let $G$ be a $\varpi$-divisible $\rmO$-module over $\rmO_K$ of height $d$ and dimension $1$. Then there exists $r\in \llbracket 0,d-1\rrbracket$ such that
$$
D_{\rmO}(G)=D_0^r\oplus D_{1/(d-r)}.
$$
\end{prop}
\medskip

\begin{proof}
Let us first reduce this claim to the case where $K$ is a finite extension of $\breve{E}$. Let $H=G\otimes_{\rmO_K}k$ be the associated special fiber, which is a $\varpi$-divisible $\rmO$-module over $k$. Then, by the representability theorem of Rapoport--Zink (\cf \cite[Theorem 3.25]{razi}), we get that the functor on $\Nilp_{\brv{\rmO}}$ defined by
$$
\brv{\CM}_H\colon R\mapsto \left\{(G,\rho) \mid \begin{array}{l} \text{ $G$ is a $\varpi$-divisible $\rmO$-module over $R$ of dimension $1$, } \\
					\rho \colon H \otimes_{k}  R/p \dashrightarrow G \otimes _R R/p \text{ an $\rmO$-linear quasi-isogeny }
					\end{array} \right \},
$$
is represented by a formal scheme over $\breve{\rmO}$ formally locally of finite type. Hence $\brv{\CM}_H(\rmO_C)\neq 0$ if and only if $\brv{\CM}_H(\rmO_K)\neq 0$ for some finite extension $K/\breve{E}$. This means that Proposition \ref{prop:ltmod2} is true for $K=C$ if and only if it is true for all finite extensions $K$ of $\breve{E}$. 

Let $K$ be a finite extension of $\breve{E}$ and let $G$ be a $\varpi$-divisible $\rmO$-module over $\rmO_K$. Let $D=D_{\rmO}(G)$ and let $0\leqslant \lambda_1\leqslant \dots \leqslant \lambda_s\leqslant 1$ be its slopes. By the classification of $\varpi$-divisible $\rmO$-modules over $\rmO_K$ we know that the Newton polygon of $G$ is above the Hodge polygon of $G$, given by $(1,\underbrace{0,\dots, 0}_{d-1})$, and that they have the same end point. Since the Newton polygon of $G$ has integral break points, it takes the following form:
\FloatBarrier
\centering

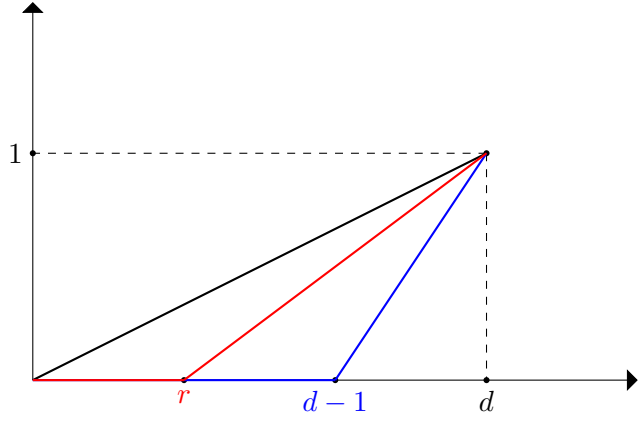
\begin{figure}[h]
  \centering 
  \begin{tikzpicture}[scale=2]
    \draw[axearrow] (0,0) -- (4,0) node[right]{}; 
    \draw[axearrow] (0,0) -- (0,2.5) node[above]{}; 
    \draw (3,1.5) node[above right,coo] {}; 
    \fill (3,1.5) circle (0.02); 
    \fill (1,0) circle (0.02); 
    \fill (2,0) circle (0.02); 
    \fill (3,0) circle (0.02); 
    \fill (0,1.5) circle (0.02); 
    \draw[dashed] (3,1.5) -- (3,0) node[below] {$d$}; 
    \draw[dashed] (3,1.5) -- (0,1.5) node[left] {$1$}; 
    \draw[thick,blue] (3,1.5) -- (2,0) node[below] {$d-1$}; 
    \draw[thick,black] (0,0) -- (3,1.5) node[below] {}; 
    \draw[thick,blue] (0,0) -- (2,0) node[below] {}; 
    \draw[thick,red] (3,1.5) -- (1,0) node[below] {$r$}; 
    \draw[thick,red] (0,0) -- (1,0) node[below] {}; 
  \end{tikzpicture}
  \caption{In red, the Newton polygon of $G$, in blue the Hodge polygon of $G$.}\label{fig:newhdg}
\end{figure}
\FloatBarrier
Hence, there exists an integer $r\in \llbracket 0,d-1\rrbracket$ such that the slopes of $D$ are $0$ (with multiplicity $r$) and $\tfrac{1}{d-r}$ (with multiplicity $d-r$), which proves the proposition.
\end{proof}
\begin{rema}
Note that the condition on the polygons in Figure \ref{fig:newhdg} can be expressed as the weak admissibility condition for the filtered $\varpi$-module $(D_{\cris}(V_{\varpi}(G)),D_{\dR}(V_{\varpi}(G)))$. 

\end{rema}

\subsection{The period morphism}
Recall that we write $\rmM_{\brv E}$ the generic fiber (in the sense of Berthelot) of $\CM_{\Dr}$ which is a rigid space over $\breve{E}$. We refer to \cite[Chapter 5]{razi}, and more specifically to \cite[5.47]{razi} for the definition of the period morphism in the Drinfeld case, which we now recall.

\subsubsection{Definition}\label{subsubsec:defpermap}
We define the Grothendieck--Messing period map $\pi\colon \rmM_{\brv E}\rightarrow \BP^{d-1}_{\breve{E}}$ where $\BP^{d-1}_{\breve{E}}$ is the projective space of dimension $d-1$ seen as a rigid space over $\breve{E}$. Let $A$ be an affinoid algebra over $\brv E$ and $A^+\subset A$ the ring of definition given by power bounded elements; $A^+$ is a $p$-adically complete $\brv{\rmO}$-algebra and $A\coloneqq A^+\otimes_{\rmO}E$, for $x=(X,\rho)\in \rmM_{E}(A)=\CM_{\Dr}(A^+)$, we write $D\coloneqq \BD_{\rmO}(X)_{A^+}$ for the Dieudonné crystal of $X$ evaluated at $A^+$ (\cf Remark \ref{rema:bbm}) and $D_0\subset D$ its $0$-th graded piece. Recall, by the definition of the crystal and \S\ref{subsubsec:godsfd}
$$
\BD_{\rmO}(X\otimes_{A^+}A^+/p)_{(A^+\rightarrow A^+/p)}\cong D=\bigoplus_{i\in \BZ/d}D_i
$$
and that the map induced by the Hodge filtration defines a graded map $h_X\colon D\rightarrow \Lie(X)$ (\cf \ref{eq:mesisgrad}). Moreover, the quasi-isogeny $\rho \colon \bX\otimes_{k}A^+/p\rightarrow X\otimes_{A^+}A^+/p$ induces a graded isomorphism $\bM\otimes_{\brvO}A\cong D\otimes_{A^+}A$ which gives
$$
\bM_0\otimes_{\rmO}A\cong D_0\bigl [ \tfrac 1 p \bigr ].
$$
Then $\pi$ is defined by
$$
\pi(x)\coloneqq \left [A\otimes_E E^d\rightarrow D_0\bigl [\tfrac 1 p \bigr]\xrightarrow{h_{X,0}} \Lie\, (X)_0\bigl [\tfrac 1 p \bigr]\right ]\in \BP^{d-1}_{\brv E}(A)
$$
where we used the identification $E^d\cong \bM_0^{\bU=1}\otimes_{\rmO}E$ from \S \ref{subsubsec:refob}. 
\medskip
\begin{rema}
Here is a slightly more conceptual way to understand the period map $\pi \colon \rmM_{\brv E}\rightarrow \BP^{d-1}_{\brv E}$:
 
 As described in \cite[5.16]{razi} (\cf also \cite[I.2.1]{far} in the Lubin--Tate case, which is similar, but in terms of convergent crystals), we can consider $(X^{\univ}, \rho^{\univ})$ the universal object over $\CM_{\Dr}$ and consider the Dieudonné crystal $\BD_{\rmO}(X^{\univ})$ of $X^{\univ}$ and evaluate it at $\CO_{\Dr}^+$, the structure sheaf of $\CM_{\Dr}$ before inverting $p$:
$$
D^{\univ,+}\coloneqq \BD_{\rmO}(X^{\univ})_{(\CO_{\Dr}^+\rightarrow\CO_{\Dr}^+/p)},\quad D^{\univ}\coloneqq D^{\univ,+}\bigl[ \tfrac 1 p \bigr ].
$$
This defines a $\CO_{\Dr}$-module and Rapoport--Zink show (\cf \cite[Proposition 5.15]{razi}) that $\rho^{\univ}$ induces an isomorphism of $\CO_{\Dr}$-modules commuting with the action of $\rmOD$:
\begin{equation}\label{eq:isot}
\bM\otimes_{\brvO}\CO_{\Dr}\cong D^{\univ}
\end{equation}
Moreover, Grothendieck--Messing theory gives a surjective morphism 
$$
h_{X^{\univ}}\colon D^{\univ,+}\rightarrow \Lie(X^{\univ})
$$
whose composition with \eqref{eq:isot} gives a surjective morphism
$$
\bM\otimes_{\brvO}\CO_{\Dr}\rightarrow \Lie(X^{\univ})\bigl[ \tfrac 1 p \bigr ].
$$
Restricting this map to the $0$-th graded piece, this defines the Grothendieck--Messing period map $\pi\colon \rmM_{\brv E}\rightarrow \BP^{d-1}_{\breve{E}}$.
\end{rema}
\medskip
\subsubsection{Injectivity on points}

By Proposition \ref{prop:compmodfil}, if $x=(X,\rho)\in \rmM_{\brv E}(C)$, then $\pi(x)\in \BP^{d-1}(C)$ corresponds by Lemma \ref{lem:odmodif} with the modification of $\CE_{\Dr}$ defined by $x$. Hence on $C$-points, the period map $\pi \colon \rmM_{\brv E}(C)\rightarrow \BP^{d-1}(C)$ is given by combining theorem \ref{thm:scwe} (to which we add an action of $\rmOD$) and lemma \ref{lem:odmodif}.
\medskip
\begin{prop}
For a valued field extension $K/\brv E$, the map $\pi\colon \rmM_{\brv E}(K)\rightarrow \BP^{d-1}(K)$ is injective. 
\end{prop}
\medskip
\begin{proof}
It is enough to prove the proposition for $K=C$ an algebraically closed complete field extension of $\brv E$. Let $x_1=(X_1,\rho_1),\ x_2=(X_2,\rho_2)\in \rmM_{\brv E}(C)$ be two points such that $\pi(x_1)=\pi(x_2)$, meaning that these two points define isomorphic modifications of $\CE_{\Dr}\coloneqq \CE(\bX)$. Hence, by Theorem \ref{thm:scwe}, we get an $\rmOD$-linear quasi-isogeny $f\colon X_1\rightarrow X_2$ of height zero. Now by Theorem \ref{thm:farclas}, if we write $(T_i,W_i)\coloneqq (T_{\varpi}(X_i),\Lie\,  X_i\otimes_{\rmO_C}C)$ for $i=1,2$, then $f$ induces a $\rmOD$-linear isomorphism
$$
f\colon T_1 \bigl [ \tfrac 1 p \bigr ]\xrightarrow{\sim} T_2 \bigl [ \tfrac 1 p \bigr ],
$$
matching $W_1$ and $W_2$ when tensoring with $C(-1)$, and we only need to check that this isomorphism induces an isomorphism between $T_1$ and $T_2$ by the fully faithfulness. But recall that 
$$
\End_{\rmOD}(T_1\bigl [ \tfrac 1 p \bigr ])= E
$$
and since $f$ is of height zero it is given by an invertible element $f\in \rmO^{\times}$. This finishes the proof.
\end{proof}
\subsubsection{Surjectivity on points}
The following lemma gives the weakly admissible condition for a $\rmD$-stable two step filtration on $\CF\subset \bM\otimes_{\breve E}K$. This describes the image of $\pi$ on classical points:
\medskip
\begin{lemm}\label{lem:wadrin}
Let $K/\brv E$ be a valued field extension. Let $\CF\subset \bM\otimes_{\brv E}K$ be a $K$-subspace of codimension $d$ which is stable by $\rmD$. The filtered $\varphi$-module defined by $\bM$ with the two step filtration $\Fil^{\bullet}$ defined by $\Fil^1=\CF$ is weakly admissible if and only if $\CF_0\subset \boeta_0\otimes_{\rmO}K$ does not contain an $E$-rational line.

\end{lemm}
\medskip
\begin{proof}
The proof is given in \cite[Proposition 1.45]{razi}, we shortly recall it. Let $\CL_0\subset \boeta_0\otimes_{\rmO}E$ be a $1$-dimensional $E$-vector space. We will show that $\CF$ defines an admissible filtration if and only if 
$$
\dim_K(\CF_0\cap (\CL_0\otimes_EK))\leqslant \frac{d-1}{d},
$$
which implies the claim. Let $\CL\subset \bM$ be the $d$-dimensional isocrystal associated with $\CL_0$ by lemma \ref{lem:trivtoslop} which is of slope $\tfrac{d-1}{d}$. Moreover:
$$
\dim_K(\CF\cap (\CL\otimes_{\brv E}K))=d \dim_K(\CF_0\cap (\CL_0\otimes_EK))
$$
Hence, writing down the weakly-admissible condition (\og the Newton point is above the Hodge point \fg) for the filtered subisocrystal given by $\CL\subset \bM$ we get
$$
d \dim_K(\CF_0\cap (\CL_0\otimes_EK))=\dim_K(\CF\cap (\CL\otimes_{\brv E}K))\leqslant \dim_{\brv E}(\CL)\cdot \frac{d-1}{d}=d-1.
$$ 
This proves the claim.
\end{proof}

\medskip
\begin{prop}\label{prop:cimadr}
Let $x\in \BP^{d-1}(C)$ defining a modification $\CE_x$ of $\CE_{\Dr}$. Then the vector bundle $\CE_x$ is trivial if and only if $x\in \BH_{\brv E}^d(C)$, meaning that the hyperplane defined by $x$ does not contain any $E$-line.
\end{prop}
\medskip
\begin{proof}
By Proposition \ref{prop:compodmod}, the vector bundle $\CE_x$ is not trivial if and only if it is of the form $\CE_x\cong \CO(\tfrac{1}{d})\oplus \CE'$. In that case we get a $\rmD$-linear injective map $\CO(\tfrac{1}{d})\incl\CE_{\Dr}$ which is in the kernel of map $\CE_{\Dr}\rightarrow \iota_{\infty,*}C^d$ defining the modification. Moreover, this inclusion defines a subisocrystal $D'\subset \bM$ which is stable by $\rmD$ and by the previous observation, if we write $\CF\subset \bM\otimes_{\brv{E}}C$ the subspace defined by $x$, we get $D'\subset \CF$. But then, using the same construction as in \ref{lem:wadrin} we get an $E$-line $(D_0')^{\bU=1}\subset  \boeta_0\otimes_{\rmO}E$ such that $(D_0')^{\bU=1}\subset \CF_0$. Hence we have proved that $\CE_x$ is not trivial if and only if $x\in \BP^{d-1}(C)$ contains an $E$-line which is equivalent to
$$
x\in \bigcup_{H\in\CH_E}H(C).
$$
\end{proof}
\medskip
\begin{rema}
Notice that the former proposition is a \og weakly admissible $\implies$ admissible\fg theorem for the Drinfeld space and it can be deduced from the general theorem of Chen--Fargues--Shen (\cf \cite{cfs}). The proof is in fact the simplified version of their argument, explained in \cite[Remark 6.2]{cfs}, which is the case where modifications can be described by subisocrystals.
\end{rema}
\medskip
\begin{coro}
Let $K/\breve E$ be a valued field extension, then
$$
\pi(\rmM_{\brv E}(K))=\BH_{\brv E}^d(K).
$$
\end{coro}

\subsubsection{The period morphism is an open immersion}
We end the proof of the theorem by showing that the period morphism is an open immersion, which implies it is an isomorphism on its image. We start with a lemma, coming from\footnote{In the scheme case, the fact that an injective étale map is an open immersion is sometimes known as \emph{Grothendieck's theorem}. For a rigid space there is the additional feature that it is enough to ask injectivity on geometric points, since the local rings are Henselian.} \cite[Lemme I.42]{far}. 
\medskip
\begin{lemm}\label{lem:farlem}
Let $K$ be a complete non-archimedean valued field and $f\colon X\rightarrow Y$ an étale map between sheafy adic spaces (in particular adic spaces associated to rigid spaces) over $K$. If for every valued field extension $L/K$ the induced map $f(L)\colon X(L)\rightarrow Y(L)$ is injective then $f$ is an open immersion.
\end{lemm}
\medskip
\begin{proof}
For convenience, we detail the proof. Recall that an étale map between rigid spaces is open (since it is smooth, \cf \cite[Proposition 1.7.8]{hub}) and the hypothesis gives that $f$ induces an injection $\lvert f \rvert \colon \lvert X \rvert \incl \lvert Y \rvert$ on the underlying topological spaces that is a homeomorphism on its image. This means that $\lvert f \rvert \colon \lvert X \rvert \incl \lvert Y \rvert$ is an open inclusion and we only need to check that for $x\in X$, the induced map $g=f^*_x\colon \CO_{Y,f(x)}\rightarrow \CO_{X,x}$ is an isomorphism. Write $A\coloneqq \CO_{Y,f(x)}$ and $B\coloneqq \CO_{X,x}$ which are henselian $K$-algebras by \cite[Proposition 3.2.10]{fuka}. Then $g\colon A\rightarrow B$ is finite étale. The hypothesis tells us that for every valued field extension $L/K$ we get a bijection
$$
\Hom(B,L)\xrightarrow[\cdot \circ g]{\sim}\Hom(A,L)
$$
meaning in particular that $g$ induces an isomorphism between the residue fields $\kappa(g)\colon \kappa(A)\xrightarrow {\sim}\kappa(B)$. But since $A$ and $B$ are henselian this implies that $g$ is an isomorphism, which concludes the proof.
\end{proof}

As a consequence of the lemma we get the following corollary, which concludes the proof of Theorem \ref{thm:genef}.
\medskip
\begin{coro}
The period morphism $\pi\colon \rmM_{\brv E}\rightarrow \BP^{d-1}_{\brv E}$ is an open immersion whose image is $\BH_{\brv{E}}^d$.
\end{coro}

\newpage

\section{The (perfected) special fibers}\label{sec:perfspe}
Recall that $k$ is the residue field of $\brvO$. Let $\brv{\rmM}_k$ be the special fiber of $\brv{\CM}_{\Dr}$ and $\rmM_k$ be the special fiber of $\CM_{\Dr}$. Let $X$ be a scheme over $k$, let $F\colon X\rightarrow X$ be its absolute Frobenius. We define the \emph{perfection} of $X$ as
\begin{equation}\label{eq:defperf}
X^{\perf}\coloneqq \varprojlim_F X.
\end{equation}
Note that we then have $\lvert X^{\perf}\rvert=\lvert X\rvert$. Moreover, $X^{\perf}$ represents the functor $\Hom(\ \cdot \ ,X)$ restricted to the category $\Perf_k$ of perfect schemes over $k$. The perfection defines a functor from schemes over $k$ to perfect schemes over $k$: if$f\colon X_1\rightarrow X_2$ is a morphism, we get a morphism $f^{\perf}\colon X_1^{\perf}\rightarrow X_2^{\perf}$,  \emph{the perfection of $f$}.

We consider $\brv{\rmM}_k^{\perf}$ the perfection of $\brv{\rmM}_k$, and $\brv{\BH}_k^{d,\perf}$ the perfection of the special fiber of $\brv{\BH}^{d}_{\brvO}$.

We will prove the following
\medskip
\begin{theo}\label{thm:perfiso}
There exists a $\GL_d(E)$-equivariant isomorphism
$$
f\colon \brv{\rmM}_k^{\perf}\xrightarrow{\sim} \brv{\BH}_k^{d,\perf}.
$$
\end{theo}
\medskip

\subsection{The special fiber of $\brv{\BH}_{\rmO}^d$}\label{subsec:speh}
The purpose of this section is to show that the irreducible components of\footnote{The results of this paragraph also hold for the unperfected version over $k_E$.} $\brv{\BH}_{k}^{d,\perf}$ are compactified Deligne--Lusztig varieties.

\subsubsection{Deligne--Lusztig varieties and compactifications}
We start by defining the flag variety where, for later convenience, we introduce a shift in the indices and work over perfect algebras. If the indices confuse the reader, it is harmless to consider $i=0$. 

Let $\eta$ be an $\rmO$-lattice of $E^d$ and $i\in \BZ$ its logarithmic index. Let $\Fl_{\eta}^{\perf}$ be the perfect $k$-scheme representing the functor defined by
$$
R\in\Perf_k\mapsto \{(D_j)_{j\in \BZ}\}
$$
where $(D_j)_{j\in \BZ}$ is a set of $W_{\rmO}(R)$-submodules of $E^d\otimes_{\rmO}W_{\rmO}(R)$ such that: 
\begin{enumerate}
\item $\forall j\in \BZ \ D_j\subsetneq D_{j+1}$ and $D_{j+1}/D_j$ is a locally free $R$-module of rank $1$,
\item $D_i=\eta\otimes_{\rmO}W_{\rmO}(R)$,
\item $\forall j\in \BZ$, $D_j=\varpi D_{j+d}$.
\end{enumerate}

Considering the data modulo $\varpi \eta$, since $W_{\rmO}(R)$ is $\varpi$-complete and torsion-free because $R$ is perfect, we get that $\Fl_{\eta}^{\perf}$ is isomorphic to the perfection of the flag variety associated to $\eta/\varpi \eta$.  Hence by \cite[Proposition (9.9.3)]{ega1}, this functor is represented by a projective scheme over $k$. To avoid heavy indices, we will just write $(D_j)_j\in \Fl_{\eta}^{\perf}(R)$. Recall that $\GL_d(E)$ acts on $E^d$ through the standard action and let 
$$
\GL(\eta)\coloneqq \{g\in \GL_d(E)\mid g.\eta = \eta\}\subset \GL_d(E),
$$
the stabilizer of $\eta\subset E^d$. Then $\GL(\eta)$ naturally acts on $\Fl_{\eta}^{\perf}$ for $g\in \GL(\eta)$ by
\begin{equation}\label{eq:gactonflag}
g\colon (D_j)_j\mapsto (g\cdot D_j)_j.
\end{equation}
Notice that this action factors through the map $\GL(\eta)\rightarrow \GL(\bar \eta)$ induced by the reduction mod $\varpi$.

Recall that for $R\in \Perf_k$, $W_{\rmO}(R)$ is equipped with a Frobenius endomorphism $\sigma\colon W_{\rmO}(R)\rightarrow W_{\rmO}(R)$ which is moreover an automorphism. In particular, $\eta \otimes_{\rmO}W_{\rmO}(R)$ is equipped with a Frobenius automorphism which we still denote by $\sigma$.
\medskip
\begin{defi}\label{def:deluvar}
\
\begin{itemize}
\itemb Let $X_k^{\eta,\perf}$ be the subfunctor of $\Fl_{\eta}^{\perf}$ defined by
$$
R\in\Perf_k\mapsto X_k^{\eta,\perf}(R)\coloneqq \{(D_j)_{j}\in \Fl_{\eta}^{\perf}(R)\mid \forall j\in \BZ,\  D_{j}\subset \sigma(\sigma^{*}D_{j+1})\}
$$
\itemb Let $Y_k^{\eta,\perf}$ be the subfunctor of $X_k^{\eta,\perf}$ defined by
$$
R\in \Perf_k\mapsto Y^{\eta,\perf}_k(R)\coloneqq\{(D_{j})_{j}\in \Fl_{\eta}^{\perf}(R)\mid \forall j\in \llbracket i+1,i+d \rrbracket,\   D_j=\bigoplus_{l=0}^{j-i-1}\sigma^l((\sigma^*)^lD_{i+1})\}
$$
\end{itemize}
\end{defi}
\medskip
It is clear that $X_k^{\eta,\perf}$ is representable by a closed subscheme of $\Fl_{\eta}^{\perf}$, hence it is projective, and $Y_k^{\eta,\perf}$ is representable by an open subscheme of $X_k^{\eta,\perf}$. Moreover, these schemes are stable under the action $\GL(\eta)$ on $\Fl_{\eta}^{\perf}$ hence $\GL(\eta)$ acts on $Y_k^{\eta,\perf}$ and $X_{k}^{\eta,\perf}$ through the formula \eqref{eq:gactonflag}. Note that the isomorphism classes of these spaces do not depend on $\eta$ but only on $d$, as illustrated by the following remark:
\medskip
\begin{rema}\label{rema:dl}
The variety $Y_k^{\eta,\perf}$ is isomorphic to the perfection of the \emph{Deligne--Lusztig variety} (\cf \cite{delu}), for $\GL_d(k_E)$ associated to the longest Coxeter element $w\coloneqq s_{d-1}\cdots s_1$, which we denote $Y_k$. Moreover, $X_k^{\eta,\perf}$ is isomorphic to the perfection of the standard compactification of $Y_k$, which we denote $X_k$. Indeed, considering the data mod $\varpi$, one can prove that (\cf \cite[2.2]{delu}):
$$
Y_k\coloneqq \BP_k^{d-1}\setminus \bigcup_{H \in \sH_{k_E}} H
$$
where $\sH_{k_E}$ is the set of hyperplanes of $\BP_k^{d-1}$ defined over $k_E$. Using a basis of $\eta/\varpi$, we get an isomorphism $\BP_{\bar \eta}\cong \BP^{d-1}_k$ and the above map is given, for $R\in \Alg_k$, by
$$
(D_j)_{j}\in Y_k(R) \mapsto D_i/D_{i-1}\in \BP_{\bar \eta}(R)\cong\BP_{k}^{d-1}(R)
$$
Moreover, $X_k$ is the \emph{Demazure--Hansen} compactification of $Y_k$ defined by Deligne and Lusztig (\cf \cite[9.1, 9.10]{delu}). Hence it is smooth and projective over $k$ (\cf \cite[Lemma 9.11]{delu}).

With a little more work (\cf \cite[Chapitre II, Proposition 1.6]{gen}), we can show that $X_k$ is an iterated blow-up of the projective space. For $i\in \BN$ we define $R_i\subset \BP^{d-1}_{k}$ the closed subscheme parametrizing $k_E$-rational linear subspaces of dimension $i$. Then $X_k$ is the successive blow-up of $\BP^{d-1}_{k}$ along $\wt R_i$, the strict transform of $R_i$. Formally, this is
$$
X_k\coloneqq \Bl_{\wt R_{d-2}}(\cdots \Bl_{R_0}(\BP^{d-1}_{k})\cdots ).
$$
For example, if $d=2$, we get that $X^{\eta,\perf}_k$ is the perfection of $\BP_{\bar \eta}\cong \BP^1_k$ and that $Y^{\eta,\perf}_k$ is the perfection of
$$
\BP_{\bar \eta}\setminus\BP(\bar \eta)\cong \BP^1_k\setminus \BP^1(k_E).
$$
\end{rema}
\medskip
\subsubsection{Connected components as lattice subschemes}

\medskip
\begin{lemm}\label{lem:valu}
The $k$-scheme $\brv{\BH}_k^d$ satisfies the valuation criterion. Meaning that any morphism $\Spec\, K\rightarrow \brv{\BH}_k^d$, where $K$ is the fraction field of a discrete valuation ring $A$ over $k$, lifts uniquely to a morphism $\Spec\, A \rightarrow \brv{\BH}_k^d$.
\end{lemm}
\medskip
\begin{proof}
The uniqueness part of the valuation criterion comes from the fact that $\brv{\BH}_{k}^d$ is separated (\cf \cite[Proposition 3.70]{razi}), so we need to check the existence part. Let $\Spec\,  K\rightarrow \brv{\BH}_k^d$ be a morphism. Using Definition \ref{def:drin}, this morphism corresponds to the data\footnote{We change our conventions on indices to make the proof more readable.} of an $\rmO$-lattice $\eta_0$, a $1$-dimensional $K$-vector space $L_0$ and a map $u_0\colon \eta_0\rightarrow L_0$ such that $u_K\colon \eta_0\otimes_{\rmO}K\rightarrow L_0$ is surjective and $\eta_0/\varpi \to L_0$ is injective. Define
$$
L_0^+\coloneqq \im\bigl (\eta_0\otimes_{\rmO}A\rightarrow L_0\bigr ),
$$
which is a rank $1$ $A$-module. Let $t$ be a uniformizer of $A$. Note that since $\eta_0/\varpi \eta_0\rightarrow L_0^+/t$ may not be injective, this does not directly define an $A$-point of $\brv{\BH}_k^d$. The existence of a lift $\Spec\,  A\rightarrow \brv{\BH}_k^d$is equivalent to the existence of an integer $r$ such that $0\leqslant r \leqslant d$ and
\begin{itemize}
\itemb $\varpi \eta_0=\eta_{r+1}\subsetneq \eta_r \subsetneq \dots \subsetneq \eta_1 \subsetneq \eta_0$, $\rmO$-lattices of $E^d$,
\itemb $L_r^+,\dots, L_0^+$, $A$-modules of rank $1$,
\itemb  for any $i=0,\dots, r$ an $\rmO$-linear morphism $u_i\colon \eta_i\rightarrow L_i^+$ such that $\eta_{i}/\eta_{i+1}\incl L_i^+/t$ is injective.
\itemb for any $i=0,\dots, r$, an $A$-linear morphism $T_{i+1}\colon L_{i+1}^+\to L_{i}^+$ such that $T_i\otimes K$ is an isomorphism, $T_i\otimes k=0$, and the following diagram commutes
\begin{equation}\label{eq:comhyp}
\begin{tikzcd}
\eta_{i+1}\ar[r,hook]\ar[d,"u_{i+1}"]& \eta_i\ar[d,"u_i"]\\
L_{i+1}^+\ar[r,"T_i"]&L_i^+.
\end{tikzcd}
\end{equation}

\end{itemize}
We will define these data recursively, suppose $\eta_i, L_i^+, T_i, u_i$ are defined for a certain $i$ and define them for $i+1$. If $\eta_i/\varpi \eta_0\rightarrow L_i^+/t$ is injective, we are done, and $r=i$. Else, set
\begin{itemize}
\itemb $\eta_{i+1}/\varpi\eta_0\coloneq \ker\bigl ( \eta_i/\varpi\eta_0\rightarrow L_i^+/t\bigr )$ which defines $\eta_{i+1}\subset \eta_i$,
\itemb $L_{i+1}^+\coloneqq u_i\bigl ( \eta_{i+1}\otimes_{\rmO}A\bigr)\subset L_i^+$ and $u_{i+1}\colon \eta_{i+1}\rightarrow L_{i+1}^+$ is the restriction of $u_i$ to $\eta_{i+1}$,
\itemb $T_{i+1}\colon L_{i+1}^+\to L_{i}^+$ defined by the inclusion $\eta_{i+1}\subset \eta_i$.
\end{itemize}
The diagram \eqref{eq:comhyp} commutes by construction and it only remains to show that $T_{i+1}\otimes K$ is an isomorphism and $T_{i+1}\otimes k=0$. But note that, by construction, we actually have
$$
L_{i+1}^+=u_i(\eta_{i+1}\otimes_{\rmO}A)=tL_{i}^+\subset L_i^+.
$$
This finishes the proof of the recursion and proves the lemma.
\end{proof}
Let us recall that since $\brv{\BH}_k^d$ is \og semi-stable\fg, irreducible components of $\brv{\BH}_k^d$ are smooth over $k$. The lemma is the following:
\medskip
\begin{lemm}
Let $Y$ be a scheme locally of finite type over $k$. Suppose that $Y=\bigcup_{i\in I} U_i$ an open cover such that for every $i\in I$, $U_i$ is an open in 
$$
\Spec \frac{k[x_0,\dots, x_n]}{\langle x_0\cdots x_m\rangle}.
$$
Then any irreducible component of $Y$ is smooth over $k$. In particular, any irreducible component of $\brv{\BH}_k^d$ is smooth over $k$.
\end{lemm}
\medskip
\begin{proof}
We can reduce the claim to $Y=\Spec \frac{k[x_0,\dots, x_n]}{\langle x_0\cdots x_m\rangle}$ where if $X\subset Y$ is an irreducible component then there is $i=0,\dots, m$ such that $x_i=0$ on $X$ so that 
$$
X\subseteq \Spec k[x_0,\dots \hat x_i, \dots,x_n].
$$
Since the right-hand side is irreducible, the inclusion is an equality and $X$ is smooth over $k$.
\end{proof}
We now pin-down the irreducible components of $\BH_k^{d}$ as the closures of the $\BH_k^{\{\eta\}}$. Let $\eta\subset E^d$ be an $\rmO$-lattice, define 
$$
\BH_k^{\eta}\coloneqq \ov{\BH}_k^{\{\eta\}}\subset \brv{\BH}_{k}^d
$$
Let us begin with a moduli description of $\BH_k^{\eta}$.
\medskip
\begin{prop}\label{prop;symlatsub}
Let $\eta\subset E^d$ be a lattice and let $i$ be its logarithmic index. Let $R\in \Alg_k$, then
$$
\BH_k^{\eta}(R)=\{(\lambda,L,u,r)\in \BH_k^{d}(R)\ \mid\ \lambda_i=\eta\}.
$$
where $\lambda^i\mapsto \lambda_i$ is given by Proposition \ref{prop:repldbt}.
\end{prop}
\medskip

\begin{proof}
Write $Z$ for the moduli problem on the right-hand side. Then $Z$ is a closed subfunctor of $\brv{\BH}_k^d$, since it is closed under specialization , hence it is representable by a closed subscheme of $\brv{\BH}_k^d$. Moreover, by definition, $\BH_k^{\{\eta\}}\subset Z$. 

Let us show that $\BH_k^{\{\eta\}}$ is dense in $Z_k^{\eta}$. Let $x\in Z$, set $A=\CO_{Z,x}$ the local ring of $Z$ at $x$ and let $(\lambda,L,u,r)$ be the universal object over $A$. Then define $A'$ to be the localisation of $A$ at the ideal generated by the sections $\Pi_j$ for $j\neq i$ (\ie inverting the $x_j$ for $j\neq i$ in \eqref{eq:semist}) which defines a generalization $y\in Z$. By definition, we get that $y\in \BH_k^{\{\eta\}}$. Hence the natural inclusion $\BH_k^{\eta} \incl Z$ is an isomorphism.

\end{proof}
\medskip
\begin{coro}\label{cor:iredcomph}
For every $\rmO$-lattice $\eta\subset E^d$, $\BH_k^{\eta}$ is an irreducible proper smooth scheme over $k$ and
$$
\brv{\BH}_k^{d}=\bigcup_{\eta \in \brv{\BrTi}_{\! d}[0]}\BH^{\eta}_k.
$$
In particular, $\eta \mapsto \BH^{\eta}_k$ is a bijection between $\brv{\BrTi}_{\! d}[0]$ and irreducible components of $\brv{\BH}_k^d$
\end{coro}
\medskip
\begin{proof}
First, recall that in a topological space, a set is irreducible if and only if its closure is irreducible (\cf \cite[Chapitre II, \S 4 {\bf 1} Proposition 2]{balgcom}). Hence, $\BH_k^{\eta}$ is irreducible if $\BH_k^{\{\eta\}}$ is irreducible which is an open subset of the affine space $\BA^{d-1}_k$ by $\eqref{eq:semist})$. Thus, $\BH_k^{\eta}$ is irreducible. Since the $\BH^{\eta}_k$ clearly cover $\brv\BH_k^d$ by Proposition \ref{prop;symlatsub}, we deduce the last claim that every irreducible component of $\brv\BH_k^d$ is of the form $\BH^{\eta}_k$ for some $\rmO$-lattice $\eta\subset E^d$ ; it is clear that two distinct lattices give two distinct components.

We justify that $\BH_k^{\eta}$ is of finite type over $k$. We have
$$
\BH_k^{\eta}\subset \bigcup_{\eta\in \Delta}\BH_k^{\Delta},
$$
where the union is over the simplices $\Delta\in \brv{\BrTi}_{\! d}$ such that $\eta\in \Delta$. But this union is finite since the set of simplices $\Delta$ containing $\eta$ as a lattice is finite\footnote{Indeed, if $\eta'\in \Delta$, then $\varpi \eta \subsetneq \eta'\subsetneq \varpi^{-1}\eta$ so we have an injection
$$
\{\Delta\in \brv{\BrTi}_{\! d}\mid \eta \in \Delta\}\incl \CP(\CP(\varpi^{-1}\eta/\varpi \eta)),
$$
where $\varpi^{-1}\eta/\varpi \eta\cong (\CO/\varpi^2)^d$ is a finite set and $\CP(\ \cdot\ )$ is the set of subsets of a set.
}. First, by Lemma \eqref{lem:semist}, $\BH_k^{\Delta}$ is open inside an affine scheme of finite type over $k$, hence it is of finite type over $k$. Thus the finite union is of finite type over $k$ and since $\BH^{\eta}_k$ is closed inside this finite union, it is of finite type over $k$ (\cite[Proposition 6.3.4 i)]{ega1}).

To prove that $\BH_k^{\eta}$ is proper over $k$, we only need to show that it satisfies the valuative criterion (\cf \cite[Théorème 7.3.8]{ega2}). By Lemma \ref{lem:valu}, $\brv\BH^d_k$ satisfies the valuative criterion and since $\BH_k^{\eta}$ is closed in $\brv\BH^d_k$, it also does. This concludes the proof of the corollary. 
\end{proof}
\medskip
\subsubsection{Irreducible components as Deligne--Lusztig varieties}
\medskip
\begin{lemm}\label{lem:opdl}
There is a natural $\GL(\eta)$-equivariant isomorphism 
$$
\BH^{\{\eta\},\perf}_k\cong Y^{\eta,\perf}_k.
$$
\end{lemm}
\medskip

\begin{proof}

We construct morphisms $u\colon Y^{\eta,\perf}_k\rightarrow \BH^{\{\eta\},\perf}_k$ and $v\colon \BH^{\{\eta\},\perf}_k\rightarrow Y^{\eta,\perf}_k$ that are inverse to one another. Let $R\in \Perf_k$. We write $\bar \eta \coloneqq \eta/\varpi$ and for $D\subset \eta\otimes_{\rmO}W_{\rmO}(R)$, $\bar D\coloneqq D/\varpi \subset \bar \eta\otimes_{\rmO} R$.

Let $x=(D_j)_{j}\in Y_k^{\eta,\perf}(R)$, we set 
$$
L(x)\coloneqq \frac{D_i}{\sigma^{-1}(D_{i-1})}.
$$
and define $u(x)\coloneqq (\eta \rightarrow L(x))$ where the map is the natural one. To check that $u(x)\in \BH^{\{\eta\},\perf}_k(R)$ we need to show:
$$
\forall s\in \Spec R,\quad \bar \eta \rightarrow L(x)\otimes_R k(s)\quad \text{is injective}.
$$
We can suppose $R=k(s)$ the residue field at $s\in \Spec R$. We need to show that
$$
\eta \cap \sigma^{-1}(D_{i-1})=\sigma^{-1}(\eta\cap D_{i-1})=0
$$
Hence by Nakayama lemma it is enough to show $\bar \eta\cap \bar D_{i-1}=0$. Recall that, by definition, for $v_1\in \bar D_{i-d+1}$ a non-zero vector, $\{\sigma^{l}(v_1)\}_{0\leqslant l \leqslant d-1}$ is a basis of $\bar \eta\otimes_{\rmO}R$ and $\{\sigma^{l}(v_1)\}_{0\leqslant l \leqslant d-2}$ is a basis of $\bar D_{i-1}$. Let $v\in \bar \eta \cap \bar D_{i-1}$, which we can write as 
$$
v=\sum_{l=0}^{d-2}a_l\sigma^{l}(v_1), \quad \text{ for } a_0,\dots, a_{d-2}\in R.
$$ 
Since $v\in \bar \eta$, we get $\sigma(v)=v$, hence:
$$
v=\sum_{l=0}^{d-2}a_l\sigma^{l}(v_1)=\sum_{l=0}^{d-2}\sigma(a_l)\sigma^{l+1}(v_1)=\sigma(v).
$$
By linear independence of $\{\sigma^{l}(v_1)\}_{0\leqslant l \leqslant d-1}$, this means that $a_0=0$ and, for $l\in \llbracket 0, d-2 \rrbracket$, $a_l=\sigma^l(a_0)=0$. Thus $v=0$ and this shows that $u(x)\in \BH^{\{\eta\}}_k(R)$. The morphism $u\colon Y^{\eta,\perf}_k\rightarrow \BH^{\{\eta\},\perf}_k$ is well defined.

Let $y=(\eta \rightarrow L)\in \BH_k^{\{\eta\}}(R)$ and set 
$$
v(y)\coloneqq \bar D_{i-1}=\sigma(\sigma^*\Ker(\eta\otimes_{\rmO}R\rightarrow L))\in \BP^{d-1}_k(R)
$$ 
We fix an isomorphism $\bar \eta \cong k_E^d$. Using \ref{rema:dl}, to show that $v(y)\in Y_k^{\bar \eta}$, we need to show that $D_{i-1}$ does not contain any $k_E$-rational line. Let $z\in D_{i-1}$ be such that $\sigma(z)=z$, we will show that $z=0$. Consider the finitely generated $R$-module $N\coloneqq Rz$, to show that it is $0$, we only need to show that all its localisations at maximal ideals of $R$ are zero: hence, by Nakayama lemma, we can suppose that $R=k(s)$ is a field. Since $(\bar \eta \otimes_{k_E} k(s))^{\sigma}=\bar \eta$, by Galois descent (\cf \cite[Chapitre 5, \S 10, {\bf 4} Proposition 6 Corollaire]{balg}) $z\in \bar \eta$, thus
$$
z\in D_{i-1}\cap \bar \eta=\sigma(\sigma^*\Ker( \eta\rightarrow L))=\Ker( \eta\rightarrow L)=0,
$$
hence $z=0$. We get $v(y)\in Y^{\eta,\perf}_k(R)$ and the morphism$v\colon \BH_k^{\{\eta\}}\rightarrow Y_k^{\eta,\perf}$ is well defined.

It is clear that the two maps $u$ and $v$ are inverse of one another and it only remains to prove that $v$ is $\GL(\eta)$-equivariant. Let $g\in \GL(\eta)$, then 
$$
g\colon y=(\eta \xrightarrow{\varphi} L)\mapsto g\cdot y =(\eta \xrightarrow{\varphi\circ g^{-1}} L)
$$
and since $\Ker(\varphi\circ g^{-1})=g\cdot \Ker(\varphi)$ we get that $v(g\cdot y)=g\cdot v(y)$.
\end{proof}

\medskip
\begin{rema}\label{rema:unperf}
The way the morphism is constructed in Lemma \ref{lem:opdl} makes this isomorphism compatible with specialization on Drinfeld's moduli space and this definition only makes sense for the perfected version. Note that the composition of the map 
$$
v\colon \BH^{\{\eta\},\perf}_k\rightarrow  Y^{\eta,\perf}_k,
$$
with $F^{-1}$ defines a map which descends to the unperfected schemes, and defines an isomorphism
$$
v'\colon \BH^{\{\eta\}}_k\xrightarrow{\sim}  Y_k.
$$
\end{rema}
\medskip

\medskip
\begin{prop}\label{prop:cpctdl}
There exists a unique isomorphism $X_k^{\eta,\perf}\xrightarrow{\sim}\BH^{\eta,\perf}_k$ making the following diagram commute
\begin{center}
\begin{tikzcd}
 Y_k^{\eta,\perf}\ar[r,"\sim"]\ar[d,hook]&\BH_k^{\{\eta\},\perf}\ar[d,hook]\\
X_k^{\eta,\perf}\ar[r,"\sim"]&\BH_k^{\eta,\perf}
\end{tikzcd}
\end{center}
Moreover, this isomorphism is $\GL(\eta)$-equivariant.
\end{prop}
\medskip

\begin{proof}
By Remark \ref{rema:unperf}, it is enough to prove the unperfected version of the proposition by composing with $F^{-1}$. We are going to use the following standard lemma:
\medskip
\begin{lemm}
Let $X,X'$ be irreducible smooth proper schemes over $k$ and let $U\subset X$ and $U'\subset X'$ be dense open subschemes. Suppose we have an isomorphism $f\colon U\xrightarrow{\sim} U'$. Then there exists a unique isomorphism $g\colon X\xrightarrow{\sim}X'$ extending $f$, in the sense that the following diagram 
\begin{center}
\begin{tikzcd}
U\ar[r,"f"]\ar[d,hook] &U'\ar[d,hook]\\
X\ar[r,"g"]&X'
\end{tikzcd}
\end{center}
commutes.
\end{lemm}
\medskip
\begin{proof}
We first prove existence and uniqueness of the morphism $g\colon X\rightarrow X'$ extending $f$. We still write $f\colon U\rightarrow U'\rightarrow X'$ for the composition. Let $G\subset X\times_kX'$ be the closure of the graph of $f$ and write $\pi \colon G\rightarrow X$ for the projection on the first factor. For $g$ to exist (and be unique) we need to show that $\pi$ is an isomorphism. But by the hypothesis $\pi$ is proper and birational. Hence, since $X$ is irreducible and $X'$ is normal and integral, $\pi$ is an isomorphism by Zariski's main theorem (\cf \cite[Lemme 8.12.10.1]{ega43}). 

We now prove that $g$ is an isomorphism. By exchanging the roles of $X$ and $X'$ we get a unique morphism $g'\colon X'\rightarrow X$ extending $f^{-1}$. But $g\circ g'$ (resp. $g'\circ g$) is the identity when restricted to $U$ (resp. $U'$). By the preceding uniqueness property applied to $X=X'$, this proves that $g'$ is the inverse of $g$ and thus that it is an isomorphism.
\end{proof}
Note that the lemma also provides the $\GL(\eta)$-equivariance.
\end{proof}

Let $r\in \llbracket 0,d-1\rrbracket$ and $\Delta\in \BrTi_{\! d}[r]$ an $r$-simplex. Using Proposition \ref{prop:cpctdl} we get the following:
\medskip
\begin{coro}\label{cor:interslatt}
For any $\eta\in \Delta$ the isomorphism $f_{\eta}\colon\BH_k^{\eta,\perf}\cong X_k^{\eta,\perf}$ identifies $\bigcap_{\eta\in \Delta} \BH^{\eta,\perf}_k$ with the closed subscheme $X_k^{\Delta,\perf}$ defined, for $R\in \Perf_k$, by
$$
X_k^{\Delta,\perf}(R)\coloneqq \{(D_j)_j\in X_k^{\eta,\perf}(R)\mid \forall \eta_i\in \Delta,\ D_{i}=\eta_i \otimes_{\rmO}W_{\rmO}(R)\}.
$$

\end{coro}
\subsection{Critical indices over perfect algebras}
In this subsection, we prove the following theorem: 
\medskip
\begin{theo}\label{thm:critexist}
Let $S$ be an irreducible perfect scheme over $k$ and $X$ a special formal $\rmOD$-module over $S$. Then $X$ admits a critical index.
\end{theo}
\medskip

\subsubsection{Subscheme of critical indices}
We define the subscheme of $\rmM_k$ of special formal $\rmOD$-modules locally having a critical index. For $R\in \Alg_k$ we define 
$$
\rmM_k^{\crit}(R)\coloneqq\left \{(X,\rho)\in \rmM_k(R)\mid\begin{array}{c}\forall S\subset \Spec R\text{ irreducible,}\\
														\text{ $X_S$ admits a critical index} \end{array} \right \}.
$$
\medskip
\begin{lemm}\label{lem:repclo}
The functor on $k$-algebras $R\longmapsto \rmM_k^{\crit}(R)$ is representable by a closed subscheme $\rmM_k^{\crit}\subset \rmM_k$.
\end{lemm}
\medskip
\begin{proof}
We only need to check that the condition defining $\rmM_k^{\crit}$ is closed, since a closed immersion is relatively representable. Let $X^{\univ}\rightarrow \rmM_k$ be the universal special formal $\rmOD$-module and for every $i\in \BZ/d$, $\sL_i\coloneqq \Lie(X^{\univ})_i$ is an invertible sheaf on $\rmM_k$. Then $\Pi\in \rmOD$ defines an morphism
$$
\Pi_i\colon \sL_i\rightarrow \sL_{i+1}
$$
and hence defines a section $\Pi_i\in H^0(\rmM_k,\sL_{i+1}\otimes \sL_i^{\vee})$. Let $V(\Pi_i)$ be the zero locus of this section, which defines a closed subscheme of $\rmM_k$ and thus the finite union
$$
\rmM_k^{\crit}=\bigcup_{i\in \BZ/d}V(\Pi_i)
$$
is a closed subscheme in $\rmM_k$.
\end{proof}
Since special formal $\rmOD$-modules have critical indices over fields of characteristic $p$ (\cf Lemma \ref{lem:havcrit}) we get the following corollary:
\medskip
\begin{coro}\label{cor:univho}
The morphism $\rmM_k^{\crit}\incl \rmM_k$ is a universal homeomorphism.
\end{coro}
\medskip
\begin{proof}
By lemma \ref{lem:repclo} the map $\rmM_k^{\crit}\rightarrow \rmM_k$ is a closed immersion. Hence it is universally injective and integral and to show that it is a universal homeomorphism we only need to show that it is surjective (\cf \cite[04DF]{stacks}). Let $\kappa/k$ be a field extension, which is a field of characteristic $p$, and let $(X,\rho)\in \rmM_k(\kappa)$. Then by Lemma \ref{lem:havcrit}, $X$ has a critical index and thus $(X,\rho)\in \rmM_k^{\crit}(\kappa)$. This means that $\rmM_k^{\crit}\rightarrow \rmM_k$ is surjective on points. This finishes the proof.
\end{proof}
\subsubsection{Perfect subscheme of critical indices}
We now prove theorem \ref{thm:critexist}. We let $\rmM_k^{\crit,\perf}$ be the perfection of $\rmM_k^{\crit}$. Using corollary \ref{cor:univho} and the fact that universal homeomorphism between perfect schemes are isomorphisms, we get the result:
\medskip
\begin{coro}
The natural map $\rmM_k^{\crit,\perf}\rightarrow \rmM_k^{\perf}$ is an isomorphism.
\end{coro}
\medskip
\begin{proof}
We give two different proofs of the statement:
\begin{itemize}
\itemb The first proof is more elementary. We know that $f\colon \rmM_k^{\crit,\perf}\rightarrow \rmM_k^{\perf}$ is a closed immersion; let $\CI$ be the ideal sheaf defining it. Since this closed immersion is bijection on points by corollary \ref{cor:univho} we deduce that $\CI$ is a nilpotent ideal. But since $\rmM_k^{\perf}$ is perfect, hence reduced, $\CI=0$, which implies $f$ is an isomorphism. Note that we only used the claim contained in \ref{cor:univho} that $f$ is a bijection on points.
\itemb The second proof is more conceptual. By \cite[Lemma 3.8]{bhsc} we know that a universal homeomorphism between perfect schemes is an isomorphism. Hence by corollary \ref{cor:univho} $\rmM_k^{\crit,\perf}\rightarrow \rmM_k^{\perf}$ is an isomorphism.
\end{itemize}
\end{proof}
\medskip
\begin{rema}
Note that this does not mean that an SFD over a perfect scheme has a critical index. But if $X$ is an SFD over a perfect scheme $S$ and if for $i\in \BZ/d$ we write $S_i\subset S$ for the closed subscheme where $i$ is critical then $S=\bigcup_{i \in \BZ/d} S_i$. In particular, if $S$ is irreducible, $X$ has a critical index which proves Theorem \ref{thm:critexist}.
\end{rema}
\medskip
\subsection{Lattice subschemes}
In this subsection, we define lattice subschemes of $\rmM_k^{\perf}$ and prove that they are represented by certain subschemes of flag varieties, given as compactifications of Deligne--Lusztig varieties.
\subsubsection{Definition}
Let $R\in \Perf_k$, $(X,\rho)\in\rmM_k(R)$ and let $M=\rmM_{\rmO}(X)$ be the special Cartier $R$-module associated to $X$. 

For $i\in \BZ$, let $r_i$ be the composition
\begin{equation}\label{eq:defri}
r_i\colon M_i\rightarrow M_i\otimes_{\Zp}\Qp\xrightarrow{\rho^{-1}}\bM_i\otimes_{\breve{\rmO}} W_{\rmO}(R)_{\Qp}\xrightarrow{\iota_{i}\otimes \id}E^d\otimes_{\rmO} W_{\rmO}(R)
\end{equation}
where:
\begin{itemize}
\itemb the first map is the natural map $M_i\incl M_i\otimes_{\Zp}\Qp$ which is an inclusion since $W_{\rmO}(R)$ is torsion-free as $R$ is perfect,
\itemb the second map is the morphism on rational Cartier modules given by the inverse of the quasi-isogeny $\rho \colon \bX \otimes_{k} R/p \rightarrow X \otimes _R R/p$ ; it is thus an isomorphism,
\itemb the last map is the isomorphism $\iota_i\colon \boeta^i\otimes_{\rmO} E\xrightarrow{\sim} E^d$, defined in paragraph \ref{subsubsec:refob} given by $g_{\Pi}^{-i}$, tensored with $W_{\rmO}(R)$ which defines an isomorphism.
\end{itemize}
We get an injective $W_{\rmO}(R)$-linear morphism $r_i\colon M_i\incl E^d\otimes_{\rmO} W_{\rmO}(R)$. We get the following lemma on these maps:
\medskip
\begin{lemm}\label{lemm:rcomp}

The following diagrams commute:
$$
\begin{array}{cc}
\begin{tikzcd}
M_i\ar[rr,"\Pi"]\ar[dr,swap,"r_i"]&&M_{i+1}\ar[dl,"r_{i+1}"]\\
& E^d\otimes_{\rmO} W_{\rmO}(R)
\end{tikzcd},
&
\begin{tikzcd}
M_i\ar[r,"V"]\ar[d,"r_i"]&\sigma^*M_{i+1}\ar[d,"\sigma^{*}r_{i+1}"]\\
E^d\otimes_{\rmO} W_{\rmO}(R)\ar[r,"(\id\otimes \sigma^{-1})"]&E^d\otimes_{\rmO} \sigma^{*}W_{\rmO}(R)
\end{tikzcd}
\end{array}
$$

\end{lemm}
\medskip
\begin{proof}
The first diagram commutes because
$$
\iota_{i+1}\circ \Pi= g_{\Pi}\circ \iota_{i+1}=\iota_{i},
$$

Now for the other diagram, recall that on $\bM$, $\bV=\Pi\circ (\id\otimes \sigma^{-1})$ so by the first diagram
$$
(\iota_{\langle i+1 \rangle}\otimes \id) \circ \bV=\iota_{\langle i \rangle}\otimes \sigma^{-1},
$$
so since $\rho^{-1}\circ V=\bV\circ \rho^{-1}$, we get
$$
\sigma^*r_{i+1}\circ V=(\iota_{\langle i+1 \rangle }\otimes \id)\circ \bV \circ \rho^{-1} =(\iota_{\langle i \rangle}\otimes \sigma^{-1}) \circ \rho^{-1}=(\id \otimes \sigma^{-1})\circ r_{i}.
$$

\end{proof}
\medskip
\medskip
\begin{rema}\label{rem:requiv}
By Remark \ref{rem:zerograd}, the map $r_i$ is $\GL_d(E)$-equivariant in the sense that for $g\in \GL_d(E)$, 
$$
r_i\circ g = g\circ r_{i+v},
$$
where $v=v_{\varpi}(\det(g))$.
\end{rema}
\medskip
\begin{defi}\label{def:lattsubsc}
 Let $\eta\subset E^d$ be a $\rmO$-lattice and let $i\in \BZ$ be its index. We define $\rmM^{\eta}_k$ as a subfunctor of $\brv{\rmM}_k^{\perf}$ defined on perfect $k$-algebras by
 $$
 R\in \Perf_k\mapsto \rmM^{\eta}_k(R)\coloneqq 
 \left\{(X,\rho)\in \brv{\rmM}_k^{\perf}(R) \mid  M_{\rmO}(X)_i\xrightarrow[\sim]{r_i}\eta\otimes_{\rmO} W_{\rmO}(R)\right \}.
 $$

\end{defi}
\medskip
Since $\rmM^{\eta}_k$ is defined by a closed subfunctor of a representable functor it is itself representable by a closed subscheme of $\brv{\rmM}_k^{\perf}$. Using theorem \ref{thm:critexist}, we get the following corollary :
\medskip
\begin{coro}

$$
\brv{\rmM}_k^{\perf}=\bigcup_{\eta\in \brv{\BrTi}_{\!d}[0]}\rmM^{\eta}_k.
$$
\end{coro}
\medskip
\begin{proof}
Let $x=(X,\rho)\in \rmM_k^{\perf}$ and let $h=\heig \rho$. After replacing $x$ by a suitable specialization, we may assume that it lies on an irreducible component. Let $M=M_{\rmO}(X)$ be the Cartier module of $X$. By Theorem \ref{thm:critexist}, $X$ has a critical index $i\in \BZ/d$ and let $\eta^{i}\coloneqq M_i^{\Pi=V}\subset \boeta^i$ defined by Lemma \ref{lem:inva}. Let $i'\in \BZ$ be the unique lift of $i$ such that $i'\in \llbracket h-d+1,h\rrbracket$, and $\eta=\iota_{i'}(\eta^{i})\subset E^d$. Then
$$
r_{i'}\colon M_i\cong \eta\otimes_{\rmO}W_{\rmO}(R),
$$
and $x\in \rmM_k^{\eta}$.
\end{proof}

\subsubsection{Constructing the morphisms}
In this paragraph we construct morphisms 
$$
\rmM^{\eta}_k\xrightleftharpoons[g_{\eta}]{f_{\eta}} X^{\eta,\perf}_k
$$ 
which will be inverse to one another. We simply write $f=f_{\eta}$ and $g=g_{\eta}$ in the proof.

We first define the map $f\colon\rmM^{\eta}_k\rightarrow X^{\eta,\perf}_k$. Let $R\in \Perf_k$ and let $x=(X,\rho)\in \rmM^{\eta}_k(R)$. Let $M=M_{\rmO}(X)$ be the special Cartier $R$-module associated to $X$ and for $j\in\BZ$ set
\begin{equation}\label{eq:flagdef}
D_j(x)\coloneqq {r_{j}(M_{j})}\subset  E^d\otimes_{\rmO} W_{\rmO}(R).
\end{equation}
\medskip
By definition, $D_i(x)=\eta\otimes_{\rmO} W_{\rmO}(R)$.
\begin{lemm}
Let $R\in \Perf_k$ and $x=(X,\rho)\in \rmM^{\eta}_k(R)$. Then $D(x)\coloneqq (D_j(x))_{j}\in X^{\eta,\perf}_k(R)$. In other words, the map $x\mapsto D(x)$ defines a morphism $f\colon\rmM^{\eta}_k\rightarrow X^{\eta,\perf}_k$.
\end{lemm}
\medskip
\begin{proof}
We first show that $D(x)=(D_j)_j\in \Fl_{\eta}^{\perf}(R)$. We write $M=M_{\rmO}(X)$ the special Cartier $R$-module of $X$ and $W=W_{\rmO}(R)$.

For $j\in \BZ$, let $P_j\coloneqq M_{j+1}/\Pi M_{j}$. Using the first diagram of lemma \ref{lemm:rcomp} and the exact sequence
$$
0\rightarrow  M_{j}\xrightarrow{\Pi} M_{j+1}\rightarrow P_j\rightarrow 0,
$$
we get an inclusion $r_j(M_j)\subset r_{j+1}(M_{j+1})$ whose quotient is naturally identified with $P_j$. Recall by Lemma \ref{lem:specarthei} that $P_j$ is a locally free $R$-module of rank $1$. Hence $D(x)\in \Fl_{\eta}^{\perf}(R)$.

To check that $D(x)\in X^{\eta,\perf}_{k}(R)$, we use the second diagram of lemma \ref{lemm:rcomp}. Let $j\in \BZ$, applying $r_j$ to $M_j\xrightarrow{V}\sigma^*M_{j+1}$ we get
$$
r_j(M_j)\xrightarrow{(1\otimes \sigma^{-1})}r_{j+1}(\sigma^*M_{j+1})=\sigma^*r_{j+1}(M_{j+1}).
$$

Rewriting this, we get $D_j\subset\sigma(\sigma^*D_{j+1})$, which proves that $D(x)\in X^{\eta,\perf}_k(R)$.
\end{proof}
\medskip
\begin{rema}\label{rema:tracer}
Tracing the action of $V$ in the proof, note that for $x=(X,\rho)\in \rmM_k^{\perf}(R)$, we get that
$$
\Lie(X)_i\cong D_i/\sigma^{-1}((\sigma^{-1})^*D_{i-1}),
$$
and because $\eta\otimes_{\rmO}W_{\rmO}(R)=D_i$ the natural map $\eta\rightarrow D_i/\sigma^{-1}((\sigma^{-1})^*D_{i-1})$ coincides with the composition $\eta\xrightarrow{r_i} M_i \rightarrow \Lie(X)_i$.
\end{rema}
\medskip
We now define a map $g\colon X^{\eta,\perf}_k\rightarrow \rmM^{\eta}_k$. Let $R\in \Perf_k$ and let $y=(D_j)_{j}\in X^{\eta,\perf}_k(R)$. We have maps
$$
\eta\otimes_{\rmO}R\xleftarrow{\mod \varpi} \eta\otimes_{\rmO}W_{\rmO}(R)\xrightarrow{\iota^{-1}_j\otimes \id}\bM_j\otimes_{\brv{\rmO}}W_{\rmO}(R)_{\Qp}.
$$
Set for $j\in \BZ$, 
$$
M_{j}=M_{j}(y)\coloneqq (\iota_j\otimes \id)^{-1}(D_j).
$$
Note that $M_j=M_{j+d}$ so that the index can be considered $j\in \BZ/d$. Then $M=M(y)\coloneqq \bigoplus_{j\in \BZ/d}M_j$ is a $\BZ/d$-graded $W_{\rmO}(R)$ submodule of $\bM\otimes_{\rmO}W_{\rmO}(R)_{\Qp}$. Notice that the natural inclusion induces an isomorphism 
$$
s(y)^{-1}\colon M(y)\otimes_{\Zp}\Qp\xrightarrow{\sim}\bM\otimes_{\rmO}W_{\rmO}(R)_{\Qp}.
$$
We now define $V,\Pi$ on $M$.
\begin{itemize}
\itemb The inclusion $D_j\subset D_{j+1}$ defines a $W_{\rmO}(R)$-linear map $\Pi_{i+j}\colon M_{i+j}\rightarrow M_{i+j+1}$ which defines $\Pi=\Pi(y)\colon M\rightarrow M$. Because $D_{j}=\varpi D_{j+d}$, we get that $\Pi^d=\varpi$.
\itemb Similarly, the injective maps $D_j\xrightarrow{1\otimes\sigma^{-1}} \sigma^*D_{j+1}$ defines a $W_{\rmO}(R)$-linear injective map $V\colon M\rightarrow \sigma^*M$, \ie a degree $1$, $\sigma^{-1}$-linear injective map $V\colon M\rightarrow M$. Moreover, $V^dM\subset \varpi M$ so that the action of $V$ is topologically nilpotent. Note that $\Pi$ and $V$ commute since $\Pi V$ and $V\Pi$ are both defined by the natural inclusions $D_j\subset \sigma(\sigma^*D_{j+2})$.
\itemb To define $F$ we will use the relation $VF=\varpi$. First note that $\varpi M\subset V M$ since $M/VM$ is an $R$-module on which $\varpi$ acts by $0$. Since $V$ is injective we can define $F\coloneqq V^{-1}\varpi \colon M\rightarrow M$ which is a degree $-1$, $\sigma$-linear map commuting with $\Pi$ and satisfying $FV=VF=\varpi$.
\end{itemize}
\medskip
Hence we have defined $M(y)$ as a $\BZ/d$-graded $\BE_{\rmO}(R)$-module with an endomorphism $\Pi(y)$.
\medskip
\begin{prop}
Let $R\in \Perf_k$ and $y\in X^{\eta,\perf}_k(R)$. Then $M(y)$ defines a special Cartier module. If $X(y)$ is its associated special formal $\rmOD$-module over $R$ then $s(y)$ defines an $\rmOD$-linear quasi-isogeny $\rho(y)\colon \bX\otimes_k R/p\rightarrow X(y)\otimes_R R/p$. In other words the map $y\mapsto (X(y),\rho(y))$ defines a morphism $g\colon X^{\eta,\perf}_k\rightarrow \rmM_k^{\eta}$.
\end{prop}
\medskip
\begin{proof}
We only need to check that $M=M(y)$ is a special Cartier module. First note that since $R$ is perfect, $M$ is $V$-complete since it is $\varpi$-complete. We only need to check that for $j\in \BZ/d$, $M_j/VM_{j-1}$ is a finite locally free $R$-module; we will show that $M_j$ is in fact a finite locally free $W_{\rmO}(R)$-module which is equivalent to showing that $D_j$ is a finite locally free $W_{\rmO}(R)$-module. We will use the following lemma :
\medskip
\begin{lemm}
Let $D$ be a $W=W_{\rmO}(R)$-module, where $R\in \Perf_k$. Let $n\geqslant 1$ be an integer and suppose that 
\begin{itemize}
\itemb $D/\varpi$ is a finite locally free $R$-module of rank $n$,
\itemb $D\subset D\bigl [ \tfrac 1 {\varpi} \bigr ]$ and $D\bigl [ \tfrac 1 {\varpi} \bigr ]$ is a finite free $W\bigl [ \tfrac 1 {\varpi} \bigr ]$-module of rank $n$.
\end{itemize}
Then $D$ is a finite locally free $W$-module of rank $n$.
\medskip
\end{lemm}
\medskip
\begin{proof}
Up to localising, we can suppose $D/\varpi$ and $D\bigl [ \tfrac 1 {\varpi} \bigr ]$ are free. Let  $x_1,\dots, x_n\in D$ be a lift of a basis of $D/\varpi$. Then by Nakayama lemma $x_1,\dots, x_n$ is a generating family of $D$ and we need to prove it is free. Since, in particular, $x_1,\dots, x_n$ generates $D\bigl [ \tfrac 1 {\varpi} \bigr ]$ as a $W\bigl [ \tfrac 1 {\varpi} \bigr ]$-module and $D\bigl [ \tfrac 1 {\varpi} \bigr ]$ is free of rank $n$, we deduce that the family $x_1,\dots, x_n$ is linearly independent in $D\bigl [ \tfrac 1 {\varpi} \bigr ]$. Because $D\subset D\bigl [ \tfrac 1 {\varpi} \bigr ]$, $x_1,\dots, x_n$ are linearly independent in $D$.
\end{proof}
First we note that $ D_j\bigl [ \tfrac 1 {\varpi} \bigr ]\cong \eta\otimes_{\rmO}W_{\rmO}(R)[ \tfrac 1 {\varpi} \bigr ]$ is free of rank $d$ over $W_{\rmO}(R)\bigl[ \tfrac 1 {\varpi} \bigr ]$ so we only need to check that $D_j/\varpi$ is locally free as an $R$-module. But since $ D_j/\varpi$ is by definition filtered by $d$ submodules whose successive quotients are projective of rank $1$, it is projective of rank $d$.

\end{proof}
It is clear that two maps constructed above are inverse to one another which gives the following corollary:
\begin{coro}
The maps $f$ and $g$ are inverse to each other and define an isomorphism $\rmM^{\eta}_k\cong X^{\eta,\perf}_k$. Moreover, $\rmM^{\eta}_k\cong \BH_k^{\eta,\perf}$.
\end{coro}
These isomorphisms are equivariant under the action of $\GL_d(E)$:
\medskip
\begin{lemm}\label{lem:gleqis}
The isomorphism $f_{\eta}$ is $\GL(\eta)$-equivariant.

Moreover, if $g\in \GL_d(E)$, then 
$$
g\cdot \rmM_{k}^{\eta}=\rmM_{k}^{g\eta},\quad g\cdot \BH_k^{\eta}= \BH_k^{g\cdot \eta},
$$
and $f_{g\eta}\circ g = g\circ f_{\eta}$.
\end{lemm}
\medskip
\begin{proof}
This is an immediate consequence of the $\GL_d(E)$-equivariance of $r_i$ given by remark \ref{rem:requiv} and the $\GL_d(E)$-equivariance of the isomorphism $X_k^{\eta}\cong \BH_k^{\eta}$ from Proposition \ref{prop:cpctdl}.
\end{proof}
\subsection{Gluing lattice subschemes}
In this subsection we will glue the isomorphisms $f_{\eta}$ into an isomorphism, getting the following proposition which concludes the proof of theorem \ref{thm:perfiso}
\medskip
\begin{prop}
There exists a unique $\GL_d(E)$-equivariant isomorphism $f\colon \brv{\rmM}^{\perf}_k\xrightarrow{\sim}\brv{\BH}_k^{d,\perf}$ such that for any lattice $\eta$ in $E^d$
$$
\restr{f}{\rmM^{\eta}_k}=f_{\eta}\colon \rmM^{\eta}_k\rightarrow \BH^{\eta,\perf}_k.
$$
\end{prop}
\medskip
\subsubsection{Gluing isomorphisms}
\medskip
\begin{lemm}
Let $X$ be a scheme and $X_1,X_2\subset X$ be two closed subschemes such that $X=X_1\cup X_2$. Let $Z=X_1\cap X_2$, then $Z$ is the fibered pushout of $X_1$ and $X_2$ along $Z$, \ie $X=X_1\sqcup_ZX_2$.
\end{lemm}
\medskip
\begin{proof}
Let $\CI_1,\CI_2\subset \CO_X$ be the ideals defining $X_1$ and $X_2$ respectively. Then $Z$ is defined by $\CI\coloneqq\CI_1+\CI_2$ and since $X$ is covered by $X_1$ and $X_2$ we have $\CI_1\cap \CI_2=0$. It suffices to show that the diagram of $\CO_X$-modules
\begin{center}
\begin{tikzcd}
\CO_X\ar[r]\ar[d]&\CO_X/\CI_1\ar[d]\\
\CO_X/\CI_2\ar[r]&\CO_X/\CI
\end{tikzcd}
\end{center}
is cartesian which can be checked on rings. The claim then becomes that for a ring $A$ and $I_1,I_2\subset A$ ideals, setting $I=I_1+I_2$, $I'=I_1\cap I_2$ the following sequence
$$
0\rightarrow A/I'\rightarrow A/I_1\oplus A/I_2\rightarrow A/I\rightarrow 0
$$
is exact, which can immediately be checked.
\end{proof}
\medskip
\begin{coro}\label{cor:gluit}
Let $I$ be a countable set and $X$ and $Y$ be schemes such that 
$$
X=\bigcup_{i\in I} X_i,\quad Y=\bigcup_{i\in I} Y_i
$$
where, for every $i\in I$, $X_i$ and $Y_i$ are respectively closed in $X$ and $Y$. Suppose that for every $i\in I$ we have isomorphisms $f_i\colon X_i\xrightarrow{\sim} Y_i$ such that for any $j\in I$, $\restr{f_i}{X_i\cap X_j}=\restr{f_j}{X_i\cap X_j}$. Then there exists a unique isomorphism $f\colon X\xrightarrow{\sim} Y$ such that for every $i\in I$, $\restr{f}{X_i}=f_i$.
\end{coro}
\medskip
\begin{proof}
By induction, we only need to prove this for $I=\{1,2\}$ which is straightforward using the previous lemma. Since we have maps $f_i\colon X_i\rightarrow Y_i\rightarrow Y$ which coincide on $X_1\cap X_2$, the pushout property gives a map $f\colon X\rightarrow Y$. By exchanging $X$ and $Y$ we get $g\colon Y\rightarrow X$. By the uniqueness of the pushout these maps are inverses of one another.
\end{proof}

To get the morphism of Theorem \ref{thm:perfiso}, we need the following lemma, which is a consequence of Corollary \ref{cor:interslatt} and the definition of $f_{\eta}$ in \eqref{eq:flagdef}:
\medskip
\begin{lemm}\label{lem:itglues}
Let $\eta,\eta'\in \brv{\BrTi}_{\! d}[0]$ be two lattices, then 
$$
\restr{f_{\eta}}{\rmM^{\eta}_k\cap\rmM^{\eta'}_k}=\restr{f_{\eta'}}{\rmM^{\eta}_k\cap\rmM^{\eta'}_k}.
$$
\end{lemm}
\medskip
This defines a morphism 
$$
f\colon \brv \rmM_k^{\perf}\rightarrow (\brv \BH^d_{k})^{\perf}.
$$
Finally, we need to check that this morphism is $\GL_d(E)$-equivariant, which is a consequence of the uniqueness and Lemma \ref{lem:gleqis}.
\medskip
\begin{coro}
The morphism $f$ is $\GL_d(E)$-equivariant.
\end{coro}
\medskip
\newpage
\section{End of proof}
\subsection{The specialization morphisms}\label{subsec:specomp}
The final theorem we need to prove is the compatibility of the specialization maps:
\medskip
\begin{theo}\label{thm:specomp}
The specialization morphisms $\spe_{\Dr}\colon \rvert \rmM_{\brv E}\rvert \rightarrow \lvert \rmM_k\rvert$ and $\spe_{\BH}\colon \lvert \BH_{\brv E}^d\rvert \rightarrow \lvert \BH_k^d\rvert$ are compatible in the sense that the following diagram commutes: 
\begin{center}
\begin{tikzcd}
\rvert \rmM_{\brv E}\rvert  \ar[r,"\pi"]\ar[d,"\spe_{\Dr}"]&\rvert \BH_{\brv E}^d\rvert \ar[d,"\spe_{\BH}" ] \\
\lvert \rmM_{k}\rvert \ar[r]&\rvert \BH_{k}^d\rvert
\end{tikzcd}
\end{center}
where the top isomorphism is given by theorem \ref{thm:genef} and the bottom isomorphism by theorem \ref{thm:perfiso}.
\end{theo}
\medskip
One feature that makes the proof of this theorem particularly easy is that we only need to check it on classical points. Let $K/\brv E$ be a finite extenson of valued fields and let $\rmO_K\subset K$ be its ring of integers. We write $\varpi_K\in \rmO_K$ the uniformizer of $K$ and $\fkm_K\subset \rmO_K$ the maximal ideal it generates.
\subsubsection{The specialization morphisms}
We first describe the specialization morphism for $\CM_{\Dr}$: $\spe_{\Dr} \colon\lvert \rmM_{\brv E}\rvert \rightarrow \lvert \rmM_k\rvert$. Let $x=(X,\rho)\in \rmM_{\brv E}(K)=\CM_{\Dr}(\rmO_K)$ 

Then the value of the specialization $\spe_{\Dr}$ on $x$ is given by
$$
\spe_{\Dr}(x) = (X/\fkm_K,\rho/\fkm_K)\in \rmM_k(k)
$$
where $X/\fkm_K\coloneqq X\otimes_{\rmO_K} k$ and $\rho/\fkm_K \coloneqq \rho\otimes_{\rmO_K} k \colon \bX\dashrightarrow (X/\fkm_K)$ is the quasi-isogeny obtained by applying $\otimes_{\rmO_K} k$ to $\rho$.

We now describe the specialization morphism for $\BH_{\brvO}^d$: $\spe_{\BH}\colon \lvert \BH_{\brv E}^d\rvert \rightarrow \lvert \BH_{k}^d\rvert$. Let $x\in \BH_{\brv E}^d(K)=\BH_{\brvO}^d(\rmO_K)$ which corresponds to data $(\eta,L,u,r)$ over the local ring $\rmO_K$. The value of the specialization $\spe_{\BH}$ on $x$ is given by
$$
\spe_{\BH}(x) =(\eta,L\otimes_{\rmO_K}k,u\otimes_{\rmO_K}k,r)\in \BH_k^d(k)
$$
Let $\Delta=\{\eta_{i_1},\dots, \eta_{i_r}\}$ be a simplex in $\brv{\BrTi}_{\! d}$, the specialization map restricts to $\BH^{\Delta}_{\rmO}$ in the obvious way: 
$$
\spe_{\BH_{\rmO}^{\Delta}}\colon (\eta_{i_j}\xrightarrow{\varphi_{i_j}}L_{i_j})_j\mapsto (\eta_{i_j}\xrightarrow{\varphi_{i_j}/\fkm_K}L_{i_j}\otimes_{\rmO_K}k)_j
$$
where $\varphi_{i_j}/\fkm_K\colon v\mapsto \varphi_{i_j}(v)\mod \fkm_K$.

\medskip
\begin{rema}\label{rema:speconsym}
Let us comment on how the specialization map behaves under the isomorphism 
$$
\BH_{\brv E}^d\cong \BP^{d-1}_{\brv E}\setminus \bigcup_{H\in \sH_E}H
$$
of Proposition \ref{prop:symgenfib}. Let $x\in  \BP^{d-1}_{\brv E}\setminus \bigcup_{H\in \sH_E}H$ be a $K$-valued point. Then $x$ corresponds to an injective map 
$$
x\colon E^d\incl K.
$$
Let $\Delta = \{\eta_{i_1},\dots, \eta_{i_r}\}$ be the simplex associated to $\lVert \cdot \rVert_x$, which according to Lemma \ref{lem:compatlatnorm} fits into a diagram of the form  $(\eta_{i_j}\rightarrow \fkm_K^{n_j})_j$ where $n_j\in \BN$. We then have 
$$
\spe_{\BH}(x)=(\eta_{i_j}\rightarrow \fkm_K^{n_j}/\fkm_K^{n_j+1})_j\in \BH_k^{\Delta}(k)
$$
\end{rema}

\subsubsection{Matching the specialization morphisms}

The main lemma describing the compatibility between the specialization morphisms is the following:
\medskip
\begin{lemm}
Let $K/\brv E$ be a finite extension with ring of integers $\rmO_K\subset K$ and maximal ideal $\fkm_K\subset \rmO_K$. Let $G$ be a $\varpi$-divisible $\rmO$-module over $\rmO_K$ and let $h_G\colon D_{\rmO}(G_k)\rightarrow \Lie (G)$ be the map defined by the Hodge filtration on $D_{\rmO}(G)\subset\BD_{\rmO}(G)_{\rmO_K}$. Then the following diagram commutes:
\begin{center}
\begin{tikzcd}
D_{\rmO}(G_k)\ar[r,"h_G"]\ar[dr,swap, "\mod V"]& \Lie(G) \ar[d,"\mod \fkm_K"]\\
 & \Lie(G_k)
\end{tikzcd}
\end{center}
In particular, $\Lie(G_k) = \Lie(G)\otimes_{\rmO_K}k$.
\end{lemm}
\medskip
\begin{proof}
First notice that $\Lie(G_k)\cong \Lie(G)\otimes_{\rmO_K}k$ by functoriality of the Lie algebra and $\BD_{\rmO}(G)_k\cong \BD_{\rmO}(G_k)_k$ by functoriality and the crystal property of $\BD_{\rmO}(G)$. By functoriality of the Hodge filtration we get a diagram
\begin{center}
\begin{tikzcd}
\BD_{\rmO}(G)_{\rmO_K}\ar[r,"h_G"]\ar[d]& \Lie(G) \ar[d,"\mod \fkm_K"]\\
 \BD_{\rmO}(G_k)_k\ar[r,"h_{G_k}"]& \Lie(G_k)
\end{tikzcd}
\end{center}
which rewrites to
\begin{center}
\begin{tikzcd}
D_{\rmO}(G_k)\ar[r,"h_G"]\ar[dr,dashed]\ar[d]& \Lie(G) \ar[d,"\mod \fkm_K"]\\
 D_{\rmO}(G_k)\otimes k\ar[r,"h_{G_k}"]& \Lie(G_k)
\end{tikzcd}.
\end{center}
We check that the diagonal dashed arrow is given by the reduction mod $V$. First, note that by \cite[Proposition 4.3.9]{bbm} 
$$
h_{G_k}={\rm mod}\,  V,\quad D_{\rmO}(G_k)\otimes_{\rmO} k=D_{\rmO}(G_k[\varpi]).
$$
Moreover, the dashed map is defined by the composition 
$$
D_{\rmO}(G_k)\xrightarrow{\mod \varpi} D_{\rmO}(G_k)\otimes_{\brvO} k\xrightarrow{\mod V} \Lie(G_k) 
$$
which is given by the reduction modulo $V$ since $\varpi D_{\rmO}(G_k)\subset VD_{\rmO}(G_k)$, which concludes the proof.
\end{proof}

Let us now describe the isomorphisms on (classical) points. As usual, we can restrict to $\CM_{\Dr}$, \ie suppose the height of the quasi-isogeny is $=0$. Let $x=(X,\rho)\in \CM_{\Dr}(\rmO_K)$, and set $X_k\coloneqq X\otimes_{\rmO_K}k$ its reduction mod $\varpi$. Let $M=D_{\rmO}(X_k)$ be the Dieudonné module of $X_k$ and let $\bar C\coloneqq \{ \bar i_0,\dots, \bar i_r\}\subset \BZ/d$ be the set of critical indices of $X_k$. For any $ i \in \bar C$ we set 
$$
\eta^i\coloneqq M_i^{\Pi=V}\subset M_i
$$
the finite free $\rmO$-module given by lemma \ref{lem:inva}. Consider $C\subset \llbracket 0,d-1\rrbracket$ the unique lift of $\bar C$ and order this set increasingly, \ie $C=\{i_1,\dots, i_r\}$. For any $i_j\in C$ let
$$
\eta_{i_j}\coloneqq r_{i_j}(\eta^{i_j})\subset E^d,
$$
where $r_{i_j}$ is defined in \eqref{eq:defri} and depending on $\rho$. This defines a simplex $\Delta=\{\eta_{i_0},\dots, \eta_{i_r}\}\in \brv{\BrTi}_{\! d}[r]$.

For the special fiber, the reduction mod $\fkm_K$ of $x$ defines a point $\bar x \in \rmM_k$, and the following proposition gives its image under $f$:
\medskip
\begin{prop}\label{prop:calsptspe}
Through the map $f\colon \lvert \rmM_k\rvert \rightarrow \lvert \BH_k^d\rvert $, the point $\bar x\in \rmM_k(k)$ is sent to
$$
f(\bar x )=(\eta_{i_j}\rightarrow \Lie(X_k)_{i_j})_j\in \lvert \BH_k^d\rvert,
$$
given for $i\in C$ by the composition $\eta_{i}\xrightarrow{r_{i}^{-1}}M_{i}\xrightarrow{\mod V} \Lie(X_k)_{i}$.
\end{prop}
\medskip
\begin{proof}
We need to trace back the isomorphisms $\rmM_k^{\eta}\cong X_{k}^{\eta}$ and $X_k^{\eta}\cong \BH_k^{\eta}$ which is done in Remark \ref{rema:tracer} for one critical index. Using Lemma \ref{lem:itglues} giving the gluing and Corollary \ref{cor:interslatt} we get the result for several critical indices.
\end{proof}

For the generic fibers, $x$ defines a classical point $\tilde x \in \rmM_{\brv E}$, and the following proposition gives its image under $\pi$:
\medskip
\begin{prop}\label{prop:calsptgen}
Through the period map $\pi \colon \lvert \rmM_{\brv E}\rvert \rightarrow \lvert \BH_{\brv E}^d\rvert$, the point $\tilde x\in \rmM_{\brv E}(K)$ is sent to 
$$
\pi(\tilde x)=(\eta_{i_j}\rightarrow \Lie(X)_{i_j}\otimes_{\rmO_K}K)_j\in \lvert \BH_{\brv E}^d\rvert,
$$
given for $i\in C$ by the composition $\eta_{i}\xrightarrow{r_{i}^{-1}}M_{i}\xrightarrow{h_{X,i}} \Lie(X)_{i}$.
\end{prop}
\medskip
\begin{proof}
By the definition of the period morphism in \S \ref{subsubsec:defpermap} and Remark \ref{rema:speconsym}, $\pi(\tilde x)$ is of the form 
$$
\pi(\tilde x)=(\lambda_{i_j}\rightarrow \Lie(X)_{i_j}\otimes_{\rmO_K}K),
$$
where $(\lambda_{i_j})_j$ is the chain of lattices determined by the $\rmO_K$-lattices
$$
\Lie(X)_{i_j}\subset \Lie(X)_{i_j}\otimes_{\rmO_K}K.
$$
By uniqueness of these lattices given by Lemma \ref{lem:compatlatnorm}, we deduce that $\lambda_{i}=\eta_{i}$ for $i\in C$.
\end{proof}

\medskip
\subsection{Lourenço--Scholze theorem}
In this section, we recall the main ingredient of the proof, namely that a flat normal formal scheme is determined by its generic fiber, perfectized special fiber and specialization morphism. Let $\CD$ be the category of triples $(X,Y,f)$ where:
\begin{itemize}
\itemb $X$ is a rigid space over $\brv E$,
\itemb $Y$ is a perfect scheme over $\Fpbar$,
\itemb $f\colon \lvert X \rvert \rightarrow \lvert Y \rvert$ is a continuous map between topological spaces, where $\lvert X \rvert$ denotes the space of classical points of $X$.
\end{itemize}
Also, let $\CC$ be the category of formal schemes over $\brvO$ that are flat, normal and topologically formally of finite type. Then, we have a natural functor $F\colon \CC\rightarrow \CD$ given, for $\fkX$ an object in $\CC$, by
$$
F(\fkX)\coloneqq (\fkX_{\brv E}, \fkX_{\Fpbar}^{\perf},\spe_{\fkX}),
$$
where 
\begin{itemize}
\itemb $\fkX_{\brv E}$ is the rigid space over $\brv E$ associated to $\fkX$,
\itemb $\fkX_{\Fpbar}^{\perf}$ is the perfection of the special fiber of $\fkX$,
\itemb $\spe_{\fkX}\colon \lvert \fkX_{\brv E}\rvert \rightarrow \lvert \fkX_{\Fpbar}\rvert$ is the specialization morphism.
\end{itemize}
The theorem is the following \cite[Corollary 4]{lou}, \cite[Theorem 18.4.2]{berk}:
\medskip
\begin{theo}\label{thm:lou}
The functor $F$ is fully faithful.
\end{theo}
\medskip
As a corollary, Theorem \ref{theo:main} follows from Theorem \ref{thm:genef}, Theorem \ref{thm:perfiso} and Theorem \ref{thm:specomp}. Note that the compatibility between the actions of $\GL_d(E)$ comes from the functoriality of $F$.

Let us explain the proof of \ref{thm:lou}. The proof in our case is fairly easy and the work done in \cite{lou} is mainly to extend this to pseudo-rigid spaces. Let $c=(X,Y,p)$ be an object of $\CD$ and define the ringed space 
$$
\fkX_c\coloneqq (\lvert Y\rvert, p_{*}\CO_{X}^{\circ}),
$$
where $\CO_X^{\circ}\subset \CO_X$ is the subsheaf of power-bounded elements of $\CO_X$. We need to check that if $c=F(\fkX)$ for $\fkX$ an object of $\CC$, then $\fkX\cong \fkX_c$ as a ringed spaces, which proves that $c\mapsto \fkX_c$ is a quasi-inverse of $F$, defined on the essential image of $F$. Locally, this reduces to the following lemma:
\medskip
\begin{lemm}
Let $A^+$ be a topologically finitely generated algebra over $\brvO$ which is normal and flat over $\brvO$. Let $A\coloneqq A^+[1/p]$, then the subring of power bounded elements of $A$ is $A^+$, \ie $A^{\circ}=A^+$.
\end{lemm}
\medskip
\begin{proof}
The argument is classical, we quickly recall it. First note that flatness ensures that $A^+\subset A$ is a subring of definition (meaning that it is an open and bounded subring). Moreover, $A^+\subset A^{\circ}$, since any element of $A^+$ is power bounded. Now since $A^+$ is normal over $\brvO$ this implies that $A^+$ is integrally closed in $A$. Hence we just need to show that any element of $A^{\circ}$ is integral over $A^+$.

For this, let $x\in A^{\circ}$. By definition, there exists an integer $N\geqslant 0$ such that for any $m\in \BN$, $x^m\in p^{-N}A^+$. Hence $A^+[x]\subset p^{-N}A^+$. But because $A^+$ is noetherian, $A[x]$ is a finite $A^+$-module since $p^{-N}A^+$ is a finite $A^+$-module, which means that $x$ is integral, thus $x\in A^+$.
\end{proof}
\subsection{Weil descent}\label{subsec:weildes}
Let us explain how the \emph{Weil descent datum} on $\brv{\CM}_{\Dr}$ is transported through the isomorphism and why the $\GL_d(E)$ invariance suffices to get the compatibility with Weil descent datum. Through the isomorphism of theorem \ref{theo:main}, translation by $1$ on $\BZ$ in \eqref{eq:concompz} corresponds with 
$$
(X,\rho)\mapsto (X^{\Pi},\Pi^{-1}\rho)
$$
on $\brv{\CM}_{\Dr}$, where $X^{\Pi}$ is the SFD associated to $X$ obtained by conjugating the action of $\rmOD$ by $\Pi$. In particular, $X^{\Pi^d}\cong X$ and the Weil descent datum is given by $(X,\rho)\mapsto (X,\varpi^{-1}\rho)$. This means that if we consider $\varpi\in \GL_d(E)$ as the scalar matrix, then the quotient of the isomorphism in theorem \ref{theo:main} by the discrete subgroup $\varpi^{\BZ}\subset \GL_d(E)$:
$$
\brv{\CM}_{\Dr}/\varpi^{\BZ}\cong \brv{\BH}_{\brvO}^{d}/\varpi^{\BZ}=(\BH_{\rmO}^d\times \BZ/d)\otimes_{\rmO}\brvO,
$$
can be descended to a $\GL_d(E)$-equivariant isomorphism between $\BH_{\rmO}^d\times \BZ/d$ and the model over $\rmO$ of $\brv{\CM}_{\Dr}/\varpi^{\BZ}$ defined in \cite[Theorem 3.49]{razi}. Thus, the compatibility with the Weil descent datum is essentially a consequence of the $\GL_d(E)$-equivariance.

\backmatter


\begin{thebibliography}{10}

\bibitem{abbr}
P.~Abramenko and K.~S.~Brown.
\newblock {\em Buildings: Theory and applications}.
\newblock Grad. Texts Math. 248, Springer (2008).

\bibitem{acz}
T.~Ahsendorf, C.~Cheng, and T.~Zink.
\newblock $\mathcal O$-displays and $\pi$-divisible formal $\mathcal O$-modules.
\newblock {\em J. Algebra}, 457:129--193 (2016).

\bibitem{bar}
S.~Bartling.
\newblock The universal special formal $\mathcal{O}_D$-module for $d=2$.
\newblock Preprint, \href{https://arxiv.org/abs/2206.13195}{arXiv:2206.13195} (2022).

\bibitem{baho}
S.~Bartling and M.~Hoff.
\newblock Moduli spaces of nilpotent displays.
\newblock {\em Int. Math. Res. Not.}, 2025(3):rnaf005 (2025).

\bibitem{bas}
H.~Bass.
\newblock {\em Algebraic $K$-theory}.
\newblock Math. Lect. Note Ser., Benjamin/Cummings, Reading, MA (1968).

\bibitem{ber}
V.~G.~Berkovich.
\newblock {\em Spectral theory and analytic geometry over non-Archimedean fields}.
\newblock Math. Surveys Monogr. 33, Amer. Math. Soc., Providence, RI (1990).

\bibitem{bbm}
P.~Berthelot, L.~Breen, and W.~Messing.
\newblock {\em Th\'eorie de {Dieudonn\'e} cristalline. II}.
\newblock Lecture Notes in Math. 930, Springer (1982).

\bibitem{beda}
M.~Bertolini and H.~Darmon.
\newblock Heegner points, $p$-adic $L$-functions, and the {Cherednik--Drinfeld} uniformization.
\newblock {\em Invent. Math.}, 131(3):453--491 (1998).

\bibitem{bhsc}
B.~Bhatt and P.~Scholze.
\newblock Projectivity of the {Witt} vector affine {Grassmannian}.
\newblock {\em Invent. Math.}, 209(2):329--423 (2017).

\bibitem{bgr}
S.~Bosch, U.~G\"untzer, and R.~Remmert.
\newblock {\em Non-Archimedean analysis: A systematic approach to rigid analytic geometry}.
\newblock Grundlehren Math. Wiss. 261, Springer (1984).

\bibitem{boer}
S.~Boucksom and D.~Eriksson.
\newblock Spaces of norms, determinant of cohomology and {Fekete} points in non-{Archimedean} geometry.
\newblock {\em Adv. Math.}, 378:107501 (2021).

\bibitem{boy}
P.~Boyer.
\newblock Bad reduction of {Drinfeld} varieties and local {Langlands} correspondence.
\newblock {\em Invent. Math.}, 138(3):573--629 (1999).

\bibitem{balg}
N.~Bourbaki.
\newblock {\em \'El\'ements de math\'ematique. Alg\`ebre. Chapitres 4 \`a 7}.
\newblock Springer, Berlin (2007).

\bibitem{balgcom}
N.~Bourbaki.
\newblock {\em \'El\'ements de math\'ematique. Alg\`ebre commutative. Chapitres 1 \`a 4}.
\newblock Springer, Berlin (2006).

\bibitem{boca}
J.-F.~Boutot and H.~Carayol.
\newblock $p$-adic uniformization of {Shimura} curves: the theorems of {Cherednik} and {Drinfeld}.
\newblock {\em Ast\'erisque}, 196--197:45--158 (1991).

\bibitem{bssw}
T.~Barthel, T.~M.~Schlank, N.~Stapleton, and J.~Weinstein.
\newblock On the rationalization of the $K(n)$-local sphere.
\newblock Preprint, \href{https://arxiv.org/abs/2402.00960}{arXiv:2402.00960} (2025).

\bibitem{car2}
H.~Carayol.
\newblock Sur les repr\'esentations galoisiennes modulo $\ell$ attach\'ees aux formes modulaires.
\newblock {\em Duke Math. J.}, 59(3):785--801 (1989).

\bibitem{car1}
H.~Carayol.
\newblock Non-abelian {Lubin--Tate} theory.
\newblock In {\em Automorphic Forms, Shimura Varieties, and $L$-Functions, Vol.~II}, pages 15--39 (1990).

\bibitem{cfs}
M.~Chen, L.~Fargues, and X.~Shen.
\newblock On the structure of some $p$-adic period domains.
\newblock {\em Camb. J. Math.}, 9(1):213--267 (2021).

\bibitem{che}
I.~V.~Cherednik.
\newblock Algebraische {Kurven}, die durch diskrete arithmetische {Untergruppen} von $PGL_2(k_w)$ uniformisierbar sind.
\newblock {\em Usp. Mat. Nauk}, 30(3):181--182 (1975).

\bibitem{col1}
P.~Colmez.
\newblock Espaces de {Banach} de dimension finie.
\newblock {\em J. Inst. Math. Jussieu}, 1(3):331--439 (2002).

\bibitem{col2}
P.~Colmez.
\newblock Espaces vectoriels de dimension finie et repr\'esentations de de~{Rham}.
\newblock In {\em Repr\'esentations $p$-adiques de groupes $p$-adiques I}, Ast\'erisque 319, pages 117--186 (2008).

\bibitem{codoniste}
P.~Colmez, G.~Dospinescu, and W.~Nizio{\l}.
\newblock Cohomology of $p$-adic {Stein} spaces.
\newblock {\em Invent. Math.}, 219(3):873--985 (2020).

\bibitem{codoni}
P.~Colmez, G.~Dospinescu, and W.~Nizio{\l}.
\newblock On the geometrization of the $p$-adic local {Langlands} correspondence.
\newblock In {\em Proceedings of the International Conference of Basic Science 2024}, pages 92--102.
\newblock International Press (2025).

\bibitem{dat}
J.-F.~Dat.
\newblock Nonabelian {Lubin--Tate} theory and elliptic representations.
\newblock {\em Invent. Math.}, 169(1):75--152 (2007).

\bibitem{dej}
A.~J.~de~Jong.
\newblock Crystalline {Dieudonn\'e} module theory via formal and rigid geometry.
\newblock {\em Publ. Math. IH\'ES}, 82:5--96 (1995).

\bibitem{delu}
P.~Deligne and G.~Lusztig.
\newblock Representations of reductive groups over finite fields.
\newblock {\em Ann. of Math. (2)}, 103:103--161 (1976).

\bibitem{doca}
G.~Dospinescu and J.~E.~Rodr\'iguez Camargo.
\newblock A {Jacquet--Langlands} functor for $p$-adic locally analytic representations.
\newblock Preprint, \href{https://arxiv.org/abs/2411.17082}{arXiv:2411.17082} (2025).

\bibitem{drin}
V.~G.~Drinfeld.
\newblock Coverings of $p$-adic symmetric regions.
\newblock {\em Funct. Anal. Appl.}, 10:107--115 (1976).

\bibitem{fal}
G.~Faltings.
\newblock A relation between two moduli spaces studied by {Drinfeld}.
\newblock In {\em Algebraic Number Theory and Algebraic Geometry}, pages 115--129.
\newblock Amer. Math. Soc. (2002).

\bibitem{rio}
L.~Fargues.
\newblock The curve.
\newblock In {\em Proc. ICM 2018, Vol.~II}, pages 291--319 (2018).

\bibitem{fafo}
L.~Fargues and J.-M.~Fontaine.
\newblock {\em Courbes et fibr\'es vectoriels en th\'eorie de {Hodge} $p$-adique}.
\newblock Ast\'erisque 406 (2018).

\bibitem{far}
L.~Fargues, A.~Genestier, and V.~Lafforgue.
\newblock {\em L'isomorphisme entre les tours de {Lubin--Tate} et de {Drinfeld}}.
\newblock Progr. Math. 262, Birkh\"auser, Basel (2008).

\bibitem{fasc}
L.~Fargues and P.~Scholze.
\newblock Geometrization of the local {Langlands} correspondence.
\newblock Preprint, \href{https://arxiv.org/abs/2102.13459}{arXiv:2102.13459} (2021).

\bibitem{fon}
J.-M.~Fontaine.
\newblock {\em Groupes $p$-divisibles sur les corps locaux}.
\newblock Ast\'erisque 47--48, Soc. Math. France (1977).

\bibitem{fuka}
K.~Fujiwara and F.~Kato.
\newblock {\em Foundations of rigid geometry. I}.
\newblock EMS Monogr. Math. (2018).

\bibitem{gen}
A.~Genestier.
\newblock {\em Espaces sym\'etriques de {Drinfeld}}.
\newblock Ast\'erisque 234 (1996).

\bibitem{ega2}
A.~Grothendieck.
\newblock \'El\'ements de g\'eom\'etrie alg\'ebrique. II: \'Etude globale \'el\'ementaire de quelques classes de morphismes.
\newblock {\em Publ. Math. IH\'ES}, 4:1--222 (1961).

\bibitem{ega42}
A.~Grothendieck.
\newblock \'El\'ements de g\'eom\'etrie alg\'ebrique. IV: \'Etude locale des sch\'emas et des morphismes de sch\'emas (deuxi\`eme partie).
\newblock {\em Publ. Math. IH\'ES}, 24:1--231 (1965).

\bibitem{ega43}
A.~Grothendieck.
\newblock \'El\'ements de g\'eom\'etrie alg\'ebrique. IV: \'Etude locale des sch\'emas et des morphismes de sch\'emas (troisi\`eme partie).
\newblock {\em Publ. Math. IH\'ES}, 28:5--255 (1966).

\bibitem{ega44}
A.~Grothendieck.
\newblock \'El\'ements de g\'eom\'etrie alg\'ebrique. IV: \'Etude locale des sch\'emas et des morphismes de sch\'emas (quatri\`eme partie).
\newblock {\em Publ. Math. IH\'ES}, 32:5--361 (1967).

\bibitem{ega1}
A.~Grothendieck and J.~A.~Dieudonn\'e.
\newblock {\em \'El\'ements de g\'eom\'etrie alg\'ebrique. I}.
\newblock Grundlehren Math. Wiss. 166, Springer (1971).

\bibitem{ful}
W.~Fulton.
\newblock {\em Intersection theory}.
\newblock Ergeb. Math. Grenzgeb. (3) 2, Springer (1984).

\bibitem{har}
M.~Harris.
\newblock Supercuspidal representations in the cohomology of {Drinfeld} upper half spaces.
\newblock {\em Invent. Math.}, 129(1):75--119 (1997).

\bibitem{haz}
M.~Hazewinkel.
\newblock {\em Formal groups and applications}.
\newblock Pure Appl. Math. 78, Academic Press, New York (1978).

\bibitem{hed}
S.~M.~H.~Hedayatzadeh.
\newblock Explicit isomorphism between {Cartier} and {Dieudonn\'e} modules.
\newblock {\em J. Algebra}, 570:611--635 (2021).

\bibitem{hub}
R.~Huber.
\newblock {\em \'Etale cohomology of rigid analytic varieties and adic spaces}.
\newblock Aspects Math. E30, Vieweg, Wiesbaden (1996).

\bibitem{kott}
R.~E.~Kottwitz.
\newblock Isocrystals with additional structure.
\newblock {\em Compos. Math.}, 56:201--220 (1985).

\bibitem{kuraya}
S.~S.~Kudla, M.~Rapoport, and T.~Yang.
\newblock {\em Modular forms and special cycles on {Shimura} curves}.
\newblock Ann. Math. Stud. 161, Princeton Univ. Press, Princeton, NJ (2006).

\bibitem{laz}
M.~Lazard.
\newblock {\em Commutative formal groups}.
\newblock Lecture Notes in Math. 443, Springer (1975).

\bibitem{leb}
A.-C.~Le~Bras.
\newblock Espaces de {Banach--Colmez} et faisceaux coh\'erents sur la courbe de {Fargues--Fontaine}.
\newblock {\em Duke Math. J.}, 167(18):3455--3532 (2018).

\bibitem{lou}
J.~N.~P.~Louren\c{c}o.
\newblock The {Riemannian} {Hebbarkeitss\"atze} for pseudorigid spaces.
\newblock Preprint, \href{https://arxiv.org/abs/1711.06903}{arXiv:1711.06903} (2017).

\bibitem{mes}
W.~Messing.
\newblock {\em The crystals associated to {Barsotti--Tate} groups}.
\newblock Lecture Notes in Math. 264, Springer (1972).

\bibitem{mus}
G.~A.~Mustafin.
\newblock Nonarchimedean uniformization.
\newblock {\em Math. USSR Sb.}, 34:187--214 (1978).

\bibitem{ravi}
M.~Rapoport and E.~Viehmann.
\newblock Towards a theory of local {Shimura} varieties.
\newblock {\em M\"unster J. Math.}, 7(1):273--326 (2014).

\bibitem{razi}
M.~Rapoport and T.~Zink.
\newblock {\em Period Spaces for $p$-divisible Groups}.
\newblock Ann. Math. Stud. 141, Princeton Univ. Press (1996).

\bibitem{razion}
M.~Rapoport and T.~Zink.
\newblock On the {Drinfeld} moduli problem of $p$-divisible groups.
\newblock {\em Camb. J. Math.}, 5(2):229--279 (2017).

\bibitem{schsur}
P.~Scholze.
\newblock Perfectoid spaces: a survey.
\newblock In {\em Current Developments in Mathematics 2012}, pages 193--227.
\newblock International Press (2013).

\bibitem{psch}
P.~Scholze.
\newblock On the $p$-adic cohomology of the {Lubin--Tate} tower.
\newblock {\em Ann. Sci. \'Ec. Norm. Sup\'er. (4)}, 51(4):811--863 (2018).

\bibitem{scwe}
P.~Scholze and J.~Weinstein.
\newblock Moduli of $p$-divisible groups.
\newblock {\em Camb. J. Math.}, 1(2):145--237 (2013).

\bibitem{berk}
P.~Scholze and J.~Weinstein.
\newblock {\em Berkeley lectures on $p$-adic geometry}.
\newblock Ann. Math. Stud. 207, Princeton Univ. Press, Princeton, NJ (2020).

\bibitem{stacks}
The {Stacks Project Authors}.
\newblock {\em Stacks Project}.
\newblock \url{https://stacks.math.columbia.edu}.

\bibitem{zinc}
T.~Zink.
\newblock {\em Cartier-Theorie kommutativer formaler Gruppen}.
\newblock Teubner-Texte zur Mathematik 68, Teubner, Leipzig (1984).

\bibitem{zin}
T.~Zink.
\newblock The display of a formal $p$-divisible group.
\newblock In {\em Cohomologies $p$-adiques et applications arithm\'etiques (I)}, pages 127--248.
\newblock Soc. Math. France (2002).

\end{thebibliography}
\end{document}